\documentclass[12pt,a4paper]{article}
\usepackage{amsmath}
\usepackage{amssymb}
\usepackage{cases}
\usepackage{bbm}
\usepackage{mathrsfs}
\usepackage{graphicx}
\usepackage{amsfonts}

\usepackage{verbatim}
\usepackage{enumerate}
\usepackage{theorem}
\usepackage{color}
\usepackage[colorlinks, linkcolor=red, citecolor=blue]{hyperref}
\usepackage{cite}
\makeatletter
\def\tank#1{\protected@xdef\@thanks{\@thanks
		\protect\footnotetext[0]{#1}}}
\def\bigfoot{
	
	\@footnotetext}
\makeatother

\newcommand{\ea}{\end{array}}
\allowdisplaybreaks

\newtheorem{theorem}{Theorem}[section]
\newtheorem{lemma}{Lemma}[section]

\newtheorem{proposition}[theorem]{Proposition}

\newtheorem{remark}{Remark}[section]

\newtheorem{definition}[theorem]{Definition}

\newtheorem{Assumption}{Assumption}[section]

\def\<{{\langle}}
\def\>{{\rangle}}

{\theorembodyfont{\rmfamily}
}

\title{Stochastic Scalar Conservation Laws on Moving Hypersurfaces}

\author{{Ping Chen}$^1$\footnote{E-mail:2025800007@hfut.edu.cn}~~~ {Tusheng Zhang}$^{2,3}$\footnote{E-mail:tusheng.zhang@manchester.ac.uk}
\\
\small  1. School of Mathematics,
\small  Hefei University of Technology,\\
\small  Hefei, Anhui 230009, China.\\
\small  2. School of Mathematics,
\small  University of Science and Technology of China,\\
\small  Hefei, Anhui 230026, China.\\
\small  3. Department of Mathematics, University of Manchester,\\
\small  Oxford Road, Manchester, M13 9PL, UK.
}
\date{}
\newenvironment{proof}{\par\noindent{\bf Proof:}}{\hspace*{\fill}$\blacksquare$\par}

\begin{document}
\maketitle
\noindent \textbf{Abstract:}
We establish the well-posedness of stochastic scalar conservation laws on moving hypersurfaces driven by Brownian motion. To handle the interaction between stochastic forcing and evolving geometry, we derive an It\^{o} formula on moving surfaces and introduce the notion of generalized entropy solutions incorporating the relevant stochastic interaction terms. A martingale entropy solution is constructed via the vanishing-viscosity method, based on a uniform $L^\infty$-bound in space and time, an $L^1$-estimate for the spatial gradient, an $L^1$-continuity estimate in time, and a suitable tightness argument. Pathwise uniqueness is established by adapting Kruzhkov's doubling-of-variables method to moving hypersurfaces, yielding an $L^1$-contraction property. Finally, together with the Yamada-Watanabe theorem, these results yield the well-posedness of the problem.



\vspace{4mm}


\vspace{3mm}
\numberwithin{equation}{section}
\hypersetup{linkcolor=black}
\tableofcontents
\hypersetup{linkcolor=red}
\section{Introduction}

Partial differential equations (PDEs) on time-dependent domains or evolving surfaces have attracted considerable attention, as they provide realistic mathematical frameworks for describing a wide range of physical and biological phenomena; see, for example, \cite{EAKDS,ESV}. Conservation laws constitute a fundamental class of nonlinear PDEs. Derived from basic balance principles, they describe the transport and redistribution of conserved quantities such as mass, momentum, energy, and chemical concentration. In addition to their broad applications in continuum mechanics, fluid dynamics, traffic flow, and biological systems, conservation laws are of considerable mathematical interest because nonlinear effects may generate discontinuities even when the initial values are smooth. Consequently, classical solutions generally cease to exist globally in time, and suitable notions of weak and entropy solutions are required to ensure well-posedness. Within this broader context, conservation laws on moving surfaces provide natural frameworks for modelling the transport and redistribution of quantities over evolving interfaces, particularly when surface-bound processes interact strongly with changes in the underlying geometry. More generally, such equations arise, for instance, in models of cell migration, where reaction-diffusion processes on an evolving cell membrane can generate localized activation signals for pseudopod formation and, when coupled with a mechanical model and a level set method, enables the simulation of both random cell movement and chemotaxis.

In this paper, we are concerned with the well-posedness of the stochastic scalar conservation laws on moving hypersurfaces $ \mathop{\scalebox{1.5}[1.5]{$\cup$}}_{t\in[0,T]} \{t\}\times\Gamma_t$, which is written as follows
\begin{align}
\left\{
\begin{aligned}\label{eq74}
&\overset{\scalebox{0.4}{$\bullet$}}{u}(t,x)+u(t,x)\nabla_{\Gamma_t}\cdot v(t,x) + \nabla_{\Gamma_t} \cdot \big( f((t,x), u(t,x))\big) = \sigma(t,x)\overset{\scalebox{0.4}{$\bullet$}}{B_t}, \quad (t,x)\in Q_T,\\
&u(0,y)=u_{0}(y), \quad y\in\Gamma_0.
\end{aligned}
\right.
\end{align}
The space--time surface $Q_T \subseteq [0,T]\times \mathbb{R}^{n+1}$ is defined as
\begin{equation}\label{eq197}
Q_T=\ \mathop{\scalebox{1.5}[1.5]{$\cup$}}_{t\in[0,T]} \{t\}\times\Gamma_t.
\end{equation}
In this equation, $u: Q_T \to \mathbb{R}$ is a random function, and the dot stands for a material derivative. The vector field $v: Q_T \to \mathbb{R}^{n+1}$ represents the velocity of the evolving surface $\Gamma_t$, and $\nabla_{\Gamma_t}$ denotes the surface gradient on $\Gamma_t$. $f: Q_T \times \mathbb{R} \to \mathbb{R}^{n+1}$ is the given flux function. Moreover, $\{ B_t \}_{t \in [0,T]}$ is a one-dimensional standard Brownian motion defined
on a filtered probability space $(\Omega,\mathcal{F}, \{\mathcal{F}_{t}\}_{t\in [0,T]},\mathbb{P})$, where $\{\mathcal{F}_{t}\}_{t\in[0,T]}$ satisfies the usual conditions.
The function $\sigma(t,x): Q_T  \to \mathbb{R}$ is a given coefficient independent of the solution $u$. The derivation of the equation (\ref{eq74}) can be found in \cite{GDT}.

Parallel to the deterministic theory, substantial effort has been devoted to the study of conservation laws subject to stochastic forcing on fixed Euclidean domains or surfaces. 
The resulting theory of stochastic conservation laws has developed into an active research area at the intersection of nonlinear PDEs, stochastic analysis, and probability theory. A fundamental challenge is to formulate an appropriate concept of entropy solutions that are compatible with both the nonlinear hyperbolic structure of the equation and the stochastic forcing.

For scalar conservation laws with additive noise, Kim \cite{K} extended Kruzhkov's entropy formulation and used vanishing viscosity together with compensated compactness arguments to establish the existence of stochastic weak entropy solutions. Uniqueness was obtained by means of a stochastic version of Kruzhkov's doubling-of-variables argument. Vallet and Wittbold \cite{VW} subsequently extended the additive-noise theory to multi-dimensional initial-boundary value problems with Dirichlet boundary conditions.
These works already reveal an important distinction between deterministic and stochastic entropy theories: stochastic forcing introduces additional interactions that must be properly reflected in the entropy framework, particularly in uniqueness arguments.

For conservation laws with multiplicative noise
on the whole space, Feng and Nualart \cite{FN} introduced the notion of a stochastic strong entropy solution. The additional strong entropy condition allowed them to establish uniqueness, whereas existence was proved, through vanishing viscosity and compensated compactness arguments in one spatial dimension. Chen et al. \cite{CDK} later considered multi-dimensional problems and obtained well-posedness under a uniform spatial BV estimate. Using a kinetic formulation, Debussche and Vovelle \cite{DV} established existence and uniqueness of entropy solutions on the $d$-dimensional torus. Their results were further extended by Galimberti and Karlsen \cite{GK} to stochastic conservation laws on Riemannian manifolds. Further developments  of stochastic conservation laws can be found in \cite{LW,H} and references therein.

The preceding results provide a substantial theory of stochastic conservation laws when the underlying domain or surfaces is fixed. By contrast, the theory on evolving geometries is much less developed. For moving domains, Neves \cite{N} studied multi-dimensional scalar conservation laws in noncylindrical Lipschitz domains and established existence and uniqueness for $L^{\infty}$ initial-boundary data. The existence proof was based on vanishing viscosity and Young measure techniques, whereas uniqueness was obtained through Kruzhkov's doubling-of-variables method.  For moving surfaces, the principal deterministic well-posedness result relevant to the present work is due to Dziuk et al. \cite{GDT}. They established existence by combining vanishing viscosity with uniform estimates in the $H^{1,1}$-norm and a compactness argument, and proved uniqueness by adapting the Kruzhkov's doubling-of-variables technique to moving surfaces.

The purpose of the present paper is to establish the well-posedness of stochastic conservation laws on moving surfaces extend the work in \cite{GDT} to the stochastic setting. Such an extension is not a direct consequence of either the deterministic theory on moving surfaces or the stochastic theory on fixed domains. Indeed, the moving of the hypersurface generates additional geometric transport and deformation terms, while the stochastic forcing introduces interaction terms that must be retained when comparing two solutions. These effects enter simultaneously into the entropy formulation, the approximation procedure, and the limiting arguments required in the existence and uniqueness proofs.

More specifically, several difficulties must be overcome. First, the entropy inequalities must retain the stochastic interaction terms required to derive the $L^1$-contraction property. A formulation obtained by directly transferring the deterministic entropy inequalities to the stochastic setting would not contain sufficient information about the interaction of the noise terms. Second, the notion of entropy solutions must be compatible with the vanishing-viscosity approximation. In particular, we seek to avoid imposing an additional strong entropy condition of the type introduced in the literature.
A further difficulty lies in deriving the compactness estimates for the viscous approximations, due to the combined effects of the stochastic forcing and the time-dependent geometry of the moving hypersurface. 

To overcome these difficulties, we first establish an It\^{o} formula adapted to moving hypersurfaces, and then introduce a class of generalized entropy solutions whose entropy inequalities are designed to retain the stochastic interaction terms required for the comparison principle. Within this framework, existence is proved by adapting the vanishing-viscosity method used in \cite{GDT} to the stochastic setting. More precisely, we derive a uniform $L^{\infty}$-estimate in space and time, an $L^{1}$-estimate for the spatial gradient, and an $L^{1}$-continuity estimate in time for the viscous approximations. The first two are obtained by introducing a transformed variable that removes the stochastic forcing, while the time-continuity estimate is established by pulling the equation back to a fixed reference hypersurface. Together with a suitable tightness argument, these estimates yield the existence of a martingale entropy solution. To prove the pathwise uniqueness, we adapt Kruzhkov's doubling-of-variables method to moving hypersurfaces and establish the corresponding $L^1$-contraction property. Finally, combining the existence result and pathwise uniqueness with the Yamada-Watanabe theorem, we establish the well-posedness of stochastic scalar conservation laws with additive noise on moving hypersurfaces.

The remainder of the paper is organized as follows. Section 2 provides the necessary preliminaries on moving hypersurfaces, including the corresponding It\^{o} formula. In Section 3, we introduce the notion of generalized entropy solutions and state the main results. Section 4 is devoted to the stochastic viscous problem underlying the approximation procedure, while Section 5 establishes the uniform estimates required for the subsequent tightness argument. The existence of a martingale entropy solution is proved in Section 6, and pathwise uniqueness is established in Section 7. Section 8 is the appendix containing some of the long technical proofs.

\vskip 0.3cm
Here are some conventions used throughout this paper:


(i) We let $C$ represent a {generic} positive constant whose value may change from line to line. The dependence of constants on parameters, if necessary, will be indicated, e.g. $C_\epsilon$.

(ii) For any metric space $E_1$ and $E_2$, let $C^k(E_1,E_2)$ be the space of k-times continuously differentiable functions from $E_1$ to $E_2$. When $E_2= \mathbb{R}$, we simply write $C^k(E_1)$. Moreover, let $C_c^k(E_1)$ denote the subspace of $C^k(E_1)$ consisting of compactly supported functions, and denote the support of a function by $supp(\cdot)$. For $f \in C^2(\mathbb{R})$, $f^\prime$ and $f^{\prime\prime}$ stand for the first and second derivatives of $f$, respectively. For $f \in C^2(\mathbb{R}^m)$, $\nabla f$ stands for the gradient of $f$.

\section{Preliminaries}
We begin with some preliminaries on the hypersurfaces. We follow the description outlined in \cite{GCM}. Let $n \in \mathbb{N}$. $\Gamma\subseteq \mathbb{R}^{n+1}$ is an $n$-dimensional compact $C^k$ hypersurface if there exists a finite collection of local parametrizations $\{ X^\alpha\}_{\alpha \leq N}$ such that for each $\alpha \leq N$, $X^\alpha \in C^k(U^\alpha, \mathbb{R}^{n+1})$, where $U^\alpha$ are connected open sets in $\mathbb{R}^n$, satisfying the following conditions:
\begin{itemize}
\item[(i)] $X^\alpha$ is injective and rank$(\nabla X^{\alpha}(\theta))=n$ for any $\theta \in U^\alpha$ and $\alpha \leq N$.
\item[(ii)] For each $\alpha$, there exists an open set $\tilde{V}^\alpha$ in $\mathbb{R}^{n+1}$ such that $X^{\alpha}(U^\alpha)=\tilde{V}^\alpha \cap \Gamma$.
\item[(iii)] $\{V^\alpha:=X^{\alpha}(U^\alpha)\}_{\alpha \leq N}$ is an open cover of $\Gamma$.
\end{itemize}

A function $\xi: \Gamma \rightarrow \mathbb{R}$ is $k$-times continuously differentiable if all the functions $\xi \circ X^\alpha: U^\alpha \rightarrow \mathbb{R}$, $\alpha \leq N$, are $k$-times continuously differentiable. We denote by $C^k(\Gamma)$ the set of functions $\xi: \Gamma \rightarrow \mathbb{R}$ that is $k$-times continuously differentiable.

For every $\alpha \leq N$, we define the first fundamental form $\{g^\alpha_{ij}(\theta)\}_{i,j=1,\cdots,n}$, $\theta \in U^\alpha$ by
\[ g^\alpha_{ij}(\theta) = \frac{\partial X^\alpha}{\partial \theta_i}(\theta) \cdot \frac{\partial X^\alpha}{\partial \theta_j}(\theta), \ i,j=1, \cdots, n. \]
We also denote the inverse of $\{ g^\alpha_{ij}(\theta)\}_{i,j=1,\cdots, n}$ as $\{ g^{ij}_{\alpha}(\theta) \}_{i,j=1, \cdots, n}$ and $\mathrm{det}(g^\alpha_{ij}(\theta))$ as $g^{\alpha}(\theta)$ for $\theta \in U^\alpha$.

Using local parametrizations, one can express the Laplace-Beltrami operator on $\Gamma$ locally. For $\xi \in C^2(\Gamma)$, we have
\[ (\Delta_{\Gamma}\xi)(X^{\alpha}(\theta)) = \frac{1}{\sqrt{g^\alpha(\theta)}}\frac{\partial}{\partial \theta_j}\big( g^{ij}_{\alpha}(\theta)\sqrt{g^\alpha(\theta)}\frac{\partial \xi\circ X^\alpha}{\partial \theta_i}(\theta) \big), \ \theta \in U^\alpha, \]
here and below, we use the Einstein summation convention (i.e., repeated indices imply summation).
The tangential gradient of $\xi: \Gamma \rightarrow \mathbb{R}$ is locally given by
\[ (\nabla_{\Gamma}\xi)(X^\alpha(\theta)) = g^{ij}_{\alpha}(\theta)\frac{\partial \xi\circ X^\alpha}{\partial \theta_j}(\theta)\frac{\partial X^\alpha}{\partial \theta_i}(\theta), \ \theta \in U^\alpha. \]
Sometimes, we will use the notation
\[ \nabla_{\Gamma}\xi=(\underline{D}_{1}\xi, \cdots, \underline{D}_{n+1}\xi)  \]
for the $n+1$ components of the tangential gradient, i.e. $\underline{D}_{k}\xi =  g^{ij}_{\alpha}(\theta)\frac{\partial \xi\circ X^\alpha}{\partial \theta_j}(\theta)\frac{\partial X^\alpha_k}{\partial \theta_i}(\theta)$. Then the divergence of $g:\Gamma \rightarrow \mathbb{R}^{n+1}$ is given by
\[ \nabla_{\Gamma} \cdot g = \sum_{i=1}^{n+1}\underline{D}_{i}g_i.  \]
Moreover, it holds that
\[ \Delta_{\Gamma}\xi = \nabla_{\Gamma}\cdot \nabla_{\Gamma}\xi = \sum_{i=1}^{n+1}\underline{D}_{i}\underline{D}_{i}\xi. \]

Let $\nu$ denote the unit normal vector on $\Gamma$, and define the matrix
\[ \mathcal{H} = \nabla_{\Gamma}\nu, \]
whose components are given by
\[\mathcal{H}_{ij} = (\nabla_{\Gamma}\nu)_{ij} = \underline{D}_{i}\nu_{j} = \underline{D}_{j}\nu_i, \quad i,j=1,\cdots, n+1.  \]
The mean curvature $H$ of $\Gamma$ is defined as the trace of $\mathcal{H}$, namely,
\[ H= \sum_{j=1}^{n+1} \mathcal{H}_{jj}. \]

The following result concerns the commutation of tangential derivatives, and the proof is provided in \cite{GDT}.
\begin{lemma}\label{lemma2.6}
For $\xi \in C^2(\Gamma)$, we have that
\[ \underline{D}_{i}\underline{D}_{k}\xi = \underline{D}_{k}\underline{D}_{i}\xi + \mathcal{H}_{kl}\underline{D}_{l}\xi\nu_i - \mathcal{H}_{il}\underline{D}_{l}\xi\nu_k, \quad i,k=1, \cdots, n+1.  \]
\end{lemma}

We now state the formula for integration by parts on a hypersurface $\Gamma$. The proof can be found in \cite{GCM}. 
\begin{lemma}\label{lemma2.3}
Assume that $\Gamma$ is a closed hypersurface and that $\xi\in C^1(\Gamma)$. Then we have
$$\int_\Gamma \nabla_\Gamma \xi \mathrm{dvol}_\Gamma=\int_{\Gamma}\xi H\nu \mathrm{dvol}_\Gamma,$$
where the volume measure $\mathrm{dvol}_{\Gamma}$ is associated with the first fundamental form $g_{ij}$ on $\Gamma$.
\end{lemma}

Below, we introduce some function spaces on a hypersurface $\Gamma$.

For $p \in [1, \infty]$, we denote by $L^p(\Gamma)$ the space of measurable functions $\xi: \Gamma \rightarrow \mathbb{R}$ with
\[ \| \xi \|_{L^p(\Gamma)} = \big(\int_{\Gamma} |\xi|^p(y) \mathrm{dvol}_{\Gamma}(dy) \big)^{\frac{1}{p}} < +\infty,  \]
and for $p = \infty$, we mean the essential supremum norm. In particular, for $p \in [1, \infty)$, using local parametrizations, we can write $\| \xi\|_{L^p(\Gamma)}^p$ as
\[ \| \xi \|^p_{L^p(\Gamma)} = \sum_{\alpha =1}^N \int_{U^\alpha} \chi^{\alpha}(X^\alpha(\theta))|\xi|^p(X^\alpha(\theta))\sqrt{g^\alpha(\theta)}d\theta,  \]
where $\{ \chi^\alpha\}_{\alpha \leq N}$ is a partition of unity with regard to the open cover $\{V^\alpha\}_{\alpha \leq N}$ on $\Gamma$. $L^p(\Gamma)$ is a Banach space for $p \in [1, \infty]$, and for $p=2$, it is a Hilbert space with inner product $ \langle \cdot,\cdot \rangle_{\Gamma}$. For $p \in [1, \infty)$, the Sobolev space $H^{1,p}(\Gamma)$ is the closure of $C^1(\Gamma)$ under the norm
\[ \| \xi\|_{H^{1,p}(\Gamma)} = \big( \| \xi \|_{L^p(\Gamma)}^p + \|\nabla_{\Gamma}\xi \|_{L^p(\Gamma)}^p \big)^{\frac{1}{p}}.  \]
For $p=2$, we use the notation $H^1(\Gamma) = H^{1,2}(\Gamma)$.
The space $H^{2}(\Gamma)$ is the closure of $C^2(\Gamma)$ under the norm
\[ \|\xi\|_{H^{2}(\Gamma)} = \big( \|\xi\|_{L^2(\Gamma)}^2 + \| \nabla_{\Gamma}\xi\|_{L^2(\Gamma)}^2
                                + \sum_{i,j =1}^{n+1}\|\underline{D}_{i}\underline{D}_{j}\xi \|_{L^2(\Gamma)}^2 \big)^{\frac{1}{2}}.  \]
Let $H^{-1}(\Gamma)$ be the dual space of $H^1(\Gamma)$, and denote by $ _{H^{-1}(\Gamma)}\langle \cdot , \cdot \rangle_{H^{1}(\Gamma)}$ the dual pair between $H^{-1}(\Gamma)$ and $H^1(\Gamma)$.

\subsection{The geometry of moving hypersurfaces}
In this subsection, we will provide a brief overview of the moving hypersurfaces as outlined in Section 5 of \cite{GCM}. Let $\{ \Gamma_t \}_{t \in [0,T]}$ be a family $C^3$ compact hypersurfaces without boundary, indexed by $t \in [0,T]$. The initial surface $\Gamma_0$ is transported by the function $G(\cdot, \cdot): [0,T] \times \Gamma_0 \rightarrow \mathbb{R}^{n+1}$, where $G \in C^2([0,T]; C^2(\Gamma_0))$, satisfying that $G(t, \cdot): \Gamma_0 \rightarrow \Gamma_t$ is a $C^2$ diffeomorphism between $\Gamma_0$ and $\Gamma_t$ for every $t \in [0,T]$ and $G(0, \cdot) = Id$. Furthermore, the velocity of the evolving of $\Gamma_t$ is denoted by
\[ v(t, G(t, \cdot)) = \frac{\partial G}{\partial t}(t,\cdot).  \]
Assume that $\{ X^\alpha \}_{\alpha \leq N}$ is a local parametrization of $\Gamma_0$. Then for every $t \in [0,T]$, by using the diffeomorphism $G(t, \cdot)$, $\Gamma_t$ has a local parametrization of given by $\{X^\alpha_t: U^\alpha \rightarrow G(t, V^\alpha) \}_{\alpha \leq N}$ where $V^\alpha = X^\alpha(U^\alpha)$, $X^\alpha_t(\theta) = G(t, X^\alpha(\theta))$ for $\theta \in U^\alpha$ and the related first fundamental form
\[ g^{t, \alpha}_{ij}(\theta) = \frac{\partial X^\alpha_t}{\partial \theta_i}(\theta) \cdot \frac{\partial X^\alpha_t}{\partial \theta_j}(\theta), \quad \theta \in U^\alpha, \quad i,j=1, \cdots, n, \quad \alpha \leq N.  \]
Moreover, we denote the inverse of $\{g^{t, \alpha}_{ij}(\theta)\}_{i,j=1,\cdots, n}$ as $\{ g^{ij}_{t,\alpha}(\theta)\}_{i,j=1,\cdots, n}$ and $\mathrm{det}(g^{t, \alpha}_{ij}(\theta))$ as $g^\alpha_t(\theta)$ for $\theta \in U^\alpha$. We further assume that each $U^\alpha$ is convex.

Consider the space-time surface given by
\[Q_T=\ \mathop{\scalebox{1.5}[1.5]{$\cup$}}_{t\in[0,T]} \{t\}\times\Gamma_t. \]
For a function $\xi: Q_T \rightarrow \mathbb{R}$, we define a pullback transformation $\xi \mapsto\widetilde{\xi}$ that satisfies
\begin{equation}\label{chongxing4}
 \widetilde{\xi}(t,y)=\xi(t,G(t,y)), (t,y) \in [0,T]\times \Gamma_0.
 \end{equation}
Furthermore, the local parametric representation of $\xi$ on the  domain $U^\alpha$ is defined as
\begin{equation}\label{chongxing5}
\widetilde{\xi}^{\alpha}_t(\theta)=\xi(t,X_t^{\alpha}(\theta)), (t,\theta) \in [0,T]\times U^\alpha, 1\leq \alpha \leq N.
\end{equation}
We denote by $C^k(Q_T)$ the set of functions $\xi: Q_T \rightarrow \mathbb{R}$ such that $\widetilde{\xi}$ is $k$-times continuously differentiable.
An appropriate time derivative, which is usually called ``material derivative", is defined on this moving surface $Q_T$, that is,
\[\overset{\scalebox{0.4}{$\bullet$}}{\xi} = \frac{\partial \xi}{\partial t} + v \cdot \nabla \xi = \Big( \frac{d}{dt}\xi(t, G(t, \cdot)) \Big)\circ G^{-1}, \quad  \xi \in C^1(Q_T). \]
The following formulae for the differentiation of time-dependent surface integrals are called transport formulae, and are proved in \cite{GCM}.
\begin{lemma}\label{lemma2.1}
Assume that $\xi$ is a function on $Q_T$ such that all the following quantities exist. Then
$$\frac{d}{dt}\int_{\Gamma_t}\xi\mathrm{dvol}_{\Gamma_t}=\int_{\Gamma_t}\overset{\scalebox{0.4}{$\bullet$}}{\xi}+\xi\nabla_{\Gamma_t}\cdot v \mathrm{dvol}_{\Gamma_t}.$$
\end{lemma}

Before the end of this subsection, we present the following formula for the commutation of tangential derivatives and material derivatives. A proof is given in the Appendix of \cite{GDT}.
\begin{lemma}\label{lemma2.5}
For $\xi \in C^2(Q_T)$, we have that
\[ (\underline{D}_{l}\xi)^{\scalebox{0.4}{$\bullet$}} = \underline{D}_{l}\overset{\scalebox{0.4}{$\bullet$}}{\xi} - A_{lr}(v)\underline{D}_{r}\xi  \]
with the matrix
\[ A_{lr}(v) = \underline{D}_{l}v_r - \nu_s\nu_l\underline{D}_{r}v_s, \quad l ,r=1,\cdots, n+1.   \]
\end{lemma}

\subsection{It\^{o} formula on moving hypersurfaces}

In this subsection, we prove an It\^{o} formula for stochastic processes on the moving hypersurface $Q_T$. It plays a crucial role in establishing the existence of solutions to equation \eqref{eq74}. Recall the pullback transformation and its associated local parametric representation
given by \eqref{chongxing4} and \eqref{chongxing5}, respectively. Let us begin by introducing necessary function spaces on moving hypersurfaces.
	\begin{align*}
		C([0,T]; L^2(\Gamma_{\cdot})) :=& \big \{  \{ u(t) \} _{t\in[0,T]}: u(t) \in L^2(\Gamma_{t}) \  \mbox{for any} \ t \in [0,T]  \\ & \  \mbox{such that} \ \widetilde{u} \in  C([0,T], L^2(\Gamma_{0})) \big \} ,
	\end{align*}
\vskip -1.0cm
	
	\begin{align*}
		L^2([0,T]; L^2(\Gamma_{\cdot})) :=& \big \{  \{ u(t) \} _{t\in[0,T]}: u(t) \in  L^2(\Gamma_{t}) \  \mbox{for any} \ t \in [0,T]  \\ & \  \mbox{such that} \ \widetilde{u} \in  L^2([0,T]; L^2(\Gamma_{0}))  \big \} ,
	\end{align*}
\vskip -1.0cm	

\begin{align*}
		 H^1(Q_T) :=& \big \{ \varphi: Q_T \rightarrow \mathbb{R} |\ \widetilde{\varphi} \in H^1([0,T] \times \Gamma_{0})   \big \},
	\end{align*}	
Analogously, we can define $C([0,T]; H^1(\Gamma_{\cdot}))$, $L^2([0,T]; H^1(\Gamma_{\cdot}))$ and $L^2([0,T]; H^2(\Gamma_{\cdot}))$.

Consider SPDEs on the moving hypersurface $Q_T$ of the following form: for $i \in \{1,\cdots, m\}$,
\begin{align}
\left\{
\begin{aligned}\label{eq61}
&\overset{\scalebox{0.4}{$\bullet$}}{z^i}(t)+z^i(t)\nabla_{\Gamma_t}\cdot v(t) = \underline{D}_{j}R^{ij}(t) + S^i(t) +\Upsilon^i(t)\overset{\scalebox{0.4}{$\bullet$}}{B_t}, \quad (t,x)\in Q_T,\\
&z^i(0,y)=z^i_0(y), \quad y\in\Gamma_0,
\end{aligned}
\right.
\end{align}
where for each $i \in \{1, \cdots, m \}$, $j \in \{1,\cdots,n+1\}$, the predictable processes $R^{ij}$, $S^i$ and $\Upsilon^i$ belong to $L^2(\Omega; L^2([0,T]; L^2(\Gamma_{\cdot})))$.
We now introduce the definition of a solution to equation \eqref{eq61}.
\begin{definition}\label{buchongdingyi}
We say that a stochastic process $\{ z(t)\}_{t\in[0,T]}$ is a solution to equation \eqref{eq61} if for each $i \in \{1,\cdots, m\}$,
\begin{itemize}
\item[(i)] $z^i \in L^2(\Omega; C([0,T]; L^2(\Gamma_{\cdot}))) \cap L^2(\Omega; L^2([0,T]; H^1(\Gamma_{\cdot})))$.
\item[(ii)] $z^i$ is an ${L}^{2}(\Gamma_{\cdot})$-valued predictable random process with respect to the filtration $ \{ \mathcal{F}_t \} _{t\in[0,T]} $.
\item[(iii)] For any $\varphi \in H^1(Q_T)$, we have $\mathbb{P}$-a.s.
\begin{eqnarray}\label{eq'76}
& & \int_{\Gamma_t} z^i(t,x)\varphi(t,x)\mathrm{dvol}_{\Gamma_t}(dx)- \int_{\Gamma_0} z^i_0(y)\varphi(0,y)\mathrm{dvol}_{\Gamma_0}(dy) \nonumber\\
& =& \int_0^t\int_{\Gamma_s}z^i(s,x)\overset{\scalebox{0.4}{$\bullet$}}{\varphi}(s,x) \mathrm{dvol}_{\Gamma_s}(dx)ds - \int_0^t\int_{\Gamma_s}R^{ij}(s,x)\underline{D}_{j}\varphi(s,x)\mathrm{dvol}_{\Gamma_s}(dx)ds\nonumber\\
& & + \int_0^t\int_{\Gamma_s}R^{ij}(s,x)\varphi(s,x)(H\nu_j)(s,x)\mathrm{dvol}_{\Gamma_s}(dx)ds +\int_0^t\int_{\Gamma_s} S^i(s,x)\varphi(s,x)\mathrm{dvol}_{\Gamma_s}(dx)ds \nonumber\\
& & +\int_0^t\int_{\Gamma_s} \Upsilon^i(s,x)\varphi(s,x)\mathrm{dvol}_{\Gamma_s}(dx)dB_s\nonumber\\
\end{eqnarray}
for every $t \in [0,T]$.
\end{itemize}
\end{definition}

\begin{lemma}\label{buchonglemma1}
Assume that $z$ is a solution to equation \eqref{eq61}. Then for every $i \in \{1,2, \cdots, m\}$ and $\alpha \in \{1,2, \cdots, N\}$, the localized process $\widetilde{z}^{i,\alpha}$ satisfies the following equation in the distributional sense over $U^\alpha$:  
\begin{eqnarray}\label{eq62}
    d \widetilde{z}^{i,\alpha}_t(\theta)
 = -\widetilde{z}^{i,\alpha}_t(\theta)\widetilde{\nabla_{\Gamma_t}\cdot v}^\alpha(\theta)dt + g^{kl}_{t,\alpha}(\theta)\frac{\partial \widetilde{R}^{ij,\alpha}_t}{\partial \theta_l}(\theta)\frac{\partial X^{\alpha}_{t,j}}{\partial \theta_{k}}(\theta)dt + \widetilde{S}^{i,\alpha}_t(\theta)dt
    + \widetilde{\Upsilon}^{i,\alpha}_t(\theta)dB_t.
\end{eqnarray}
\end{lemma}\label{buchonglemma2}
\begin{proof}
For any $\psi \in C_c^{\infty}(U^\alpha)$, let
\begin{numcases}{\varphi(t,x)=}
  \frac{\psi((X^\alpha_t)^{-1}(x))}{\sqrt{g^\alpha_t((X^\alpha_t)^{-1}(x))}}, & $(t,x)\in \bigcup_{t \in [0,T]} \{t\}\times V_t^{\alpha}$,\nonumber\\
  0, & $(t,x) \in Q_T\setminus\bigcup_{t \in [0,T]} \{t\}\times V_t^{\alpha}$,\nonumber
\end{numcases}
where $ V_t^{\alpha} = X^\alpha_t(U^\alpha)$. It follows that $\varphi \in H^1(Q_T)$. Then by \eqref{eq'76} and the facts that $\frac{\partial \sqrt{g^\alpha_t(\theta)}}{\partial t} = \sqrt{g^\alpha_t(\theta)}\widetilde{\nabla_{\Gamma_t}\cdot v}^\alpha(\theta)$ and $\widetilde{(H\nu_j)}^\alpha_t(\theta)=-\frac{1}{\sqrt{g^{\alpha}_t(\theta)}}\frac{\partial}{\partial \theta_l}\big( g^{kl}_{t,\alpha}(\theta)\frac{\partial X^{\alpha}_{t,j}}{\partial \theta_{k}}(\theta)\sqrt{g^{\alpha}_t(\theta)} \big) $, we obtain that $\mathbb{P}$-a.s.
\begin{eqnarray*}\label{buchongeq2}
& & \int_{U^\alpha} \widetilde{z}^{i,\alpha}_t(\theta)\psi(\theta)d\theta- \int_{U^\alpha} \widetilde{z}^{i,\alpha}_0(\theta)\psi(\theta)d\theta \nonumber\\
& =& - \int_0^t\int_{U^\alpha}\widetilde{z}^{i,\alpha}_s(\theta)\widetilde{\nabla_{\Gamma_s}\cdot v}^\alpha(\theta)\psi(\theta)d\theta ds
 -\int_0^t\int_{U^\alpha} \widetilde{R}^{ij,\alpha}_s(\theta)g^{kl}_{s,\alpha}(\theta)\frac{\partial(\psi/\sqrt{g^{\alpha}_s})}{\partial \theta_l}(\theta)\frac{\partial X^{\alpha}_{s,j}}{\partial \theta_{k}}(\theta)\sqrt{g^{\alpha}_s(\theta)} d\theta ds\nonumber\\
& & -\int_0^t\int_{U^\alpha} \frac{\widetilde{R}^{ij,\alpha}_s(\theta)\psi(\theta)}{\sqrt{g^{\alpha}_s(\theta)} }\frac{\partial}{\partial \theta_l}\big( g^{kl}_{s,\alpha}(\theta)\frac{\partial X^{\alpha}_{s,j}}{\partial \theta_{k}}(\theta)\sqrt{g^{\alpha}_s(\theta)} \big)d\theta ds   + \int_0^t\int_{U^\alpha} \widetilde{S}^{i,\alpha}_s(\theta)\psi(\theta)d\theta ds \nonumber\\
& & +\int_0^t\int_{U^\alpha} \widetilde{\Upsilon}^{i,\alpha}_s( \theta)\psi(\theta)d\theta dB_s \nonumber\\
\end{eqnarray*}
for every $t \in [0,T]$. Since $\psi \in C_c^{\infty}(U^\alpha)$ is arbitrary, we conclude that $\widetilde{z}^{i,\alpha}$ satisfies equation \eqref{eq62} in the distributional sense over $U^\alpha$.
\end{proof}

\begin{lemma}\label{lemma2.2}
(It\^{o} formula) Assume that $z$ is a solution to equation \eqref{eq61}. For any $\varrho \in [C^2(Q_T)]^m$, let $\Psi(t,x)= \int_0^t \varrho(s, G(s, G^{-1}(t,x)))dB_s$, $(t,x) \in Q_T$. Then for any $\eta \in C^1(Q_T)$ and $\phi \in C^2(\mathbb{R}^m)$ with bounded second-order partial derivatives, we have $\mathbb{P}$-a.s.
\begin{eqnarray}\label{eq66}
& &  \langle \phi(z(t) - \Psi(t)), \eta(t) \rangle_{\Gamma_t} \nonumber\\
&=&  \langle \phi(z_0), \eta(0) \rangle_{\Gamma_0}
      + \int_0^t \langle \big( \phi(z(s) - \Psi(s))- \nabla\phi(z(s) - \Psi(s))\cdot z(s) \big)\nabla_{\Gamma_s}\cdot v(s), \eta(s) \rangle_{\Gamma_s}ds\nonumber\\
   & &+ \int_0^t \langle \phi(z(s) - \Psi(s)), \overset{\scalebox{0.4}{$\bullet$}}{\eta}(s) \rangle_{\Gamma_s}ds
      + \int_0^t \langle \nabla\phi(z(s) -\Psi(s))\cdot (\nabla_{\Gamma_s}\cdot R(s)), \eta(s) \rangle_{\Gamma_s}ds \nonumber\\
   & & + \int_0^t \langle \nabla\phi(z(s) - \Psi(s))\cdot S(s), \eta(s) \rangle_{\Gamma_s}ds
       + \int_0^t \langle \nabla\phi(z(s) - \Psi(s))\cdot \big( \Upsilon(s)- \varrho(s)\big), \eta(s)\rangle_{\Gamma_s}dB_s \nonumber\\
   & & +\frac{1}{2}\int_0^t \langle (\Upsilon(s)-\varrho(s) )^\top \nabla^2\phi(z(s) -\Psi(s))(\Upsilon(s)- \varrho(s)), \eta(s) \rangle_{\Gamma_s}ds,
\end{eqnarray}
for all $t \in [0,T]$.
\end{lemma}

\begin{remark}\label{remark1}
The term involving $R$ in the above equality is understood as follows:
\begin{eqnarray*}\label{eq67}
& &  _{H^{-1}(\Gamma_s)}\langle \nabla\phi(z(s) -\Psi(s))\cdot (\nabla_{\Gamma_s}\cdot R(s)), \eta(s) \rangle_{H^1(\Gamma_s)} \nonumber\\
&=&_{H^{-1}(\Gamma_s)}\langle \partial_{y_i}\phi(z(s) -\Psi(s))\cdot \underline{D}_{j}R^{ij}(s), \eta(s) \rangle_{H^1(\Gamma_s)} \nonumber\\
&=&- \langle  R^{ij}(s)\partial^2_{y_iy_k}\phi(z(s) - \Psi(s))\underline{D}_{j}(z^k(s) - \Psi^k(s)), \eta(s) \rangle_{\Gamma_s}
   - \langle \underline{D}_{j}\eta(s)R^{ij}(s), \partial_{y_i}\phi(z(s)- \Psi(s)) \rangle_{\Gamma_s} \nonumber\\
   & &+ \langle R^{ij}(s)\partial_{y_i}\phi(z(s) - \Psi(s))H\nu_j, \eta(s) \rangle_{\Gamma_s},\nonumber\\
\end{eqnarray*}
where $ \partial_{y_i}\phi(y)$ and $\partial^2_{y_iy_k}\phi(y)$ stand for the first-order and second-order partial derivatives of $\phi$ with respect to $y_i$ and $(y_i, y_k)$, respectively.

\end{remark}

\begin{proof}
By Lemma \ref{buchonglemma1}, for each $i \in \{1,2, \cdots, m\}$ and $\alpha \in \{1,2, \cdots, N\}$, the process $\widetilde{z}^{i,\alpha}_t(\theta) - \Psi^i(t, X^\alpha_t(\theta))$ satisfies the following equation in the distributional sense over $U^\alpha$:
\begin{eqnarray*}\label{buchongeq3}
    d \big( \widetilde{z}^{i,\alpha}_t(\theta) - \Psi(t, X^\alpha_t(\theta)) \big)
 &=& -\widetilde{z}^{i,\alpha}_t(\theta)\widetilde{\nabla_{\Gamma_t}\cdot v}^\alpha(\theta)dt + g^{kl}_{t,\alpha}(\theta)\frac{\partial \widetilde{R}^{ij,\alpha}_t}{\partial \theta_l}(\theta)\frac{\partial X^{\alpha}_{t,j}}{\partial \theta_{k}}(\theta)dt  \nonumber\\
    & & + \widetilde{S}^{i,\alpha}_t(\theta)dt + \big( \widetilde{\Upsilon}^{i,\alpha}_t(\theta) - \varrho^i(t, X^\alpha_t(\theta)) \big)dB_t.
\end{eqnarray*}

Recall that $\{ \chi^\alpha\}_{\alpha \leq N}$ is a partition of unity with regard to the open cover $\{ V^\alpha = X^\alpha(U^\alpha) \}_{\alpha \leq N}$ on $\Gamma_0$. Then $\{\chi^\alpha_t= \chi^\alpha \circ G^{-1}(t, \cdot)\}_{\alpha \leq N}$ is a partition of unity with regard to the open cover $\{ G(t, V^\alpha) \}_{\alpha \leq N}$ on $\Gamma_t$. Now, for any $\eta \in C^1(Q_T)$ and $\phi \in C^2(\mathbb{R}^m)$ with bounded second-order partial derivatives, applying It\^o's formula to
\[ \phi\big( \widetilde{z}^{\alpha}_t(\theta)- \widetilde{\Psi}^{\alpha}_t(\theta) \big)\widetilde{\eta}^\alpha_t(\theta)\chi^\alpha(X^\alpha(\theta))\sqrt{g^\alpha_t(\theta)} \]
and integrating with respect to $\theta$ yields
\begin{eqnarray*}\label{eq65}
& &  \langle \chi^\alpha_t \phi(z(t) - \Psi(t)), \eta(t) \rangle_{\Gamma_t} \nonumber\\
&=&  \langle \chi^\alpha\phi(z_0), \eta(0) \rangle_{\Gamma_0}
      + \int_0^t \langle \chi^\alpha_s \big( \phi(z(s) - \Psi(s))- \nabla\phi(z(s) - \Psi(s))\cdot z(s) \big)\nabla_{\Gamma_s}\cdot v(s), \eta(s) \rangle_{\Gamma_s}ds\nonumber\\
   & &+ \int_0^t \langle \chi^\alpha_s\phi(z(s) - \Psi(s)), \overset{\scalebox{0.4}{$\bullet$}}{\eta}(s) \rangle_{\Gamma_s}ds
       + \int_0^t \langle \chi^\alpha_s \nabla\phi(z(s) -\Psi(s))\cdot (\nabla_{\Gamma_s}\cdot R(s)), \eta(s) \rangle_{\Gamma_s}ds \nonumber\\
   & & + \int_0^t \langle \chi^\alpha_s\nabla\phi(z(s) - \Psi(s))\cdot S(s), \eta(s) \rangle_{\Gamma_s}ds
       + \int_0^t \langle \chi^\alpha_s\nabla\phi(z(s) - \Psi(s))\cdot \big( \Upsilon(s)- \varrho(s)\big), \eta(s)\rangle_{\Gamma_s}dB_s \nonumber\\
   & &+ \frac{1}{2}\int_0^t \langle \chi^\alpha_s(\Upsilon(s) - \varrho(s))^\top \nabla^2\phi(z(s) -\Psi(s))(\Upsilon(s) - \varrho(s)), \eta(s) \rangle_{\Gamma_s}ds.
      \nonumber\\
\end{eqnarray*}
Summing both sides of the above equality with respect to $\alpha$ from $1$ to $N$  yields \eqref{eq66}, thereby completing the proof.

\end{proof}

\section{Stochastic generalized entropy solutions--definition and main results}

\subsection{Definition}
In this subsection, we introduce the precise definition of generalized entropy solutions. Let us begin by introducing the assumptions on the coefficients $f$ and $\sigma$ in equation \eqref{eq74}.
\begin{Assumption}\label{Assumption1}
Let $f: Q_T \times \mathbb{R} \rightarrow \mathbb{R}^{n+1}$ be twice continuously differentiable and satisfy the following conditions:
\begin{itemize}
  \item [(B1)]\label{aaumB1} $f((t,x),u)$ is a tangent vector at the surface $\Gamma_t$ for $t \in [0,T]$, $x \in \Gamma_t$ and $u \in \mathbb{R}$.
  \item [(B2)]\label{assumB2} For all $((t,x), u) \in Q_T \times \mathbb{R}$,
  \[ \nabla_{\Gamma_t}\cdot f((t, x), u)=0.  \]
  \item [(B3)]\label{assumB3} There exist $q_0 \geq 2$ and a positive constant $L$ such that for every $((t,x), u) \in Q_T \times \mathbb{R}$,
  \begin{enumerate}[(i)]\label{flux}
			\item\label{i} $|f((t,x),u)| \leq L(1 + |u|^{q_0})$,
            \item\label{ii} $|f_u((t,x),u)| \vee |\nabla_{\Gamma_t}f_u((t,x),u)| \vee |\overset{\scalebox{0.4}{$\bullet$}}{f_u}((t,x),u)| \leq L(1+ |u|)$,
			\item\label{iii}$|\nabla_{\Gamma_t}f((t,x),u)|\vee | \overset{\scalebox{0.4}{$\bullet$}}{f}((t,x),u)| \leq L(1+|u|^{q_0})$,
			\item \label{iv} $|f_{uu}((t,x),u)|\vee |\nabla_{\Gamma_t}^2f((t,x),u)|\vee |\nabla_{\Gamma_t}\overset{\scalebox{0.4}{$\bullet$}}{f}((t,x),u)|\vee |\overset{\scalebox{0.4}{$\bullet$$\bullet$}}{f}((t,x),u)| \leq L(1+|u|^{q_0})$.
  \end{enumerate}
  Here $f_u$ and $f_{uu}$ represent the first-order and second-order partial derivatives of $f$ with respect to the variable $u$, respectively.
\end{itemize}

\end{Assumption}

\begin{remark}
Condition (B2) is the so-called geometry-compatible condition. Moreover, conditions (B1) and (B2) are imposed in order to formulate the conservation law intrinsically on the moving hypersurface (see Section 3 of \cite{GDT} for details).
\end{remark}

\begin{Assumption}\label{Assumption2}
Let $\sigma: Q_T \rightarrow \mathbb{R}$ be four times continuously differentiable.
\end{Assumption}

We now introduce several functions associated with $f$. For any $\varrho \in C^2(Q_T)$ and $\phi \in C^2(\mathbb{R})$ with $\phi^{\prime\prime}$ bounded, we denote $q_{\varrho,\phi}: Q_T \times \mathbb{R} \times \mathbb{R} \rightarrow \mathbb{R}^{n+1}$ and $l_{\varrho,\phi}: Q_T \times \mathbb{R} \times \mathbb{R} \rightarrow \mathbb{R}$ by
\[q_{\varrho,\phi, k}((t,x),\lambda ; \mu) = \int_{\mu}^{\lambda} \phi^{\prime}(\tau - \mu)f_{ku}((t,x),\tau + \Psi(t,x))d\tau, \quad k=1,2,\cdots, n+1,   \]
and
\[ l_{\varrho,\phi}((t,x),\lambda ; \mu)= \int_{\mu}^{\lambda} \phi^{\prime}(\tau - \mu) \nabla_{\Gamma_t} \cdot \big(f_{u}((t,x),\tau + \Psi(t,x))\big)d\tau,  \]
respectively, where $\Psi(t,x) = \int_0^t \varrho(s, G(s, G^{-1}(t,x)))dB_s$. For simplicity, we write $q_{\varrho,\phi}((t,x),\lambda ; 0)$ and $l_{\varrho,\phi}((t,x),\lambda ; 0)$ as $q_{\varrho,\phi}((t,x),\lambda)$ and $l_{\varrho,\phi}((t,x),\lambda)$, respectively. 

\vskip 0.4cm
For the motivation of the definition of  the generalized entropy solution  given below (\eqref{eq72}), let us consider the following Lemma. We assume that $u_{0}^{\epsilon} \in H^{2}(\Gamma_0)\cap L^{\infty}(\Gamma_0)$ and
\begin{equation}\label{eq58}
\| u_{0}^{\epsilon} \|_{L^{\infty}(\Gamma_0)} + \| \nabla_{\Gamma_0}u_{0}^{\epsilon}\|_{L^1(\Gamma_0)} + \epsilon\| \nabla^2_{\Gamma_0} u_{0}^{\epsilon} \|_{L^1(\Gamma_0)} \leq C_0
\end{equation}
with a positive constant $C_0$ which is independent of the parameter $\epsilon \in (0,1)$.
\begin{lemma}\label{lemma2.4}
Assume that $u^\epsilon \in L^2(\Omega; C([0,T]; H^1(\Gamma_{\cdot}))) \cap L^2(\Omega; L^2([0,T]; H^2(\Gamma_{\cdot})))$ satisfies the following equation in $H^{-1}(\Gamma_{\cdot})$,
\begin{align}
\left\{
\begin{aligned}\label{eq68}
&\overset{\scalebox{0.4}{$\bullet$}}{u^\epsilon}(t)+u^\epsilon(t)\nabla_{\Gamma_t}\cdot v(t) + \nabla_{\Gamma_t} \cdot \big(f((t,\cdot), u^\epsilon(t))\big) = \epsilon\Delta_{\Gamma_t}u^\epsilon(t) + \sigma(t)\overset{\scalebox{0.4}{$\bullet$}}{B_t}, \quad (t,x)\in Q_T,\\
&u^\epsilon(0,y)=u_{0}^{\epsilon}(y), \quad y\in\Gamma_0.
\end{aligned}
\right.
\end{align}
If $\mathbb{P}$-a.s. $u^\epsilon \rightarrow u$ a.e. on $Q_T$ and $u^\epsilon_0 \rightarrow u_0$ a.e. on $\Gamma_0$ for $\epsilon \rightarrow 0$, 
then for any $\varrho \in C^2(Q_T)$, $\phi \in C^2(\mathbb{R})$ with $\phi^{\prime\prime} \geq 0$, and any $\eta \in C^1(Q_T)$ satisfying $\eta \geq 0$ and $\eta(\cdot, T)= 0$, $u$ satisfies the entropy condition
\begin{eqnarray}\label{eq72}
 & &  \langle \phi(u_0), \eta(0) \rangle_{\Gamma_0}
   + \int_0^T \langle \big(\phi(u(s)- \Psi(s))-u(s)\phi^{\prime}(u(s) - \Psi(s))\big)\nabla_{\Gamma_s}\cdot v(s), \eta(s) \rangle_{\Gamma_s} ds \nonumber\\
 & &  + \int_0^T \langle \phi(u(s)- \Psi(s)), \overset{\scalebox{0.4}{$\bullet$}}{\eta}(s) \rangle_{\Gamma_s} ds
 +\int_0^T \int_{\Gamma_s} q_{\varrho,\phi}((s,x),u(s,x)-\Psi(s,x))\cdot \nabla_{\Gamma_s}\eta(s,x)\mathrm{dvol}_{\Gamma_s}(dx)ds\nonumber\\
 & & -\int_0^T \langle \phi^{\prime}(u(s)- \Psi(s))f_u((s,\cdot), u(s))\cdot \nabla_{\Gamma_s}\Psi(s), \eta(s) \rangle_{\Gamma_s} ds
     +\int_0^T  \langle l_{\varrho,\phi}((s,\cdot),u(s) - \Psi(s)), \eta(s) \rangle_{\Gamma_s} ds \nonumber\\
 & &  + \int_0^T \langle \phi^{\prime}(u(s)- \Psi(s))(\sigma(s)- \varrho(s)), \eta(s) \rangle_{\Gamma_s} dB_s
      + \frac{1}{2}\int_0^T \langle \phi^{\prime\prime}(u(s)- \Psi(s))(\sigma(s)-\varrho(s))^2, \eta(s) \rangle_{\Gamma_s} ds \nonumber\\
& &  \geq 0,
\end{eqnarray}
where $\Psi(s,x) = \int_0^s \varrho(r, G(r, G^{-1}(s,x)))dB_r$.

\end{lemma}

\begin{proof}
For any $\varrho \in C^2(Q_T)$, let $\Psi(s,x) = \int_0^s \varrho(r, G(r, G^{-1}(s,x)))dB_r$. Then for any $\phi \in C^2(\mathbb{R})$ with $\phi^{\prime\prime} \geq 0$, and any $\eta \in C^2(Q_T)$ satisfying $\eta \geq 0$ and $\eta(\cdot, T)= 0$, it follows from Lemma \ref{lemma2.2} that
\begin{eqnarray}\label{eq69}
 & &  \langle \phi(u_0^\epsilon), \eta(0) \rangle_{\Gamma_0}
   + \int_0^T \langle \big(\phi(u^\epsilon(s)- \Psi(s))-u^\epsilon(s)\phi^{\prime}(u^\epsilon(s) - \Psi(s))\big)\nabla_{\Gamma_s}\cdot v(s), \eta(s) \rangle_{\Gamma_s} ds \nonumber\\
 & &  + \int_0^T \langle \phi(u^\epsilon(s)- \Psi(s)), \overset{\scalebox{0.4}{$\bullet$}}{\eta}(s) \rangle_{\Gamma_s} ds
  -\int_0^T \langle \phi^{\prime}(u^\epsilon(s)- \Psi(s))\nabla_{\Gamma_s}\cdot \big(f((s,\cdot), u^\epsilon(s))\big), \eta(s) \rangle_{\Gamma_s} ds\nonumber\\
  & &+ \epsilon\int_0^T\langle  \phi^{\prime\prime}(u^\epsilon(s)- \Psi(s))\nabla_{\Gamma_s}\Psi(s)\cdot \nabla_{\Gamma_s}u^\epsilon(s),  \eta(s) \rangle_{\Gamma_s} ds \nonumber\\
  & & - \epsilon \int_0^T \langle \nabla_{\Gamma_s}\eta(s)\cdot \nabla_{\Gamma_s}u^\epsilon(s), \phi^{\prime}(u^\epsilon(s)- \Psi(s))\rangle_{\Gamma_s} ds
      + \int_0^T \langle \phi^{\prime}(u^\epsilon(s)- \Psi(s))(\sigma(s)- \varrho(s)), \eta(s) \rangle_{\Gamma_s} dB_s \nonumber\\
 & &  + \frac{1}{2}\int_0^T \langle \phi^{\prime\prime}(u^\epsilon(s)- \Psi(s))(\sigma(s)-\varrho(s))^2, \eta(s) \rangle_{\Gamma_s} ds \nonumber\\
  & &= \epsilon \int_0^T \langle \phi^{\prime\prime}(u^\epsilon(s)- \Psi(s))|\nabla_{\Gamma_s}u^\epsilon(s)|^2, \eta(s)\rangle_{\Gamma_s} ds \geq 0.
 \end{eqnarray}

By Assumption \ref{Assumption1}(B2) and the definitions of $q_{\varrho,\phi}$ and $l_{\varrho,\phi}$, we observe that
\begin{eqnarray}\label{eq70}
  & &  \phi^{\prime}(u^\epsilon(s)- \Psi(s))\nabla_{\Gamma_s}\cdot \big(f((s,x), u^\epsilon(s))) \nonumber\\
&=&\phi^{\prime}(u^\epsilon(s)- \Psi(s))f_u((s,x), u^\epsilon(s))\cdot \nabla_{\Gamma_s}u^\epsilon(s) \nonumber\\
&=&\nabla_{\Gamma_s} \cdot \big(q_{\varrho,\phi}((s,x), u^\epsilon(s) - \Psi(s))\big) + \phi^{\prime}(u^\epsilon(s)- \Psi(s))f_u((s,x), u^\epsilon(s))\cdot \nabla_{\Gamma_s}\Psi(s)\nonumber\\
   & & -l_{\varrho,\phi}((s,x), u^\epsilon(s) - \Psi(s)).
\end{eqnarray}
Furthermore, using Lemma \ref{lemma2.3} and Assumption \ref{Assumption1}(B1), we get
\begin{eqnarray}\label{fdexiding1}
& & \int_0^T \langle \nabla_{\Gamma_s} \cdot \big(q_{\varrho,\phi}((s,x), u^\epsilon(s) - \Psi(s))\big) , u^\epsilon(s))\big), \eta(s) \rangle_{\Gamma_s} ds \nonumber\\
   & & = - \int_0^T \int_{\Gamma_s} q_{\varrho,\phi}((s,x),u(s,x)-\Psi(s,x))\cdot \nabla_{\Gamma_s}\eta(s,x)\mathrm{dvol}_{\Gamma_s}(dx)ds
\end{eqnarray}
Therefore, substituting \eqref{eq70} and \eqref{fdexiding1} into \eqref{eq69}, and then letting $\epsilon \rightarrow 0$, we get \eqref{eq72}. The proof is complete.
\end{proof}

Now, the property \eqref{eq72} motivates the definition of a generalized entropy solution below.
\begin{definition}\label{definition1}
Let $u_0 \in L^\infty(\Gamma_0)$. We say that a stochastic process $\{u(t)\}_{t\in[0,T]}$ is a generalized entropy solution to \eqref{eq74} if
\begin{itemize}
\item[(i)] $\mathbb{P}$-a.s. $u \in L^\infty(Q_T)$.
\item[(ii)] $u$ is a predictable random process with respect to the filtration $ \{ \mathcal{F}_t \} _{t\in[0,T]} $.
\item[(iii)] For any $\varrho \in C^2(Q_T)$, $\phi \in C^2(\mathbb{R})$ with $\phi^{\prime\prime} \geq 0$, and any $\eta \in C^1(Q_T)$ with $\eta \geq 0$ and $\eta(\cdot, T)= 0$, $\mathbb{P}$-a.s.
\begin{eqnarray}\label{eq73}
 & &  \langle \phi(u_0), \eta(0) \rangle_{\Gamma_0}
   + \int_0^T \langle \big(\phi(u(s)- \Psi(s))-u(s)\phi^{\prime}(u(s) - \Psi(s))\big)\nabla_{\Gamma_s}\cdot v(s), \eta(s) \rangle_{\Gamma_s} ds \nonumber\\
 & &  + \int_0^T \langle \phi(u(s)- \Psi(s)), \overset{\scalebox{0.4}{$\bullet$}}{\eta}(s) \rangle_{\Gamma_s} ds
   + \int_0^T \langle \phi^{\prime}(u(s)- \Psi(s))(\sigma(s)- \varrho(s)), \eta(s) \rangle_{\Gamma_s} dB_s \nonumber\\
 & & + \frac{1}{2}\int_0^T \langle \phi^{\prime\prime}(u(s)- \Psi(s))(\sigma(s)-\varrho(s))^2, \eta(s) \rangle_{\Gamma_s} ds \nonumber\\
 & & + \int_0^T \int_{\Gamma_s} q_{\varrho,\phi}((s,x),u(s,x)-\Psi(s,x))\cdot \nabla_{\Gamma_s}\eta(s,x) \mathrm{dvol}_{\Gamma_s}(dx)ds \nonumber\\
  & & -\int_0^T \langle \phi^{\prime}(u(s)- \Psi(s))f_u((s,\cdot), u(s))\cdot \nabla_{\Gamma_s}\Psi(s), \eta(s) \rangle_{\Gamma_s} ds\nonumber\\
  & &  +\int_0^T  \langle l_{\varrho,\phi}((s,\cdot),u(s) - \Psi(s)), \eta(s) \rangle_{\Gamma_s} ds
  \geq 0,
\end{eqnarray}
where $\Psi(s,x) = \int_0^s \varrho(r, G(r, G^{-1}(s,x)))dB_r$.
\end{itemize}

\end{definition}

Let $u$ be a generalized entropy solution to the stochastic conservation law \eqref{eq74}. Let
\[ V(t,x) = u(t,x) - w(t,x), \quad (t,x) \in Q_T, \]
where
\begin{equation*}\label{eq86}
w(t,x) = \int_0^t \sigma(s, G(s, G^{-1}(t,x))) dB_s.
\end{equation*}
Then $V$ satisfies the following Kruzkov entropy condition, which plays an important role in proving the uniqueness of solutions to \eqref{eq74}.
\begin{remark}\label{remark3}
For any $k \in\mathbb{R}$, and any $\eta \in C^2(Q_T)$ with $\eta \geq 0$ and $\eta(\cdot, T)= 0$, $V$ satisfies the following Kruzkov entropy condition:
\begin{eqnarray}\label{eq87}
& & \langle |u_0 - k|, \eta(0) \rangle_{\Gamma_0}
    - \int_0^T \langle \mathrm{sign}(V(s) - k)(k+ w(s))\nabla_{\Gamma_s}\cdot v(s), \eta(s)\rangle_{\Gamma_s} ds\nonumber\\
& & + \int_0^T \langle |V(s) - k|, \overset{\scalebox{0.4}{$\bullet$}}{\eta}(s) \rangle_{\Gamma_s} ds
    - \int_0^T \langle \mathrm{sign}(V(s) - k)f_u((s,\cdot), k+w(s))\cdot \nabla_{\Gamma_s}w(s), \eta(s) \rangle_{\Gamma_s}ds\nonumber\\
& &+ \int_0^T\int_{\Gamma_s}\mathrm{sign}(V(s,x) - k)\big( f((s,x), V(s,x) + w(s,x)) - f((s,x), k+ w(s,x)) \big) \nonumber\\
   & & \cdot \nabla_{\Gamma_s}\eta(s,x)\mathrm{dvol}_{\Gamma_s}(dx)ds
\geq 0.
\end{eqnarray}

\begin{proof}
Let $\zeta: \mathbb{R} \rightarrow \mathbb{R}$ be a nonnegative smooth function satisfying
\[ \zeta(0)=0, \quad \zeta(-r) = \zeta(r), \quad \zeta^{\prime\prime}\geq 0, \]
and
\begin{numcases}{\zeta^{\prime}(r)=}
	-1, \text{   when $r\leq -1$,}\nonumber\\
	\in [-1,1],\text{ when $|r|< 1$},\nonumber\\
	1,\ \text{      when $r\geq 1$}.\nonumber
\end{numcases}
For any $\epsilon > 0$, define $\zeta_\epsilon: \mathbb{R} \rightarrow \mathbb{R}$ by
\[ \zeta_{\epsilon}(r) = \epsilon\zeta(\frac{r}{\epsilon}).  \]
Then
\[  |r|-M_1\epsilon \leq \zeta_{\epsilon}(r) \leq |r| \ \text{and} \ |\zeta^{\prime\prime}_{\epsilon}(r)| \leq \frac{M_2}{\epsilon}I_{|r|\leq\epsilon},  \]
where
\[ M_1= \sup_{|r|\leq 1}\big||r|- \zeta(r)\big|, \quad M_2=\sup_{|r|\leq 1}|\zeta^{\prime\prime}(r)|.  \]

For any $\epsilon > 0$, $k \in \mathbb{R}$, and $\eta \in C^2(Q_T)$ with $\eta \geq 0$ and $\eta(\cdot, T)= 0$, taking $\phi(r) = \zeta_{\epsilon}(r-k)$ and $\Psi(s,x) =w(s,x)$ in \eqref{eq73}, we obtain
\begin{eqnarray}\label{eq88}
 & &  \langle \zeta_{\epsilon}(u_0 - k), \eta(0) \rangle_{\Gamma_0}
   + \int_0^T \langle \big(\zeta_\epsilon(V(s) - k)-(V(s)+w(s))\zeta_{\epsilon}^{\prime}(V(s) -k)\big)\nabla_{\Gamma_s}\cdot v(s), \eta(s) \rangle_{\Gamma_s} ds \nonumber\\
 & &  + \int_0^T \langle \zeta_{\epsilon}(V(s) -k), \overset{\scalebox{0.4}{$\bullet$}}{\eta}(s) \rangle_{\Gamma_s} ds
     + \int_0^T \int_{\Gamma_s} q_{\sigma,\zeta_{\epsilon}}((s,x),V(s,x);k)\cdot \nabla_{\Gamma_s}\eta(s,x) \mathrm{dvol}_{\Gamma_s}(dx)ds \nonumber\\
  & &   -\int_0^T \langle \zeta_{\epsilon}^{\prime}(V(s) - k)f_u((s,\cdot),V(s)+w(s))\cdot \nabla_{\Gamma_s}w(s), \eta(s) \rangle_{\Gamma_s} ds
  +\int_0^T  \langle l_{\sigma,\zeta_{\epsilon}}((s,\cdot),V(s); k), \eta(s) \rangle_{\Gamma_s} ds \nonumber\\
 & &  \geq 0.
\end{eqnarray}
Notice that for every $r \in \mathbb{R}$,
\[ \zeta_\epsilon(r) \rightarrow |r|, \quad \zeta^{\prime}_\epsilon(r) \rightarrow \mathrm{sign}(r) \ \text{as}\ \epsilon \rightarrow 0.  \]
Letting $\epsilon \rightarrow 0$ in \eqref{eq88}, we get \eqref{eq87}. The proof is complete.
\end{proof}
\end{remark}

\subsection{Statements of main results}
In this subsection, we state the main results of this paper.
\begin{proposition}\label{Thm3}
Let Assumptions \ref{Assumption1} and \ref{Assumption2} hold. Then for every $u_0 \in H^{1,1}(\Gamma_0) \cap L^{\infty}(\Gamma_0)$, there exists a martingale entropy solution $({\Omega}^*,{\mathcal{F}}^*,\{{\mathcal{F}}_t^*\}_{t \in [0,T]},\mathbb{P}^*,u^*,B^*)$ to the stochastic conservation law \eqref{eq74} started from $u_0$. Moreover, $\mathbb{P}^*$-a.s.
\begin{equation}\label{eq180}
\|u^*\|_{L^{\infty}(Q_T)} \leq e^{C^*}(1+ \| u_0 \|_{L^\infty(\Gamma_0)}),
\end{equation}
where
\begin{eqnarray*}
C^*&=& (T+1)(L_1+1)(1+\sup_{t\in[0,T]}\|\nabla_{\Gamma_t} \cdot v(t)\|_{L^{\infty}(\Gamma_t)} +\sup_{t\in[0,T]}\|w^*(t)\|_{L^{\infty}(\Gamma_t)}  \nonumber\\
      & & + \sup_{t\in[0,T]}\|\nabla_{\Gamma_t}w^*(t)\|_{L^{\infty}(\Gamma_t)} + \sup_{t\in[0,T]}\|\Delta_{\Gamma_t}w^*(t)\|_{L^{\infty}(\Gamma_t)}  ),
\end{eqnarray*}
and $w^{*}(t,x) = \int_0^t \sigma(s, G(s, G^{-1}(t,x))) dB^{*}_s, \ (t,x) \in Q_T$.
\end{proposition}
\begin{proof}
The proof of existence is based on the vanishing viscosity method. The argument is divided into three parts, which are respectively given in the next three sections. We first show in Section \ref{The regularized problem} that for every $\epsilon \in (0,1)$, the associated stochastic viscous conservation law \eqref{eq68} admits a unique solution $u^\epsilon$. Then, in Section \ref{A priori estimates for the regularized problem}, we establish a priori estimates for $u^\epsilon$ that are uniform in $\epsilon$, including the $L^\infty$-bound in both space and time, the $L^1$-estimate for the spatial gradient, and the $L^1$-continuity in time. These estimates imply that $\{u^\varepsilon\}_{\epsilon \in (0,1)}$ is tight in $L^1([0,T]; L^1(\Gamma_\cdot))$. Consequently, the existence of a martingale entropy solution to the stochastic conservation law \eqref{eq74} is proved in Section \ref{Existence for the conservation law}.
\end{proof}

\begin{proposition}\label{Thm5}($L^1$-contraction)
Let Assumptions \ref{Assumption1} and \ref{Assumption2} hold. Suppose that $u_1$, $u_2$ are two generalized entropy solutions to the stochastic conservation law \eqref{eq74} with initial values $u_{1,0}$, $u_{2,0} \in L^{\infty}(\Gamma_0)$, respectively. Then there exists a positive constant $C_{1}$ such that $\mathbb{P}$-a.s.
\begin{eqnarray*}
    & & \mathrm{ess\,sup}_{t\in [0,T]}\|u_{1}(t) - u_{2}(t)\|_{L^1(\Gamma_t)} \nonumber\\
&\leq& C_{1} e^{C_{1} (1+ \|u_1\|_{L^\infty(Q_T)} + \|u_2\|_{L^\infty(Q_T)})}\|u_{1,0} - u_{2,0}\|_{L^1(\Gamma_0)}.
\end{eqnarray*}
In particular, the generalized entropy solution to \eqref{eq74} is unique.
\end{proposition}

\begin{proof}
The proof suitably adapts Kruzkov's method of doubling the variables to the case of moving surfaces. The details is given in Section \ref{Uniqueness for the conservative law}.
\end{proof}
\vskip 0.2cm
By the above propositions, we can further establish the well-posedness of equation \eqref{eq74} for general initial value $u_0 \in L^{\infty}(\Gamma_0)$.
\begin{theorem}\label{Thm4}
Let Assumptions \ref{Assumption1} and \ref{Assumption2} hold. Then for every $u_0 \in L^{\infty}(\Gamma_0)$, there exists a unique generalized entropy solution $\{u(t)\}_{t \in [0,T]}$ to the stochastic conservation law \eqref{eq74} started from $u_0$.
\end{theorem}
\begin{proof}
By Proposition \ref{Thm3}, Proposition \ref{Thm5} and the Yamada-Watanabe theorem, for every initial value $u_0 \in H^{1,1}(\Gamma_0) \cap L^{\infty}(\Gamma_0)$, there exists a unique generalized entropy solution $\{u(t)\}_{t \in [0,T]}$ to the stochastic conservation law \eqref{eq74} such that $\mathbb{P}$-a.s.
\begin{equation}\label{eq191}
\|u\|_{L^{\infty}(Q_T)} \leq e^{C_2}(1+ \| u_0 \|_{L^\infty(\Gamma_0)}),
\end{equation}
where
\begin{eqnarray*}
C_2&=& (T+1)(L_1+1)(1+\sup_{t\in[0,T]}\|\nabla_{\Gamma_t} \cdot v(t)\|_{L^{\infty}(\Gamma_t)} +\sup_{t\in[0,T]}\|w(t)\|_{L^{\infty}(\Gamma_t)} \nonumber\\
     & &  + \sup_{t\in[0,T]}\|\nabla_{\Gamma_t}w(t)\|_{L^{\infty}(\Gamma_t)} + \sup_{t\in[0,T]}\|\Delta_{\Gamma_t}w(t)\|_{L^{\infty}(\Gamma_t)}  ),
\end{eqnarray*}
and $w(t,x) = \int_0^t \sigma(s, G(s, G^{-1}(t,x))) dB_s, \ (t,x) \in Q_T$.

For a general initial value $u_0 \in L^{\infty}(\Gamma_0)$, there exists a sequence $\{u_0^n\}_{n \in \mathbb{N}} \subseteq H^{1,1}(\Gamma_0) \cap L^{\infty}(\Gamma_0)$ such that $u^n_0 \rightarrow u_0$ in $L^1(\Gamma_0)$ as $n \rightarrow \infty$ and $\|u^n_0\|_{L^\infty(\Gamma_0)} \leq  \|u_0\|_{L^\infty(\Gamma_0)}$. Let $u^n$ be the generalized entropy solution to equation \eqref{eq74} with initial value $u^n_0$. Then using the $L^1$ contraction property and \eqref{eq191}, we deduce that there exists $u \in L^1(\Omega \times Q_T; \mathbb{R})$ such that $u^n \rightarrow u$ in $L^1(\Omega \times Q_T; \mathbb{R})$ and for any $p \geq 1$,
\[ \mathbb{E}[\sup_{n \in \mathbb{N}}\| u^n \|^p_{L^\infty(Q_T)}] \leq \infty. \]
Similar to the proof of Lemma \ref{lemma2.4}, one can verify that $u$ is a generalized entropy solution to equation \eqref{eq74}. The uniqueness follows from Proposition \ref{Thm5}.
\end{proof}

\section{The viscous problem}\label{The regularized problem}

For any $\epsilon \in (0,1)$, define the function $f_\epsilon: Q_T \times \mathbb{R} \rightarrow \mathbb{R}^{n+1}$ as
\begin{equation}\label{buchong11}
 f_\epsilon((t,x), u) = \gamma(\epsilon|u|^2)f((t,x),u),
\end{equation}
where $\gamma: \mathbb{R} \rightarrow \mathbb{R}$ is a smooth function with compact support satisfying $0\leq \gamma \leq 1$, $\gamma(y) = 1$ for $|y|<1$ and $\gamma(y)=0$ for $|y|>2$, and $|\gamma^{\prime}| \leq 2$.
In order to solve the stochastic conservative law \eqref{eq74} with initial value $u_0 \in H^{1,1}(\Gamma_0) \cap L^{\infty}(\Gamma_0)$, we first consider the following stochastic viscous conservation law:
\begin{align}
\left\{
\begin{aligned}\label{eq75}
&\overset{\scalebox{0.4}{$\bullet$}}{u^\epsilon}(t)+u^\epsilon(t)\nabla_{\Gamma_t}\cdot v(t) + \nabla_{\Gamma_t} \cdot \big(f_{\epsilon}((t,\cdot), u^\epsilon(t))\big) = \epsilon\Delta_{\Gamma_t}u^\epsilon(t) + \sigma(t)\overset{\scalebox{0.4}{$\bullet$}}{B_t}, \quad (t,x)\in Q_T,\\
&u^\epsilon(0,y)=u_{0}^{\epsilon}(y), \quad y\in\Gamma_0,
\end{aligned}
\right.
\end{align}
where the initial value $u^\epsilon_0$ satisfies \eqref{eq58}.

We now introduce the definition of solution to equation \eqref{eq75}.
\begin{definition}\label{definition2}
We say that a stochastic process $\{u^\epsilon(t)\}_{t\in[0,T]}$ is a solution to equation \eqref{eq75} if
\begin{itemize}
\item[(i)] $u^\epsilon \in L^2(\Omega; C([0,T]; H^1(\Gamma_{\cdot}))) \cap L^2(\Omega; L^2([0,T]; H^2(\Gamma_{\cdot})))$.
\item[(ii)] $u^\epsilon$ is an an ${H}^{1}(\Gamma_{\cdot})$-valued predictable random process with respect to the filtration $ \{ \mathcal{F}_t \} _{t\in[0,T]} $.
\item[(iii)] For any $\varphi \in H^1(Q_T)$, we have $\mathbb{P}$-a.s.
\begin{eqnarray}\label{eq76}
& & \int_{\Gamma_t} u^\epsilon(t,x)\varphi(t,x)\mathrm{dvol}_{\Gamma_t}(dx)- \int_{\Gamma_0} u^\epsilon_0(y)\varphi(0,y)\mathrm{dvol}_{\Gamma_0}(dy) \nonumber\\
& =& \int_0^t\int_{\Gamma_s}u^\epsilon(s,x)\overset{\scalebox{0.4}{$\bullet$}}{\varphi}(s,x) \mathrm{dvol}_{\Gamma_s}(dx)ds + \int_0^t\int_{\Gamma_s}f_{\epsilon}((s,x),u^\epsilon(s,x))\cdot \nabla_{\Gamma_s}\varphi(s,x)\mathrm{dvol}_{\Gamma_s}(dx)ds\nonumber\\
& & -\epsilon\int_0^t\int_{\Gamma_s}\nabla_{\Gamma_s}u^\epsilon(s,x) \cdot \nabla_{\Gamma_s}\varphi(s,x)\mathrm{dvol}_{\Gamma_s}(dx)ds
+\int_0^t\int_{\Gamma_s}\sigma(s,x)\varphi(s,x)\mathrm{dvol}_{\Gamma_s}(dx)dB_s\nonumber\\
\end{eqnarray}
for every $t \in [0,T]$.
\end{itemize}
\end{definition}

With this definition, we have the following result.
\begin{proposition}\label{proposition2}
Let Assumptions \ref{Assumption1} and \ref{Assumption2} hold. Then for every $\epsilon \in (0,1)$, the stochastic viscous conservation law \eqref{eq75} admits a unique solution $u^\epsilon$. Moreover,
\[ u^\epsilon = V^\epsilon +w, \]
where
$w(t,x) = \int_0^t \sigma(s, G(s, G^{-1}(t,x)))dB_s$, $(t,x) \in Q_T$  and $V^\epsilon$ is the unique solution to the following equation:
\begin{align}
\left\{
\begin{aligned}\label{eq77}
&\overset{\scalebox{0.4}{$\bullet$}}{V^\epsilon}(t)+(V^\epsilon(t) + w(t))\nabla_{\Gamma_t}\cdot v(t) + \nabla_{\Gamma_t} \cdot \big(f_{\epsilon}((t,\cdot), V^\epsilon(t)+w(t))\big) = \epsilon\Delta_{\Gamma_t}(V^\epsilon(t) + w(t)), \ (t,x)\in Q_T,\\
&V^\epsilon(0,y)=u_{0}^{\epsilon}(y), \quad y\in\Gamma_0.
\end{aligned}
\right.
\end{align}
Here the definition of solutions to equation \eqref{eq77} is similar to Definition \ref{definition2}.
\end{proposition}

\begin{proof}
Using the similar arguments as in the proof of Theorem 4.4 in \cite{DGCM}, one can show that the equation \eqref{eq77} has a unique solution $V^\epsilon $ such that $\mathbb{P}$-a.s. $V^\epsilon  \in H^1(Q_T)$ and
\begin{eqnarray}\label{eq198}
& & \sup_{t \in [0,T]}\|V^\epsilon(t)\|^2_{L^2(\Gamma_t)} + \int_0^T\| V^\epsilon(t)\|^2_{L^2(\Gamma_t)}dt \nonumber\\
&\leq& C_{T,\epsilon}(1 + \| u_{0}^{\epsilon}\|^2_{L^2(\Gamma_0)} + \int_0^T\| w(t)\|^2_{L^2(\Gamma_t)}dt  + \int_0^T\| \nabla_{\Gamma_t}w(t)\|^2_{L^2(\Gamma_t)}dt  ),
\end{eqnarray}
and
\begin{eqnarray}\label{eq199}
& &  \int_0^T\| \overset{\scalebox{0.4}{$\bullet$}}{V^\epsilon}(t)\|^2_{L^2(\Gamma_t)}dt + \sup_{t \in [0,T]}\|\nabla_{\Gamma_t}V^\epsilon(t)\|^2_{L^2(\Gamma_t)}  \nonumber\\
&\leq& C_{T,\epsilon}(1 + \| u_{0}^{\epsilon}\|^2_{H^1(\Gamma_0)} + \int_0^T\| w(t)\|^2_{L^2(\Gamma_t)}dt  + \int_0^T\| \nabla_{\Gamma_t}w(t)\|^2_{L^2(\Gamma_t)}dt + \int_0^T\| \triangle_{\Gamma_t}w(t)\|^2_{L^2(\Gamma_t)}dt ).  \nonumber\\
\end{eqnarray}
Furthermore, by \eqref{eq198}, \eqref{eq199} and Theorem 4.5 in \cite{DGCM}, we get that $\mathbb{P}$-a.s. $V^\epsilon  \in  L^2([0,T]; H^2(\Gamma_{\cdot}))$ and
\begin{eqnarray}\label{eq200}
& &  \int_0^T\| {V^\epsilon}(t)\|^2_{H^2(\Gamma_t)}dt \nonumber\\
&\leq& C_{T,\epsilon}(1 + \| u_{0}^{\epsilon}\|^2_{H^1(\Gamma_0)} + \int_0^T\| w(t)\|^2_{L^2(\Gamma_t)}dt  + \int_0^T\| \nabla_{\Gamma_t}w(t)\|^2_{L^2(\Gamma_t)}dt + \int_0^T\| \triangle_{\Gamma_t}w(t)\|^2_{L^2(\Gamma_t)}dt ).  \nonumber\\
\end{eqnarray}
Combining \eqref{eq198}, \eqref{eq199} and \eqref{eq200}, it follows that
$$V^\epsilon \in L^2(\Omega; C([0,T]; H^1(\Gamma_{\cdot}))) \cap L^2(\Omega; L^2([0,T]; H^2(\Gamma_{\cdot}))).$$
Now, we denote $u^\epsilon = V^\epsilon +w$. Then $u^\epsilon \in L^2(\Omega; C([0,T]; H^1(\Gamma_{\cdot}))) \cap L^2(\Omega; L^2([0,T]; H^2(\Gamma_{\cdot})))$, and for any $\varphi \in H^1(Q_T)$, $u^\epsilon$ satisfies \eqref{eq76}. The other condition (ii)
in Definition \ref{definition2} follows easily. Thus, the stochastic process $u^\epsilon$ is a solution to the stochastic viscous conservation law \eqref{eq75}.

Below, we prove the uniqueness of equation \eqref{eq75}. For any solutions $u^\epsilon_1$ and $u^\epsilon_2$ of equation \eqref{eq75}. Take $\varphi = u^\epsilon_1 - u^\epsilon_2$ in \eqref{eq76}. Then we have for every $t \in [0,T]$,
\begin{eqnarray*}\label{eq78}
& & \int_{\Gamma_t} |u^\epsilon_{1}(t,x) - u^\epsilon_{2}(t,x)|^2 \mathrm{dvol}_{\Gamma_t}(dx)
    + \epsilon\int_0^t \int_{\Gamma_s} |\nabla_{\Gamma_s}(u^\epsilon_{1}(s,x) - u^\epsilon_{2}(s,x))|^2\mathrm{dvol}_{\Gamma_s}(dx)ds \nonumber\\
&=& \frac{1}{2}\int_0^t\int_{\Gamma_s}\big((u^\epsilon_{1}(s,x) - u^\epsilon_{2}(s,x))^2\big)^{\scalebox{0.4}{$\bullet$}}\mathrm{dvol}_{\Gamma_s}(dx)ds  \nonumber\\
    & & +\int_0^t\int_{\Gamma_s} \big(f_{\epsilon}((s,x), u^\epsilon_{1}(s,x)) - f_{\epsilon}((s,x), u^\epsilon_{2}(s,x))\big) \cdot \nabla_{\Gamma_s}(u^\epsilon_{1}(s,x) - u^\epsilon_{2}(s,x))\mathrm{dvol}_{\Gamma_s}(dx)ds  \nonumber\\
&\leq& \frac{1}{2}\int_{\Gamma_t} |u^\epsilon_{1}(t,x) - u^\epsilon_{2}(t,x)|^2 \mathrm{dvol}_{\Gamma_t}(dx)
    - \frac{1}{2}\int_0^t\int_{\Gamma_s}|u^\epsilon_{1}(s,x) - u^\epsilon_{2}(s,x)|^2\nabla_{\Gamma_s}\cdot v(s,x) \mathrm{dvol}_{\Gamma_s}(dx)ds  \nonumber\\
    & &+ C_\epsilon\int_0^t\int_{\Gamma_s}|u^\epsilon_{1}(s,x) - u^\epsilon_{2}(s,x)|\cdot|\nabla_{\Gamma_s}(u^\epsilon_{1}(s,x) - u^\epsilon_{2}(s,x))|\mathrm{dvol}_{\Gamma_s}(dx)ds \nonumber\\
&\leq&  \frac{1}{2}\int_{\Gamma_t} |u^\epsilon_{1}(t,x) - u^\epsilon_{2}(t,x)|^2 \mathrm{dvol}_{\Gamma_t}(dx)
        + C_\epsilon \int_0^t\int_{\Gamma_s}|u^\epsilon_{1}(s,x) - u^\epsilon_{2}(s,x)|^2\mathrm{dvol}_{\Gamma_s}(dx)ds\nonumber\\
    & &+\frac{\epsilon}{2}\int_0^t\int_{\Gamma_s}|\nabla_{\Gamma_s}(u^\epsilon_{1}(s,x) - u^\epsilon_{2}(s,x))|^2\mathrm{dvol}_{\Gamma_s}(dx)ds,
\end{eqnarray*}
where we have used Lemma \ref{lemma2.1} in the second inequality.

Therefore, Gronwall's lemma yields
\begin{equation*}\label{eq79}
\sup_{t \in [0,T]}\int_{\Gamma_t} |u^\epsilon_{1}(t,x) - u^\epsilon_{2}(t,x)|^2 \mathrm{dvol}_{\Gamma_t}(dx) = 0,
\end{equation*}
which implies the uniqueness of equation \eqref{eq76}. The proof is complete.
\end{proof}

\vskip 0.2cm
Before the end of this section, we give a moment estimate for $\{u^\epsilon \}_{\epsilon \in (0,1)}$.
\begin{lemma}\label{2.3}
For any $\epsilon \in (0,1)$, let $u^\epsilon$ be the solution of equation \eqref{eq75}. Then $u^\epsilon$ satisfies the following moment estimate:
\begin{equation*}\label{eq80}
\mathbb{E}[\sup_{t \in [0,T]}\| u^\epsilon(t)\|^2_{L^2(\Gamma_t)}] + \epsilon\mathbb{E}[\int_0^T\int_{\Gamma_t}|\nabla_{\Gamma_t}u^\epsilon(t,x)|^2\mathrm{dvol}_{\Gamma_t}(dx)dt]
\leq C,
\end{equation*}
where the positive constant $C$ does not depend on $\epsilon$.
\end{lemma}

\begin{proof}
By It\^{o}'s formula on moving hypersurfaces (see Lemma \ref{lemma2.2}), we obtain that for every $t \in [0,T]$,
\begin{eqnarray}\label{eq81}
  & & \|u^\epsilon(t) \|^2_{L^2(\Gamma_t)} + 2\epsilon\int_0^t\int_{\Gamma_s}|\nabla_{\Gamma_s}u^\epsilon(s,x)|^2\mathrm{dvol}_{\Gamma_s}(dx)ds\nonumber\\
&=& \| u^\epsilon_0\|^2_{L^2(\Gamma_0)} - \int_0^t\int_{\Gamma_s} |u^\epsilon(s,x)|^2\nabla_{\Gamma_s}\cdot v(s,x) \mathrm{dvol}_{\Gamma_s}(dx)ds
   + 2\int_0^t\int_{\Gamma_s} u^\epsilon(s,x)\sigma(s,x)\mathrm{dvol}_{\Gamma_s}(dx)dB_s \nonumber\\
   & & + \int_0^t\int_{\Gamma_s} \sigma^2(s,x)\mathrm{dvol}_{\Gamma_s}(dx)ds
       + 2\int_0^t\int_{\Gamma_s} \nabla_{\Gamma_s}u^\epsilon(s,x) \cdot f_\epsilon((s,x), u^\epsilon(s,x))\mathrm{dvol}_{\Gamma_s}(dx)ds.
\end{eqnarray}

Notice that
\begin{eqnarray*}\label{eq82}
   & &  \nabla_{\Gamma_t} \cdot \big( \int_0^{u^\epsilon(t,x)} f_\epsilon((t,x),\tau)d\tau \big) \nonumber\\
&=& f_\epsilon((t,x),u^\epsilon(t,x)) \cdot \nabla_{\Gamma_t}u^\epsilon(t,x) + \int_0^{u^\epsilon(t,x)} \nabla_{\Gamma_t} \cdot \big(f_\epsilon((t,x),\tau)\big)d\tau\nonumber\\
&=& f_\epsilon((t,x),u^\epsilon(t,x)) \cdot \nabla_{\Gamma_t}u^\epsilon(t,x).
\end{eqnarray*}
Thus, by Lemma \ref{lemma2.3} and Assumption \ref{Assumption1}(B1), we derive that
\begin{equation}\label{eq83}
\int_0^t\int_{\Gamma_s} \nabla_{\Gamma_s}u^\epsilon(s,x)\cdot f_\epsilon((s,x), u^\epsilon(s,x))\mathrm{dvol}_{\Gamma_s}(dx)ds=0.
\end{equation}

In combination with \eqref{eq81} and \eqref{eq83}, and using the Burkholder--Davis--Gundy inequality, H\"older inequality and Young's inequality, we further obtain
\begin{eqnarray*}\label{eq84}
& & \mathbb{E}[\sup_{t \in [0,T]}\|u^\epsilon(t) \|^2_{L^2(\Gamma_t)}] + 2\epsilon\mathbb{E}[\int_0^T\int_{\Gamma_s}|\nabla_{\Gamma_s}u^\epsilon(s,x)|^2\mathrm{dvol}_{\Gamma_s}(dx)ds]\nonumber\\
&\leq& \| u^\epsilon_0\|^2_{L^2(\Gamma_0)}  + C\mathbb{E}[\int_0^T\|u^\epsilon(s)\|^2_{L^2(\Gamma_s)}ds]
   + C\mathbb{E}[(\int_0^T|\int_{\Gamma_s} u^\epsilon(s,x)\sigma(s,x)\mathrm{dvol}_{\Gamma_s}(dx)|^2ds)^\frac{1}{2}]
   +C  \nonumber\\
&\leq &  \| u^\epsilon_0\|^2_{L^2(\Gamma_0)}+  C\int_0^T\mathbb{E}[\sup_{r\in [0,s]}\|u^\epsilon(r)\|^2_{L^2(\Gamma_r)}]ds + C. \nonumber\\
\end{eqnarray*}
Then by Gronwall's lemma and \eqref{eq58}, we obtain that there exists a positive constant $C$ independent of $\epsilon$ such that
\begin{eqnarray*}\label{eq85}
\mathbb{E}[\sup_{t \in [0,T]}\|u^\epsilon(t) \|^2_{L^2(\Gamma_t)}] + 2\epsilon\mathbb{E}[\int_0^T\int_{\Gamma_s}|\nabla_{\Gamma_s}u^\epsilon(s,x)|^2\mathrm{dvol}_{\Gamma_s}(dx)ds]
\leq C.
\end{eqnarray*}

\end{proof}

\section{A priori estimates for the viscous problem}\label{A priori estimates for the regularized problem}
For any $\epsilon \in (0,1)$, let $u^\epsilon$ be the solution of equation \eqref{eq75}. In this section, we will further derive a priori estimates for $u^\epsilon$ that are independent of $\epsilon$. These estimates are essential for proving the existence of solutions to the stochastic conservation law \eqref{eq74}. Recall that $f_\epsilon$ is defined by \eqref{buchong11}. By Assumption \ref{Assumption1}(B1), $f_\epsilon$ has the local representation
\begin{equation}\label{buchong7}
f_\epsilon((t,x),u) =a_{i}^{\epsilon, \alpha}((t,x),u)\frac{\partial X^\alpha_t}{\partial \theta_i}(X^{\alpha}_t(x)^{-1}), \quad ((t,x),u) \in \Big(\bigcup_{t \in [0,T]} \{t\}\times V_t^{\alpha}\Big) \times \mathbb{R},
\end{equation}
where $a_{i}^{\epsilon, \alpha}: \Big(\bigcup_{t \in [0,T]} \{t\}\times V_t^{\alpha}\Big) \times \mathbb{R} \rightarrow \mathbb{R} $, $1 \leq \alpha \leq N$,
and by Assumption \ref{Assumption1}(B3), there exists a positive constant $L_1$ independent of $\epsilon$ such that for all $((t,x), u) \in Q_T \times \mathbb{R}$,
\begin{equation}\label{eq171}
 |f_{\epsilon u}((t,x),u)| \leq L_1(1 + |u|),
\end{equation}
and
\begin{equation}\label{eq176}
|\nabla_{\Gamma_t}f_{\epsilon u}((t,x),u)| \vee |f_{\epsilon uu}((t,x),u)| \vee |\nabla_{\Gamma_t}^2f_{\epsilon}((t,x),u)| \leq L_1(1 + |u|^{q_0}).
\end{equation}

\subsection{Estimate of the solution}
In this subsection, we prove that the solution $u^{\epsilon}$ to the stochastic viscous
conservation law \eqref{eq75} is bounded in the $L^{\infty}$-norm in both space and time independently of the parameter $\epsilon$. Recall that $w(t,x) = \int_0^t \sigma(s, G(s, G^{-1}(t,x))) dB_s$, $(t,x) \in Q_T$.

\begin{lemma}\label{lemma3.1}
Let $u^\epsilon$ be the solution of \eqref{eq75}. Then
\begin{equation*}\label{eq1}
\sup_{t \in [0,T]}\|u^{\epsilon}(t)\|_{L^{\infty}(\Gamma_t)} \leq e^{C_3T}C_4,
\end{equation*}
where $C_3=\sup_{t\in[0,T]}\|\nabla_{\Gamma_t} \cdot v(t)\|_{L^{\infty}(\Gamma_t)} + (L_1+1)(\sup_{t\in[0,T]}\|w(t)\|_{L^{\infty}(\Gamma_t)} + \sup_{t\in[0,T]}\|\nabla_{\Gamma_t}w(t) \|_{L^{\infty}(\Gamma_t)} +1)$ and $L_1$ is the constant from \eqref{eq171},
and $C_4=\sup_{\epsilon \in (0,1)}\|u_0^\epsilon\|_{L^\infty(\Gamma_0)}+ sup_{t\in[0,T]}\|\nabla_{\Gamma_t} \cdot v(t) \|_{L^{\infty}(\Gamma_t)}+ \sup_{t\in[0,T]}\|w(t)\|_{L^{\infty}(\Gamma_t)}+ \sup_{t\in[0,T]}\|\nabla_{\Gamma_t}w(t) \|_{L^{\infty}(\Gamma_t)} + \sup_{t\in[0,T]}\|\Delta_{\Gamma_t}w(t)\|_{L^{\infty}(\Gamma_t)} $.
\end{lemma}

\begin{proof}
By Proposition \ref{proposition2}, $u^\epsilon = V^\epsilon +w$, where $V^\epsilon$ is the solution to equation \eqref{eq77}. Therefore, it is sufficient to estimate the $L^\infty$-norm of $V^\epsilon$. We begin by setting
\[ \overline{V^\epsilon}(t,x)= e^{-\lambda t}V^{\epsilon}(t,x), \quad (t,x)\in Q_T, \]
where $\lambda=\sup_{t\in[0,T]}\|\nabla_{\Gamma_t} \cdot v(t) \|_{L^{\infty}(\Gamma_t)} + (L_1+1)(\sup_{t\in[0,T]}\|w(t)\|_{L^{\infty}(\Gamma_t)} + \sup_{t\in[0,T]}\|\nabla_{\Gamma_t}w(t) \|_{L^{\infty}(\Gamma_t)} +1)$, and $L_1$ is the constant appearing in \eqref{eq171}. Set $g_\epsilon((t,x),u)= e^{-\lambda t}f_{\epsilon}((t, x), e^{\lambda t}u + w(t,x))$, $(t,x) \in Q_T$, $u\in \mathbb{R}$. Then for every $\varphi \in H^1({Q_T})$, by equation \eqref{eq77} we have
\begin{eqnarray}\label{eq2}
\int_{\Gamma_t}\overset{\scalebox{0.4}{$\bullet$}}{\overline{V^\epsilon}}(t,x)\varphi(t,x)\mathrm{dvol}_{\Gamma_t}(dx)
&=& - \int_{\Gamma_t}(\lambda + \nabla_{\Gamma_t} \cdot v(t,x))\overline{V^\epsilon}(t,x)\varphi(t,x)\mathrm{dvol}_{\Gamma_t}(dx) \nonumber\\
& &  -\int_{\Gamma_t}e^{-\lambda t}w(t,x)\nabla_{\Gamma_t} \cdot v(t,x)\varphi(t,x)\mathrm{dvol}_{\Gamma_t}(dx)\nonumber\\
& & -\int_{\Gamma_t}\nabla_{\Gamma_t} \cdot \big(g_{\epsilon}((t,x),\overline{V^\epsilon}(t))\big)\varphi(t,x)\mathrm{dvol}_{\Gamma_t}(dx)\nonumber\\
& & +\epsilon\int_{\Gamma_t}\Delta_{\Gamma_t}\overline{V^\epsilon}(t,x)\varphi(t,x)\mathrm{dvol}_{\Gamma_t}(dx)\nonumber\\
& &  +\epsilon\int_{\Gamma_t}e^{-\lambda t}\Delta_{\Gamma_t}w(t,x)\varphi(t,x)\mathrm{dvol}_{\Gamma_t}(dx).
\end{eqnarray}
If we choose $\varphi=(\overline{V^\epsilon}-M^*)_{+}=\max\{\overline{V^\epsilon}-M^*, 0 \}$ with $M^*= \sup_{\epsilon \in (0,1)}\|u_0^\epsilon\|_{L^\infty(\Gamma_0)}+ \sup_{t\in[0,T]}\|\nabla_{\Gamma_t} \cdot v(t) \|_{L^{\infty}(\Gamma_t)}+ \sup_{t\in[0,T]}\|\nabla_{\Gamma_t}w(t) \|_{L^{\infty}(\Gamma_t)} + \sup_{t\in[0,T]}\|\Delta_{\Gamma_t}w(t)\|_{L^{\infty}(\Gamma_t)} $ in \eqref{eq2}, then by Lemma \ref{lemma2.3} and the fact that $g_\epsilon((t,x),u)$ is a tangent vector at the surface $\Gamma_t$ for $t \in [0,T]$, $x \in \Gamma_t$ and $u \in \mathbb{R}$, we arrive at
\begin{eqnarray*}\label{eq3}
& &\frac{1}{2}\int_{\Gamma_t}\big((\overline{V^\epsilon}-M^*)_{+}^2\big)^{\scalebox{0.4}{$\bullet$}}(t,x)\mathrm{dvol}_{\Gamma_t}(dx)
     + \epsilon \int_{\Gamma_t}|\nabla_{\Gamma_t}(\overline{V^{\epsilon}}-M^*)_{+}(t,x)|^2\mathrm{dvol}_{\Gamma_t}(dx)\nonumber\\
&=&- \int_{\Gamma_t}(\lambda + \nabla_{\Gamma_t} \cdot v(t,x))\overline{V^\epsilon}(t,x)(\overline{V^\epsilon}-M^*)_{+}(t,x)\mathrm{dvol}_{\Gamma_t}(dx)\nonumber\\
& &+\int_{\Gamma_t}e^{-\lambda t}(\epsilon \Delta_{\Gamma_t}w(t,x) - w(t,x)\nabla_{\Gamma_t}
    \cdot v(t,x))(\overline{V^\epsilon}-M^*)_{+}(t,x)\mathrm{dvol}_{\Gamma_t}(dx)\nonumber\\
&&+\int_{\Gamma_t}(g_{\epsilon}((t,x),\overline{V^\epsilon}(t,x))-g_{\epsilon}((t,x),M^*))\nabla_{\Gamma_t}(\overline{V^{\epsilon}}-M^*)_{+}(t,x)\mathrm{dvol}_{\Gamma_t}(dx)\nonumber\\
& &+\int_{\Gamma_t}g_{\epsilon}((t,x),M^*)\nabla_{\Gamma_t}(\overline{V^{\epsilon}}-M^*)_{+}(t,x)\mathrm{dvol}_{\Gamma_t}(dx).\nonumber\\
\end{eqnarray*}
For the right-hand side of the above equation, we observe that for our choice of $M^*$
\begin{eqnarray*}\label{eq7}
e^{-\lambda t}|\epsilon \Delta_{\Gamma_t}w(t,x) -w(t,x)\nabla_{\Gamma_t} \cdot v(t,x)| \leq (1 + \|w(t)\|_{L^{\infty}(\Gamma_t)})M^*,
\end{eqnarray*}
\begin{eqnarray*}\label{eq4}
   & &  |g_{\epsilon}((t,x),\overline{V^\epsilon}(t,x))-g_{\epsilon}((t,x),M^*)|\nonumber\\
&=&e^{-\lambda t}|f_{\epsilon}((t,x),e^{\lambda t}\overline{V^\epsilon}(t,x)+w(t,x)) - f_{\epsilon}((t,x),e^{\lambda t}M^*+w(t,x))|\nonumber\\
&\leq& C_{\epsilon}|\overline{V^{\epsilon}}(t,x) - M^*|,
\end{eqnarray*}
and by \eqref{eq171}
\begin{eqnarray*}\label{eq5}
|\nabla_{\Gamma_t} \cdot \big(g_{\epsilon}((t,x),M^*))\big)|
&=&e^{-\lambda t}|f_{\epsilon u}((t, x), e^{\lambda t}M^* + w(t,x)) \cdot \nabla_{\Gamma_t}w(t,x)|\nonumber\\
&\leq&  L_1e^{-\lambda t}(1+|e^{\lambda t}M^* + w(t,x)|)|\nabla_{\Gamma_t}w(t,x)|\nonumber\\
& \leq& L_1(1+ M^* +\|w(t)\|_{L^{\infty}(\Gamma_t)}  )\|\nabla_{\Gamma_t}w(t)\|_{L^{\infty}(\Gamma_t)} \nonumber\\
&\leq&  L_1(1+\|\nabla_{\Gamma_t}w\|_{L^{\infty}(\Gamma_t)}+\|w\|_{L^{\infty}(\Gamma_t)})M^*. \nonumber\\
\end{eqnarray*}
Combining these three estimates with the transport formulae from Lemma \ref{lemma2.1}, we obtain
\begin{eqnarray*}\label{eq6}
& & \frac{1}{2}\frac{d}{dt}\int_{\Gamma_t}(\overline{V^\epsilon}-M^*)_{+}^2(t,x)\mathrm{dvol}_{\Gamma_t}(dx)
          +\epsilon\int_{\Gamma_t}| \nabla_{\Gamma_t}(\overline{V^{\epsilon}}-M^*)_{+}(t,x) |^2\mathrm{dvol}_{\Gamma_t}(dx)\nonumber\\
&\leq & \frac{1}{2}\|\nabla_{\Gamma_t}\cdot v(t) \|_{L^{\infty}(\Gamma_t)}\int_{\Gamma_t}(\overline{V^\epsilon}-M^*)_{+}^2(t,x)\mathrm{dvol}_{\Gamma_t}(dx) \nonumber\\
     & &   - \int_{\Gamma_t}(\lambda + \nabla_{\Gamma_t} \cdot v(t,x))\overline{V^\epsilon}(t,x)(\overline{V^\epsilon}-M^*)_{+}(t,x)\mathrm{dvol}_{\Gamma_t}(dx)\nonumber\\
&&+(L_1+1)\int_{\Gamma_t}(1+\|\nabla_{\Gamma_t}w(t)\|_{L^{\infty}(\Gamma_t)}+\|w(t)\|_{L^{\infty}(\Gamma_t)})M^*(\overline{V^\epsilon}-M^*)_{+}(t,x)\mathrm{dvol}_{\Gamma_t}(dx)\nonumber\\
     & &+C_{\epsilon}\int_{\Gamma_t}|\overline{V^\epsilon}(t,x)-M^*|\cdot|\nabla_{\Gamma_t}(\overline{V^{\epsilon}}-M^*)_{+}(t,x)|\mathrm{dvol}_{\Gamma_t}(dx).
\end{eqnarray*}
Noticing that $M^*(\overline{V^\epsilon}-M^*)_{+} \leq \overline{V^\epsilon}(\overline{V^\epsilon}-M^*)_{+}$, by Young's inequality it follows that
\begin{eqnarray*}\label{eq8}
& & \frac{1}{2}\frac{d}{dt}\int_{\Gamma_t}(\overline{V^\epsilon}-M^*)_{+}^2(t,x)\mathrm{dvol}_{\Gamma_t}(dx)
          +\frac{\epsilon}{2}\int_{\Gamma_t}|\nabla_{\Gamma_t}(\overline{V^{\epsilon}}-M^*)_{+}(t,x)|^2\mathrm{dvol}_{\Gamma_t}(dx)\nonumber\\
&\leq& ( \frac{1}{2}\|\nabla_{\Gamma_t}\cdot v(t) \|_{L^{\infty}(\Gamma_t)}+\frac{C_{\epsilon}^2}{2\epsilon}) \int_{\Gamma_t}(\overline{V^\epsilon}-M^*)_{+}^2(t,x)\mathrm{dvol}_{\Gamma_t}(dx)\nonumber\\
       & & +\int_{\Gamma_t}\{(L_1+1)(1+\|\nabla_{\Gamma_t}w(t)\|_{L^{\infty}(\Gamma_t)}+\|w(t)\|_{L^{\infty}(\Gamma_t)})-\lambda - \nabla_{\Gamma_t} \cdot v(t,x)\}\overline{V^\epsilon}(t,x)\nonumber\\
      & & \cdot(\overline{V^\epsilon}-M^*)_{+}(t,x)\mathrm{dvol}_{\Gamma_t}(dx)\nonumber\\
&\leq& ( \frac{1}{2}\|\nabla_{\Gamma_t}\cdot v(t) \|_{L^{\infty}(\Gamma_t)}+\frac{C_{\epsilon}^2}{2\epsilon}) \int_{\Gamma_t}(\overline{V^\epsilon}-M^*)_{+}^2(t,x)\mathrm{dvol}_{\Gamma_t}(dx),
\end{eqnarray*}
here we have used the fact that $(L_1+1)(1+\|\nabla_{\Gamma_t}w(t)\|_{L^{\infty}(\Gamma_t)}+\|w(t)\|_{L^{\infty}(\Gamma_t)})-\lambda - \nabla_{\Gamma_t} \cdot v \leq 0$ in the last step.
Since $\int_{\Gamma_0}(\overline{V^\epsilon}(0,y)-M^*)_{+}^2\mathrm{dvol}_{\Gamma_0}(dy)=0$, we then obtain by the Gronwall's lemma that
\[ \int_{\Gamma_t}(\overline{V^\epsilon}-M^*)_{+}^2(t,x)\mathrm{dvol}_{\Gamma_t}(dx)=0,  \]
which implies that $\overline{V^\epsilon} \leq M$ or
\[ V^\epsilon(t) \leq e^{\lambda T}M^*   \ \mbox{on}\ \Gamma_t.  \]
The estimate from below follows similarly. Hence, we conclude that
\begin{equation*}\label{eq175}
\sup_{t \in [0,T]}\| V^\epsilon(t)\|_{L^{\infty}(\Gamma_t)} \leq e^{\lambda T}M^*,
\end{equation*}
which yields that
\[ \sup_{t \in [0,T]}\| u^\epsilon(t)\|_{L^{\infty}(\Gamma_t)}
\leq e^{\lambda T}M^*+ \sup_{t \in [0,T]}\| w(t) \|_{L^{\infty}(\Gamma_t)},  \]
completing the proof.
\end{proof}


\subsection{Estimate of the spatial gradient}
In this subsection, we establish the uniform $L^1$-estimate for $\nabla_{\Gamma_{\cdot}}u^{\epsilon}$.
\begin{lemma}\label{lemma3.2}
Assume that $u^\epsilon$ solves the stochastic viscous conservation law \eqref{eq75}. For every $M >0$, let
\begin{eqnarray}\label{eq50}
\Omega_{M} &=& \big \{ \omega \in \Omega: \sup_{t\in[0,T]}\|w(t)\|_{L^{\infty}(\Gamma_t)} + \sup_{t\in[0,T]}\|\nabla_{\Gamma_t}w(t)\|_{L^{\infty}(\Gamma_t)}
                       + \sup_{t\in[0,T]}\|\nabla_{\Gamma_t}^2w(t)\|_{L^{\infty}(\Gamma_t)} \nonumber\\
                & &\ + \sup_{t\in[0,T]}\|\nabla_{\Gamma_t}\triangle_{\Gamma_t}w(t)\|_{L^{\infty}(\Gamma_t)} \leq M  \big\}.
\end{eqnarray}
Then  $\mathbb{P}$-a.s.
\[ I_{\Omega_{M}}\sup_{t\in[0,T]}\int_{\Gamma_t}|\nabla_{\Gamma_t}u^{\epsilon}(t,x)|\mathrm{dvol}_{\Gamma_t}(dx) \leq C_{M,T,C_0,L_1}.  \]
\end{lemma}
\begin{proof}
Since $u^\epsilon = V^\epsilon +w$, where $V^\epsilon$ is the solution to equation \eqref{eq77}, it suffices to prove that $\mathbb{P}$-a.s.
\[ I_{\Omega_{M}}\sup_{t\in[0,T]}\int_{\Gamma_t}|\nabla_{\Gamma_t}V^{\epsilon}(t,x)|\mathrm{dvol}_{\Gamma_t}(dx) \leq C_{M,T,C_0,L_1}.  \]
Set $h_i=\underline{D}_iV^{\epsilon}$ and $h=(h_1,\cdots, h_{n+1})$. Then for each $i=1,\cdots, n+1$, $h_i \in L^2(\Omega; C([0,T]; L^2(\Gamma_{\cdot}))) \cap L^2(\Omega; L^2([0,T]; H^1(\Gamma_{\cdot}))) $ and satisfies the following equation in $H^{-1}(\Gamma_{\cdot})$,
\begin{eqnarray*}\label{eq9}
&  & \overset{\scalebox{0.4}{$\bullet$}}{h_i} + A_{ir}(v)h_r + h_i\nabla_{\Gamma} \cdot v + \underline{D}_iw\nabla_{\Gamma} \cdot v + (V^\epsilon+w)\underline{D}_i\nabla_{\Gamma} \cdot v
     + \underline{D}_i\nabla_{\Gamma}\cdot \big(f_{\epsilon}(\cdot,V^{\epsilon}+w)\big) \nonumber\\
& & - \epsilon\underline{D}_i\underline{D}_kh_k - \epsilon\underline{D}_i\Delta_{\Gamma}w = 0,
\end{eqnarray*}
here we have used Lemma \ref{lemma2.5}.

For each $\delta > 0$, choose a function $\eta_\delta(\cdot):\mathbb{R}\rightarrow\mathbb{R}^+$ such that $\eta_\delta(0)=0$ and
\begin{numcases}{\eta'_\delta(r)=}
	1,\ \text{if $r>2\delta$,}\nonumber\\
	\nonumber\frac{1+\mathrm{sin}\big((\frac{\pi}{2\delta})(2r-3\delta)\big)}{2},\ \text{if $\delta\leq r\leq2\delta$},\\ 0,\ \text{if $r<\delta$}.\nonumber\
\end{numcases}
It is easy to see that $\eta_\delta \in C^{2}(\mathbb{R})$ and $|\eta_\delta(r) - r_{+}| \leq 2\delta$. Denote $\phi_\delta(\cdot) = \eta_\delta(\cdot) + \eta_\delta(-\cdot)$ and let $S_\delta:\mathbb{R}^{n+1} \rightarrow \mathbb{R}$ such that $S_\delta(\cdot) = \phi_\delta(|\cdot|)$. Then $S_\delta \in C^{2}(\mathbb{R}^{n+1})$ is an approximation function for the abosulte value function $|\cdot|$ and we have
\begin{align}
	                               |\bar{x}| - 2\delta &\leq S_\delta(\bar{x}) \leq |\bar{x}|,\label{eq11}\\
	                       \partial_{\bar{x}_i}S_\delta(\bar{x}) &= \phi_\delta'(|\bar{x}|)\frac{\bar{x}_i}{|\bar{x}|},\label{eq12}\\
	                  \partial^2_{\bar{x}_i\bar{x}_j}S_\delta(\bar{x}) &= \phi_\delta''(|\bar{x}|)\frac{\bar{x}_i\bar{x}_j}{|\bar{x}|^2}+\phi_\delta'(|\bar{x}|)\frac{\delta_{ij}|\bar{x}|^2-\bar{x}_i\bar{x}_j}{|\bar{x}|^3}.\nonumber\
\end{align}
Applying the It\^o's formula in Lemma \ref{lemma2.2} to $h$, we obtain
\begin{eqnarray}\label{eq41}
&  & \int_{\Gamma_t} S_\delta(h(t,x))\mathrm{dvol}_{\Gamma_t}(dx)-\int_{\Gamma_0} S_\delta(h(0,y))\mathrm{dvol}_{\Gamma_0}(dy)\nonumber\\
&=&\int_0^t\int_{\Gamma_s}\big(S_\delta(h(s,x)) - h_i(s,x)\partial_{\bar{x}_i}S_\delta(h(s,x))\big)\nabla_{\Gamma_s}\cdot v(s,x)\mathrm{dvol}_{\Gamma_s}(dx)ds \nonumber\\
    & & -\int_{0}^{t}\int_{\Gamma_s}A_{ir}(v)(s,x)h_r(s,x)\partial_{\bar{x}_i}S_\delta(h(s,x))\mathrm{dvol}_{\Gamma_s}(dx)ds\nonumber\\
   & & -\int_{0}^{t}\int_{\Gamma_s}\underline{D}_iw(s,x)\nabla_{\Gamma_s}\cdot v(s,x)\partial_{\bar{x}_i}S_\delta(h(s,x))\mathrm{dvol}_{\Gamma_s}(dx)ds \nonumber\\
   & & -\int_{0}^{t}\int_{\Gamma_s}(V^{\epsilon}(s,x)+w(s,x))\underline{D}_i\nabla_{\Gamma_s}\cdot v(s,x)\partial_{\bar{x}_i}S_\delta(h(s))\mathrm{dvol}_{\Gamma_s}ds\nonumber\\
   & &-\int_{0}^{t}\int_{\Gamma_s}\underline{D}_i\nabla_{\Gamma_s}\cdot \big(f_{\epsilon}((s,x),V^{\epsilon}(s,x)+w(s,x))\big)\partial_{\bar{x}_i}S_\delta(h(s,x))\mathrm{dvol}_{\Gamma_s}(dx)ds \nonumber\\
   & & +\epsilon\int_{0}^{t}\int_{\Gamma_s}\underline{D}_i\underline{D}_kh_k(s,x)\partial_{\bar{x}_i}S_\delta(h(s,x))\mathrm{dvol}_{\Gamma_s}(dx)ds\nonumber\\
   & &+\epsilon\int_{0}^{t}\int_{\Gamma_s}\underline{D}_i\Delta_{\Gamma_s}w(s,x)\partial_{\bar{x}_i}S_\delta(h(s,x))\mathrm{dvol}_{\Gamma_s}(dx)ds\nonumber\\
&=:&\sum_{i=1}^{7}I_i.
\end{eqnarray}

Next, we estimate each term separately. For the terms $I_1-I_4$ and $I_7$, by \eqref{eq11} and \eqref{eq12}, we have
\begin{equation*}\label{eq14}
|I_1| \leq C\int_{0}^{t}\int_{\Gamma_s}|h(s,x)|\mathrm{dvol}_{\Gamma_s}(dx)ds,
\end{equation*}
\begin{equation*}\label{eq15}
|I_2| \leq C\int_{0}^{t}\int_{\Gamma_s}|h(s,x)|\mathrm{dvol}_{\Gamma_s}(dx)ds,
\end{equation*}
\begin{equation*}\label{eq16}
|I_3|\leq C\int_{0}^{t}\int_{\Gamma_s}|\nabla_{\Gamma_s}w(s,x)|\mathrm{dvol}_{\Gamma_s}(dx)ds,
\end{equation*}
\begin{equation*}\label{eq17}
|I_4|\leq C\int_{0}^{t}\int_{\Gamma_s}|V^{\epsilon}(s,x) + w(s,x)|\mathrm{dvol}_{\Gamma_s}(dx)ds,
\end{equation*}
\begin{equation*}\label{eq18}
|I_7|\leq C\epsilon\int_{0}^{t}\int_{\Gamma_s}|\nabla_{\Gamma_s}\triangle_{\Gamma_s}w(s,x)|\mathrm{dvol}_{\Gamma_s}(dx)ds.
\end{equation*}

For $I_5$, by Lemma \ref{lemma2.6}, we have
\begin{eqnarray}\label{eq19}
   & & I_5 \nonumber\\
&=&-\int_{0}^{t}\int_{\Gamma_s}\underline{D}_i\big[\underline{D}_kf_{\epsilon k}((s,x), V^{\epsilon}(s,x)+w(s,x))
                      + f_{\epsilon ku}((s,x), V^{\epsilon}(s,x)+w(s,x))(h_k(s,x) \nonumber\\
                      & &+ \underline{D}_kw(s,x)) \big] \partial_{\bar{x}_i}S_\delta(h(s,x))\mathrm{dvol}_{\Gamma_s}(dx)ds\nonumber\\
&=&-\int_{0}^{t}\int_{\Gamma_s}\big[ \underline{D}_i\underline{D}_kf_{\epsilon k}((s,x), V^{\epsilon}(s,x)+w(s,x))
                               + \underline{D}_kf_{\epsilon ku}((s,x), V^{\epsilon}(s,x)+w(s,x))(h_i(s,x) \nonumber\\
                               &&+\underline{D}_iw(s,x) )+\underline{D}_if_{\epsilon ku}((s,x), V^{\epsilon}(s,x)+w(s,x))(h_k(s,x) + \underline{D}_kw(s,x))\nonumber\\
                               &&+f_{\epsilon kuu}((s,x), V^{\epsilon}(s,x)+w(s,x))(h_k(s,x) + \underline{D}_kw(s,x))(h_i(s,x) + \underline{D}_iw(s,x)) \nonumber\\
                               &&+ f_{\epsilon ku}((s,x), V^{\epsilon}(s,x)+w(s,x))(\underline{D}_kh_i(s,x) + \underline{D}_k\underline{D}_iw(s,x) +\mathcal{H}_{kl}(h_l(s,x)+\underline{D}_lw(s,x))\nu_i
                               \nonumber\\&&-\mathcal{H}_{il}(h_l(s,x)+\underline{D}_lw(s,x))\nu_k)\big]\partial_{x_i}S_\delta(h(s,x))\mathrm{dvol}_{\Gamma_s}(dx)ds.
\end{eqnarray}
We observe that
\begin{eqnarray*}\label{eq20}
& &\big[ \underline{D}_kf_{\epsilon ku}(\cdot, V^{\epsilon}+w)h_i+f_{\epsilon kuu}(\cdot, V^{\epsilon}+w)(h_k + \underline{D}_kw)h_i
       +f_{\epsilon ku}(\cdot, V^{\epsilon}+w)\underline{D}_kh_i \big]\partial_{\bar{x}_i}S_\delta(h)\nonumber\\
&=&\underline{D}_kf_{\epsilon ku}(\cdot, V^{\epsilon}+w)\phi_\delta'(|h|)|h|
 +f_{\epsilon kuu}(\cdot, V^{\epsilon}+w)(h_k + \underline{D}_kw)\phi_\delta'(|h|)|h|\nonumber\\
& &+f_{\epsilon ku}(\cdot, V^{\epsilon}+w)\underline{D}_kh_i\phi_\delta'(|h|)\frac{h_i}{|h|}\nonumber\\
&=&\phi_\delta'(|h|)\underline{D}_k(f_{\epsilon ku}(\cdot, V^{\epsilon}+w)|h|).
\end{eqnarray*}
Combining the above equality with \eqref{eq19} and the facts that $h \cdot \nu =0$ and $f_{\epsilon u}(\cdot, V^{\epsilon}+w) \cdot \nu =0$, $I_5$ can be rewritten as follows:
\begin{eqnarray*}\label{eq21}
   & & I_5 \nonumber\\
&=&-\int_{0}^{t}\int_{\Gamma_s}(\underline{D}_i\underline{D}_kf_{\epsilon k})((s,x), V^{\epsilon}(s,x)+w(s,x))\partial_{\bar{x}_i}S_\delta(h(s,x))\mathrm{dvol}_{\Gamma_s}(dx)ds\nonumber\\
   & &-\int_{0}^{t}\int_{\Gamma_s}\underline{D}_kf_{\epsilon ku}((s,x), V^{\epsilon}(s,x)+w(s,x))\underline{D}_iw(s,x)\partial_{\bar{x}_i}S_\delta(h(s,x))\mathrm{dvol}_{\Gamma_s}(dx)ds\nonumber\\
   & &-\int_{0}^{t}\int_{\Gamma_s}\underline{D}_if_{\epsilon ku}((s,x), V^{\epsilon}(s,x)+w(s,x))
                                  (h_k(s,x) + \underline{D}_kw(s,x))\partial_{\bar{x}_i}S_\delta(h(s,x))\mathrm{dvol}_{\Gamma_s}(dx)ds\nonumber\\
   & &-\int_{0}^{t}\int_{\Gamma_s}\phi_\delta'(|h(s,x)|)\underline{D}_k(f_{\epsilon ku}((s,x), V^{\epsilon}(s,x)+w(s,x))|h(s,x)|)\mathrm{dvol}_{\Gamma_s}(dx)ds\nonumber\\
   & &-\int_{0}^{t}\int_{\Gamma_s}f_{\epsilon kuu}((s,x), V^{\epsilon}(s,x)+w(s,x))(h_k(s,x) + \underline{D}_kw(s,x))\underline{D}_iw(s,x) \nonumber\\
       & & \cdot\partial_{\bar{x}_i}S_\delta(h(s,x))\mathrm{dvol}_{\Gamma_s}(dx)ds\nonumber\\
   & &-\int_{0}^{t}\int_{\Gamma_s}f_{\epsilon ku}((s,x), V^{\epsilon}(s,x)+w(s,x))\underline{D}_k\underline{D}_iw(s,x)\partial_{\bar{x}_i}S_\delta(h(s,x))\mathrm{dvol}_{\Gamma_s}(dx)ds\nonumber\\
   &=:&\sum_{i=1}^{6}I^{i}_5.
\end{eqnarray*}
By \eqref{eq12}, $\partial_{\bar{x}_i}S_\delta(h(s,x))$ is uniformly bounded. Combining this with \eqref{eq176}, Lemma \ref{lemma2.3} and the fact that $f_{\epsilon u}((t,x),u)$ is a tangent vector at the surface $\Gamma_t$ for $t \in [0,T]$, $x \in \Gamma_t$ and $u \in \mathbb{R}$, we further obtain
\begin{equation*}\label{eq22}
|I^1_5|\leq C\int_{0}^{t}\int_{\Gamma_s} (1+|V^{\epsilon}(s,x)+w(s,x)|^{q_0})\mathrm{dvol}_{\Gamma_s}(dx)ds.
\end{equation*}
\begin{equation*}\label{eq23}
|I^2_5|\leq C\int_{0}^{t}\int_{\Gamma_s}(1+|V^{\epsilon}(s,x)+w(s,x)|^{q_0})|\nabla_{\Gamma_s} w(s,x)|\mathrm{dvol}_{\Gamma_s}(dx)ds.
\end{equation*}
\begin{equation*}\label{eq24}
|I^3_5|\leq C\int_{0}^{t}\int_{\Gamma_s}(1+|V^{\epsilon}(s,x)+w(s,x)|^{q_0})(|h(s,x)|+|\nabla_{\Gamma_s} w(s,x)|)\mathrm{dvol}_{\Gamma_s}(dx)ds.
\end{equation*}
\begin{eqnarray*}\label{eq25}
|I^4_5| &=& |\int_{0}^{t}\int_{\Gamma_s}\underline{D}_k\phi_\delta'(|h(s,x)|)f_{\epsilon ku}((s,x), V^{\epsilon}(s,x)+w(s,x))|h(s,x)|\mathrm{dvol}_{\Gamma_s}(dx)ds |\nonumber\\
         &=& |\int_{0}^{t}\int_{\Gamma_s}\phi_\delta''(|h(s,x)|)h_l(s,x)\underline{D}_kh_l(s,x)f_{\epsilon ku}((s,x), V^{\epsilon}(s,x)+w(s,x))\mathrm{dvol}_{\Gamma_s}(dx)ds |\nonumber\\
         &\leq& C\int_{0}^{t}\int_{\Gamma_s}\frac{1}{\delta}I_{\{\delta \leq |h(s,x)| \leq 2\delta\}}|h(s,x)||\nabla_{\Gamma_s}h(s,x)||f_{\epsilon u}((s,x), V^{\epsilon}(s,x)+w(s,x))|\mathrm{dvol}_{\Gamma_s}(dx)ds, \nonumber\\
\end{eqnarray*}
here we have used the fact that $|\phi_\delta''(r)| \leq \frac{C}{\delta}I_{\{\delta \leq |r| \leq 2\delta\}}$ in the last step. 
Obviously, by the dominated convergence theorem, we have
\begin{eqnarray*}\label{eq26}
|I^4_5| \rightarrow 0,\ \text{as} \ \delta \rightarrow 0.
\end{eqnarray*}
For $I^5_5$ and $I^6_5$, by \eqref{eq171}, \eqref{eq176} and \eqref{eq12}, we have
\begin{equation*}\label{eq27}
|I^5_5| \leq C\int_{0}^{t}\int_{\Gamma_s}(1+|V^{\epsilon}(s,x)+w(s,x)|^{q_0})(|h(s,x)|+|\nabla_{\Gamma_s} w(s,x)|)|\nabla_{\Gamma_s} w(s,x)|\mathrm{dvol}_{\Gamma_s}(dx)ds,
\end{equation*}
and
\begin{equation*}\label{eq28}
|I^6_5|\leq C\int_{0}^{t}\int_{\Gamma_s}(1+|V^{\epsilon}(s,x)+w(s,x)|^{q_0})|\nabla_{\Gamma_s}^{2} w(s,x)|\mathrm{dvol}_{\Gamma_s}(dx)ds.
\end{equation*}

Below, we turn to the estimate for $I_6$. By Lemma \ref{lemma2.6}, we have
\begin{eqnarray*}\label{eq29}
\underline{D}_i\underline{D}_kh_k &=& \underline{D}_k\underline{D}_ih_k + \mathcal{H}_{kl}\underline{D}_lh_k\nu_i - \mathcal{H}_{il}\underline{D}_lh_k\nu_k \nonumber\\
                                  &=& \underline{D}_k( \underline{D}_kh_i + \mathcal{H}_{kl}h_l\nu_i - \mathcal{H}_{il}h_l\nu_k ) + \mathcal{H}_{kl}\underline{D}_lh_k\nu_i - \mathcal{H}_{il}\underline{D}_lh_k\nu_k \nonumber\\
                                  &=& \underline{D}_k\underline{D}_kh_i + \underline{D}_k(\mathcal{H}_{kl}h_l)\nu_i + \mathcal{H}_{kl}\mathcal{H}_{ik}h_l
                                       - \mathcal{H}_{il}Hh_l + \mathcal{H}_{kl}\underline{D}_lh_k\nu_i + \mathcal{H}_{il}\mathcal{H}_{kl}h_k. \nonumber\\
\end{eqnarray*}
For the last step, we have used the facts that $\underline{D}_k( \mathcal{H}_{il}h_l)\nu_k=0$ and $\underline{D}_lh_k\nu_k = -\mathcal{H}_{kl}h_k$.
Thus, using \eqref{eq12} and the fact that $h \cdot \nu =0$, we get
\begin{eqnarray*}\label{eq30}
    & & I_6  \nonumber\\
&=& \epsilon\int_{0}^{t}\int_{\Gamma_s}(\triangle_{\Gamma_s}h_i(s,x) + \mathcal{H}_{kl}(s,x)\mathcal{H}_{ik}(s,x)h_l(s,x) - \mathcal{H}_{il}(s,x)H(s,x)h_l(s,x) \nonumber\\
     & & + \mathcal{H}_{il}(s,x)\mathcal{H}_{kl}(s,x)h_k(s,x) )\partial_{\bar{x}_i}S_\delta(h(s,x))\mathrm{dvol}_{\Gamma_s}(dx)ds\nonumber\\
&\leq& \epsilon\int_{0}^{t}\int_{\Gamma_s}\triangle_{\Gamma_s}h_i(s,x)\partial_{\bar{x}_i}S_\delta(h(s,x))\mathrm{dvol}_{\Gamma_s}(dx)ds
           + C\epsilon\int_{0}^{t}\int_{\Gamma_s}|h(s,x)|\mathrm{dvol}_{\Gamma_s}(dx)ds\nonumber\\
&=&   -\epsilon\int_{0}^{t}\int_{\Gamma_s}\underline{D}_kh_i(s,x)\underline{D}_kh_j(s,x)\partial_{\bar{x}_i\bar{x}_j}^{2}S_\delta(h(s,x))\mathrm{dvol}_{\Gamma_s}(dx)ds
           + C\epsilon\int_{0}^{t}\int_{\Gamma_s}|h(s,x)|\mathrm{dvol}_{\Gamma_s}(dx)ds.
\end{eqnarray*}
Due to the fact that $S_\delta:\mathbb{R}^{n+1}\rightarrow\mathbb{R}$ is convex, it follows that $\nabla^2 S_\delta$ is non-negative definite. Therefore, we have
\begin{eqnarray*}\label{eq31}
-\epsilon\int_{0}^{t}\int_{\Gamma_s}\underline{D}_kh_i(s,x)\underline{D}_kh_j(s,x)\partial_{\bar{x}_i\bar{x}_j}^{2}S_\delta(h(s,x))\mathrm{dvol}_{\Gamma_s}(dx)ds \leq 0,
\end{eqnarray*}
and
\begin{eqnarray*}\label{eq32}
I_6 \leq C\epsilon\int_{0}^{t}\int_{\Gamma_s}|h(s,x)|\mathrm{dvol}_{\Gamma_s}(dx)ds.
\end{eqnarray*}

Recall that $u^\epsilon = V^\epsilon +w$. Now, collecting the previous estimates and \eqref{eq11}, and then letting $\delta \rightarrow 0$ in \eqref{eq41}, we obtain $\mathbb{P}$-a.s.
\begin{eqnarray}\label{eq33}
 & &  \int_{\Gamma_t}|h(t,x)|\mathrm{dvol}_{\Gamma_t}(dx) -\int_{\Gamma_0}|h(0,y)|\mathrm{dvol}_{\Gamma_0}(dy) \nonumber\\
&\leq& C\int_0^t\int_{\Gamma_s}(1+|u^{\epsilon}(s,x)|^{q_0})(1+|\nabla_{\Gamma_s}w(s,x)|^2+|\nabla_{\Gamma_s}^2w(s,x)|
       +|\nabla_{\Gamma_s}\triangle_{\Gamma_s}w(s,x)|)\mathrm{dvol}_{\Gamma_s}(dx)ds \nonumber\\
     & &+ C\int_0^t\int_{\Gamma_s}(1+|u^{\epsilon}(s,x)|^{q_0})(1+|\nabla_{\Gamma_s}w(s,x)|)|h(s,x)|\mathrm{dvol}_{\Gamma_s}(dx)ds, \nonumber\\
\end{eqnarray}
for all $t \in [0,T]$.

For every $M >0$, let
\begin{eqnarray*}\label{eq177}
\Omega_{M} &=& \big \{ \omega \in \Omega: \sup_{t\in[0,T]}\|w(t)\|_{L^{\infty}(\Gamma_t)} + \sup_{t\in[0,T]}\|\nabla_{\Gamma_t}w(t)\|_{L^{\infty}(\Gamma_t)}
                       + \sup_{t\in[0,T]}\|\nabla_{\Gamma_t}^2w(t)\|_{L^{\infty}(\Gamma_t)} \nonumber\\
                & &\ + \sup_{t\in[0,T]}\|\nabla_{\Gamma_t}\triangle_{\Gamma_t}w(t)\|_{L^{\infty}(\Gamma_t)} \leq M  \big\}.
\end{eqnarray*}
Then by Lemma \ref{lemma3.1}, \eqref{eq58} and \eqref{eq33}, we have $\mathbb{P}$-a.s.
\begin{eqnarray*}\label{eq34}
I_{\Omega_{M}}\int_{\Gamma_t}|h(t,x)|\mathrm{dvol}_{\Gamma_t}(dx) \leq C_{M,T, C_0,L_1} + C_{M,T,C_0,L_1}I_{\Omega_{M}}\int_0^t\int_{\Gamma_s}|h(s,x)|\mathrm{dvol}_{\Gamma_s}(dx)ds.
\end{eqnarray*}
Thus, by Gronwall's lemma, we get $\mathbb{P}$-a.s.
\begin{eqnarray*}\label{eq35}
I_{\Omega_{M}}\sup_{t\in[0,T]}\int_{\Gamma_t}|h(t,x)|\mathrm{dvol}_{\Gamma_t}(dx) \leq C_{M,T,C_0,L_1}.
\end{eqnarray*}

The proof is complete.
\end{proof}

\subsection{Estimate of the uniform $L^1$-continuity in time}
Recall the pullback transformation and its associated local parametric representation $\widetilde{\cdot}$ and ${\underset{}{\widetilde{\cdot}}{}}^\alpha$ given by \eqref{chongxing4} and \eqref{chongxing5}, respectively. For any $M > 0$, recall that
\begin{eqnarray*}\label{eq178}
\Omega_{M} &=& \big \{ \omega \in \Omega: \sup_{t\in[0,T]}\|w(t)\|_{L^{\infty}(\Gamma_t)} + \sup_{t\in[0,T]}\|\nabla_{\Gamma_t}w(t)\|_{L^{\infty}(\Gamma_t)}
                       + \sup_{t\in[0,T]}\|\nabla_{\Gamma_t}^2w(t)\|_{L^{\infty}(\Gamma_t)} \nonumber\\
                & &\ + \sup_{t\in[0,T]}\|\nabla_{\Gamma_t}\triangle_{\Gamma_t}w(t)\|_{L^{\infty}(\Gamma_t)} \leq M  \big\}.
\end{eqnarray*}
In this subsection, we will establish the uniform $L^1$-continuity in time of $u^{\epsilon}$.
\begin{lemma}\label{lemma3.3}
Assume that $u^\epsilon$ solves the stochastic viscous conservation law \eqref{eq75}.
Then for any $M>0$ and small $\Delta t > 0$, there exist a constant $C_M>0$ independent of $\epsilon$ and $\Delta t$ such that
\begin{eqnarray*}\label{eq37}
\mathbb{E}[I_{\Omega_{M}}\int_{0}^{T-\Delta t}\|\widetilde{u}^\epsilon(t+\Delta t) - \widetilde{u}^\epsilon(t)\|_{L^1(\Gamma_0)}dt]
\leq C_M(\sqrt{\Delta t} +  \Delta t).
\end{eqnarray*}
\end{lemma}
\begin{proof}
By Lemma \ref{buchonglemma1} and \eqref{buchong7}, the localized process $\widetilde{u}^{\epsilon,\alpha}$ satisfies the following equation in the sense of distribution over $U^\alpha$:
\begin{eqnarray}\label{eq38}
d \big(\widetilde{u}^{\epsilon,\alpha}_t(\theta)\sqrt{g^\alpha_t(\theta)}\big) &=&\epsilon\frac{\partial}{\partial\theta_j}\Big(g_{t,\alpha}^{ij}(\theta)\sqrt{g^\alpha_t(\theta)}\frac{\partial \widetilde{u}^{\epsilon,\alpha}_t}{\partial \theta_i}(\theta)\Big)dt
    -\frac{\partial}{\partial \theta_i}\Big(\widetilde{a}^{\epsilon,\alpha}_i(t,\theta,\widetilde{u}^{\epsilon,\alpha}_t(\theta))\sqrt{g^\alpha_t(\theta)} \Big)dt\nonumber\\
& & +\widetilde{\sigma}^\alpha_t( \theta)\sqrt{g^\alpha_t(\theta)}dB_t,
\end{eqnarray}
where $\widetilde{a}^{\epsilon,\alpha}_i(t,\theta,\widetilde{u}^{\epsilon,\alpha}_t(\theta))= a^{\epsilon,\alpha}_i((t,X^{\alpha}_t(\theta)),\widetilde{u}^{\epsilon,\alpha}_t(\theta))$. 
Fix small $\Delta t > 0$. For each $t \in [0, T- \Delta t]$, we set $\widetilde{w}^{\epsilon,\alpha}(t,\cdot):=\widetilde{u}^{\varepsilon,\alpha}_{t+\Delta t}(\cdot)\sqrt{g^\alpha_{t+\Delta t}(\cdot)}-\widetilde{u}^{\epsilon,\alpha}_{t}(\cdot)\sqrt{g^\alpha_{t}(\cdot)}$. We also take some bounded domain $U^{\alpha}_1$ such that $supp(\chi^{\alpha}\circ X^{\alpha})\subset U_1^\alpha\subset\subset U^\alpha $ and for each $\delta > 0$, set
\[ U^{\alpha}_{1,\delta}:= \{ \theta \in U^{\alpha}_1: dist(\theta, \partial U^{\alpha}_1) \geq \delta   \}.  \]

Let $J:\mathbb{R}^n\rightarrow\mathbb{R}$ be the standard mollifier defined by
\begin{align}\label{eq39}
J(x):=\left\{
\begin{aligned}
 &C\mathrm{exp}\big(\frac{1}{|x|^2-1}\big), \quad \text{if} \ |x|<1,\\
 &0, \quad \text{if} \ |x|\geq 1,\\
\end{aligned}
\right.
\end{align}
where the constant $C>0$ is chosen so that $\int_{\mathbb{R}^n}J(x)dx=1$. For each $\delta > 0$, we define
\begin{equation*}\label{eq179}
 J_{\delta}(x):= \frac{1}{\delta^n}J(\frac{x}{\delta}), \quad x \in \mathbb{R}^n,
\end{equation*}
and
$$\varphi^\alpha_\delta(t,\beta):=\int_{U^\alpha_{1,\delta}} J_{\delta}(\beta-\theta)\mathrm{sgn}(\widetilde{w}^{\varepsilon,\alpha}(t,\theta))d\theta, \quad \beta \in U^{\alpha}.$$
It is easy to see that $\varphi_\delta^\alpha\in L^\infty(0,T;C_c^\infty(U^\alpha))$ and
\begin{eqnarray}\label{eq44}
\|\varphi_\delta^\alpha\|_{L^\infty(U^\alpha)}+\delta\|\nabla \varphi^\alpha_\delta\|_{L^\infty(U^\alpha)}\leq C,
\end{eqnarray}
uniformly in $t$, for some constant $C>0$ independent of $\delta > 0$.

By the equation \eqref{eq38}, it follows that
\begin{eqnarray}\label{eq40}
   \int_{U_1^\alpha}\widetilde{w}^{\epsilon,\alpha}(t,\beta)\varphi^\alpha_\delta(t,\beta)d\beta
&=&-\epsilon\int_{t}^{t+\Delta t}\int_{U^\alpha_1}g^{ij}_{s,\alpha}(\beta)\sqrt{g_s^\alpha(\beta)}
       \frac{\partial\widetilde{u}^{\epsilon,\alpha}_s}{\partial \beta_i}(\beta)\frac{\partial \varphi^\alpha_\delta}{\partial \beta_j}(t,\beta)d\beta ds\nonumber\\
  & &+\int_{t}^{t+\Delta t}\int_{U^\alpha_1}\widetilde{a}^{\epsilon,\alpha}_i(s,\beta,\widetilde{u}^{\epsilon,\alpha}_s(\beta))\sqrt{g_s^\alpha(\beta)}
        \frac{\partial \varphi^\alpha_\delta}{\partial \beta_i}(t,\beta)d\beta ds\nonumber\\
  & &+\int_{U_1^\alpha}\Big(\int_t^{t+\Delta t} \widetilde{\sigma}^{\alpha}_s(\beta)\sqrt{g_s^\alpha(\beta)} dB_s\Big)\varphi_\delta^\alpha(t,\beta)d\beta.
\end{eqnarray}
Integrating \eqref{eq40} in $t$ from $0$ to $T - \Delta t$ yields
\begin{eqnarray*}\label{eq46}
  &  & \int_{0}^{T-\Delta t}\int_{U_1^\alpha}|\widetilde{w}^{\epsilon,\alpha}(t,\beta)|d\beta dt\nonumber\\
 &=& \int_{0}^{T-\Delta t}\int_{U_1^\alpha}\widetilde{w}^{\epsilon,\alpha}(t,\beta)
                                            \big(\mathrm{sgn}(\widetilde{w}^{\epsilon,\alpha}(t,\beta))-\varphi_\delta^\alpha(t,\beta)\big)d\beta dt\nonumber\\
    & & +\int_{0}^{T-\Delta t}\int_{U_1^\alpha}\widetilde{w}^{\epsilon,\alpha}(t,\beta)\varphi_\delta^\alpha(t,\beta) d\beta dt\nonumber\\
 &=& \int_{0}^{T-\Delta t}\int_{U_1^\alpha}\widetilde{w}^{\epsilon,\alpha}(t,\beta)\big(\mathrm{sgn}(\widetilde{w}^{\epsilon,\alpha}(t,\beta))-\varphi_\delta^\alpha(t,\beta)\big)d\beta dt\nonumber\\
   & &-\epsilon\int_{0}^{T-\Delta t}\int_{t}^{t+\Delta t}\int_{U_1^\alpha}g^{ij}_{s,\alpha}(\beta)\sqrt{g_s^\alpha(\beta)}\frac{\partial \widetilde{u}^{\epsilon,\alpha}_s}{\partial \beta_j}(\beta)\frac{\partial \varphi^\alpha_\delta}{\partial \beta_i}(t,\beta)d\beta ds dt\nonumber\\
   & &+\int_{0}^{T-\Delta t}\int_{t}^{t+\Delta
   t}\int_{U_1^\alpha}\widetilde{a}^{\epsilon,\alpha}_i(s,\beta,\widetilde{u}^{\epsilon,\alpha}_s(\beta))\sqrt{g_s^\alpha(\beta)}\frac{\partial \varphi^\alpha_\delta}{\partial \beta_i}(t,\beta)d\beta dsdt\nonumber\\
   & &+\int_{0}^{T-\Delta t}\int_{U_1^\alpha}\Big(\int_t^{t+\Delta t}\widetilde{\sigma}^{\alpha}_s(\beta)\sqrt{g_s^\alpha(\beta)}  dB_s\Big)\varphi_\delta^\alpha(t,\beta)d\beta dt\nonumber\\
 &=:& \sum_{i=1}^4I_i.
\end{eqnarray*}
Thus, for every $M>0$, we have
\begin{eqnarray}\label{eq43}
    \mathbb{E}[I_{\Omega_{M}}\int_{0}^{T-\Delta t}\int_{U_1^\alpha}|\widetilde{w}^{\epsilon,\alpha}(t,\beta)|d\beta dt]
= \sum_{i=1}^4\mathbb{E}[I_{\Omega_{M}}I_i].
\end{eqnarray}

Now, we examine these terms separately. For the term $\mathbb{E}[I_{\Omega_{M}}I_2]$, by \eqref{eq44} and Lemma \ref{lemma3.2}, we have
\begin{eqnarray}\label{eq42}
       \big|\mathbb{E}[ I_{\Omega_{M}}I_2]\big|
&\leq& \frac{C_{T, U^\alpha_1} \epsilon \Delta t}{\delta}\mathbb{E}\big[ I_{\Omega_{M}}\sup_{s\in[0,T]}\int_{U^\alpha_1}|\nabla\widetilde{u}^{\epsilon,\alpha}_s(\beta)|d\beta \big]\nonumber\\
&\leq& C_{T, U^\alpha_1,M,C_0,L_1}\epsilon\frac{\Delta t}{\delta}.
\end{eqnarray}

Thanks to the polynomial growth of $f_\epsilon$, we can deduce that for any $(s,\beta,u)\in [0,T] \times U^\alpha_1 \times \mathbb{R}$,
\[|\widetilde{a}^{\epsilon,\alpha}(s,\beta,u)|\leq C_{T,U_1^\alpha}(1+|u|^{q_0}).\]
Combining the above inequality with \eqref{eq44} and Lemma \ref{lemma3.1}, and using an argument similar to that proving \eqref{eq42} shows that
\begin{eqnarray}\label{eq45}
   &  & \big| \mathbb{E}[I_{\Omega_{M}}I_3] \big|\nonumber\\
&\leq&  C_{T, U^\alpha_1}\Big(\mathbb{E}\big[I_{\Omega_{M}}\int_{0}^{T - \Delta t}\int_t^{t + \Delta t}\int_{U^\alpha_1} |\widetilde{a}^{\epsilon,\alpha}(s,\beta,\widetilde{u}^{\epsilon,\alpha}_s(\beta))|d\beta dsdt\big]\Big)
        \times \Big(\sup_{t\in[0,T]}\|\nabla \varphi_\delta^\alpha(t,\beta)\|_{L^\infty(U^\alpha)} \Big)\nonumber\\
&\leq& C_{T, U^\alpha_1,M,C_0,L_1}\frac{\Delta t}{\delta}.
\end{eqnarray}

Concerning $\mathbb{E}[I_{\Omega_{M}}I_4]$, by the Burkholder--Davis--Gundy inequality, we have
\begin{eqnarray}\label{eq47}
 \big|\mathbb{E}[ I_{\Omega_{M}} I_4 ]\big|
&\leq&  \mathbb{E}\big[\int_{0}^{T-\Delta t} \big|\int_t^{(t+ \Delta t)}\int_{U_1^\alpha} \widetilde{\sigma}^\alpha_s(\beta)\sqrt{g_s^\alpha(\beta)}\varphi_\delta^\alpha(t,\beta)d\beta dB_s \big| dt\big]\nonumber\\
&\leq& C\int_0^{T-\Delta t}\mathbb{E}\big[ (\int_t^{(t+\Delta t)}(\int_{U_1^\alpha}\widetilde{\sigma}^\alpha_s(\beta)\sqrt{g_s^\alpha(\beta)}\varphi_\delta^\alpha(t,\beta)d\beta)^2 ds)^{\frac{1}{2}}  \big]dt\nonumber\\
&\leq& C_{T, U_1^\alpha}\sqrt{\Delta t}.
\end{eqnarray}

For the remaining term $\mathbb{E}[I_{\Omega_{M}}I_1]$, we see that
\begin{eqnarray}\label{eq48}
      &  & \big|\mathbb{E}[I_{\Omega_{M}}I_1]\big|\nonumber\\
&\leq& \mathbb{E}\big[I_{\Omega_{M}}\int_{0}^{T-\Delta t}\int_{U_{1,2\delta}^\alpha}\int_{U_{1,\delta}^\alpha}\delta^{-d}J\big(\frac{\beta-\theta}{\delta}\big)
\Big|\big|\widetilde{w}^{\epsilon,\alpha}(t,\beta)\big|-\widetilde{w}^{\epsilon,\alpha}(t,\beta)\mathrm{sgn}(\widetilde{w}^{\epsilon,\alpha}(t,\theta))\Big|d\theta d\beta dt \big]\nonumber\\
     & & +C\mathbb{E}\big[I_{\Omega_{M}}\int_{0}^{T-\Delta t}\int_{U_1^\alpha\text{\textbackslash} U_{1,2\delta}^\alpha}|\widetilde{w}^{\epsilon,\alpha}(t,\beta)|d\beta dt \big]\nonumber\\
&\leq & 2\mathbb{E}\big[I_{\Omega_{M}}\int_{0}^{T-\Delta t}\int_{U_{1,2\delta}^\alpha}\int_{U_{1,\delta}^\alpha}\delta^{-d}J\big(\frac{\beta-\theta}{\delta}\big)\Big|\widetilde{w}^{\epsilon,\alpha}(t,\beta)
                -\widetilde{w}^{\epsilon,\alpha}(t,\theta)\Big| d\theta d\beta dt \big] +  C_{T,U_1^{\alpha},M, C_0,L_1}\delta\nonumber\\
&\leq& 4\mathbb{E}\big[I_{\Omega_{M}}\int_{0}^{T}\int_{U_{1,2\delta}^\alpha}\int_{B(0,1)}J\big(z\big)
                 \Big|\widetilde{u}^{\epsilon,\alpha}_t(\beta)\sqrt{g^{\alpha}_t(\beta)}-\widetilde{u}^{\epsilon,\alpha}_t(\beta-z\delta)\sqrt{g^{\alpha}_t(\beta-z\delta)}\Big|dz d\beta dt \big] \nonumber\\
      & & +  C_{T,U_1^{\alpha},M, C_0,L_1}\delta\nonumber\\
&\leq&  C_{T,U_1^{\alpha},M, C_0,L_1}\delta,
\end{eqnarray}
where the second inequality follows from Lemma \ref{lemma3.1} and the fact that $||a|-a\mathrm{sgn}(b)|\leq 2|a-b|$ for any $a,b\in\mathbb{R}$, and the last inequality follows, because of Lemma \ref{lemma3.1} and Lemma \ref{lemma3.2}.

Substituting \eqref{eq42}-\eqref{eq48} into \eqref{eq43}, we get that for every $M > 0$
\begin{eqnarray*}\label{eq51}
   & &  \mathbb{E}\Big[I_{\Omega_{M}}\int_{0}^{T-\Delta t}\int_{U_1^\alpha}|\widetilde{w}^{\epsilon,\alpha}(t,\beta)|d\beta dt\Big]\nonumber\\
&=& \mathbb{E}\Big[I_{\Omega_{M}}\int_{0}^{T-\Delta t}\int_{U_1^\alpha}\big| \widetilde{u}^{\epsilon,\alpha}_{t+\Delta t}(\beta)\sqrt{g^\alpha_{t+\Delta t}(\beta)}-\widetilde{u}^{\epsilon,\alpha}_t(\beta)\sqrt{g^\alpha_{t}(\beta)} \big|d\beta dt\Big]\nonumber\\
&\leq& C_{T,U_1^\alpha,M,C_0,L_1}\big(\delta+\sqrt{\Delta t}+(1+\epsilon)\frac{\Delta t}{\delta}\big).
\end{eqnarray*}
It follows that
\begin{eqnarray*}\label{eq52}
& & \mathbb{E}\Big[I_{\Omega_{M}}\int_{0}^{T-\Delta t}\int_{U_1^\alpha}\big| \widetilde{u}^{\epsilon,\alpha}_{t+\Delta t}(\beta)
                                                                                      -\widetilde{u}^{\epsilon,\alpha}_t(\beta) \big|d\beta dt\Big]\nonumber\\
&\leq& C_{T,U^{\alpha}_1}\mathbb{E}\Big[I_{\Omega_{M}}\int_{0}^{T-\Delta t}\int_{U_1^\alpha}\big| \widetilde{u}^{\epsilon,\alpha}_{t+\Delta t}(\beta)
\sqrt{g^\alpha_{t+\Delta t}(\beta)}-\widetilde{u}^{\epsilon,\alpha}_t(\beta)\sqrt{g^\alpha_{t}(\beta)} \big|d\beta dt\Big] \nonumber\\
   & & + C_{T,U^{\alpha}_1}\mathbb{E}\Big[I_{\Omega_{M}}\int_{0}^{T-\Delta t}\int_{U_1^\alpha}\big| \widetilde{u}^{\epsilon,\alpha}_t(\beta)\big| \cdot \big| \sqrt{g^\alpha_{t+\Delta t}(\beta)}-\sqrt{g^\alpha_{t}(\beta)} \big| d\beta dt\Big]\nonumber\\
&\leq& C_{T,U_1^\alpha,M,C_0,L_1}\big(\delta+\sqrt{\Delta t}+(1+\epsilon)\frac{\Delta t}{\delta} + \Delta t\big) .
\end{eqnarray*}

Finally, taking $\delta = \sqrt{\Delta t}$, we can deduce that
\begin{eqnarray*}\label{eq53}
& & \mathbb{E}\Big[I_{\Omega_{M}}\int_{0}^{T-\Delta t}\big\| \widetilde{u}^{\epsilon}(t+\Delta t)
                                                                                 -\widetilde{u}^{\epsilon}(t) \big\|_{L^1(\Gamma_0)} dt\Big]\nonumber\\
&=&\sum_{\alpha =1}^{N}\mathbb{E}\Big[I_{\Omega_{M}}\int_{0}^{T-\Delta t}\int_{U^{\alpha}_1}\chi^{\alpha}(X^{\alpha}(\beta))\sqrt{g^\alpha(\beta)}\big| \widetilde{u}^{\epsilon,\alpha}_{t+\Delta t}(\beta)-\widetilde{u}^{\epsilon,\alpha}_{t}( \beta) \big|d\beta dt\Big]\nonumber\\
&\leq& C_M(\sqrt{\Delta t} +  \Delta t),
\end{eqnarray*}
for some positive constant $C_M$ independent of $\epsilon$ and $\Delta t$.

The proof is complete.
\end{proof}

\section{Existence of the solution of the conservation law}\label{Existence for the conservation law}
In this section, we prove the existence of a martingale entropy solution to the stochastic conservation law \eqref{eq74} with initial value $u_0 \in H^{1,1}(\Gamma_0) \cap L^{\infty}(\Gamma_0)$.
\begin{proposition}\label{Thm1}
Let Assumptions \ref{Assumption1} and \ref{Assumption2} hold. Then for every $u_0 \in H^{1,1}(\Gamma_0) \cap L^{\infty}(\Gamma_0)$, there exists a martingale entropy solution $({\Omega}^*,{\mathcal{F}}^*,\{{\mathcal{F}}_t^*\}_{t \in [0,T]},\mathbb{P}^*,u^*,B^*)$ to the stochastic conservation law \eqref{eq74} started from $u_0$. Moreover, $\mathbb{P}^*$-a.s.
\begin{equation}\label{eq193}
\|u^*\|_{L^{\infty}(Q_T)} \leq e^{C^*}(1+ \| u_0 \|_{L^\infty(\Gamma_0)}),
\end{equation}
where
\begin{eqnarray*}
C^*&=& (T+1)(L+1)(1+\sup_{t\in[0,T]}\|\nabla_{\Gamma_t} \cdot v\|_{L^{\infty}(\Gamma_t)} +\sup_{t\in[0,T]}\|w^*(t)\|_{L^{\infty}(\Gamma_t)}+ \sup_{t\in[0,T]}\|\nabla_{\Gamma_t}w^*(t)\|_{L^{\infty}(\Gamma_t)} \nonumber\\
      & & + \sup_{t\in[0,T]}\|\Delta_{\Gamma_t}w^*(t)\|_{L^{\infty}(\Gamma_t)}  ),
\end{eqnarray*}
and $w^{*}(t,x) = \int_0^t \sigma(s, G(s, G^{-1}(t,x))) dB^{*}_s, \ (t,x) \in Q_T$.
\end{proposition}

Before proving Proposition \ref{Thm1}, we prepare the following result.
\begin{lemma}\label{lemma3.4}
Assume that $u^\epsilon$ solves the stochastic viscous conservation law \eqref{eq75}. Then $\{u^\epsilon\}_{\epsilon \in (0,1)}$ is tight in $L^1([0,T]; L^1(\Gamma_\cdot))$.
\end{lemma}
\begin{proof}
Recall that
\[\widetilde{u}^\epsilon(t,y):=u^\epsilon(t,G(t,y)), \quad (t,y)\in[0,T] \times \Gamma_{0}.\]
We only need to show that
$\{\widetilde{u}^\epsilon\}_{\epsilon\in(0,1)}$ is tight in $L^1([0,T]; L^1(\Gamma_0))$.

Recall $\Omega_{M}$ given by \eqref{eq50}. For any $M, N>0$, by Chebyshev's inequality and Lemmas \ref{lemma3.1} and \ref{lemma3.2}, we have
\begin{eqnarray*}\label{eq54}
&  &\sup_{\epsilon \in (0,1)}\mathbb{P}\big(\int_0^T \| \widetilde{u}^\epsilon(t)\|_{H^{1,1}(\Gamma_0)}dt > N\big)\nonumber\\
&\leq& \sup_{\epsilon \in (0,1)}\mathbb{P}\big(I_{\Omega_{M}}\int_0^T \| \widetilde{u}^\epsilon(t)\|_{H^{1,1}(\Gamma_0)}dt > N\big) + \mathbb{P}(\Omega_{M}^{c})\nonumber\\
&\leq& \frac{\sup_{\epsilon \in (0,1)}\mathbb{E}\big[I_{\Omega_{M}}\int_0^T \| \widetilde{u}^\epsilon(t)\|_{H^{1,1}(\Gamma_0)}\big]}{N}   +\frac{1}{M}\Big\{\mathbb{E}[\sup_{t\in[0,T]}\|w(t)\|_{L^{\infty}(\Gamma_t)}] + \mathbb{E}[\sup_{t\in[0,T]}\|\nabla_{\Gamma_t}w(t)\|_{L^{\infty}(\Gamma_t)}]\nonumber\\
                  & & + \mathbb{E}[\sup_{t\in[0,T]}\|\nabla_{\Gamma_t}^2w(t)\|_{L^{\infty}(\Gamma_t)}] + \mathbb{E}[\sup_{t\in[0,T]}\|\nabla_{\Gamma_t}\triangle_{\Gamma_t}w(t)\|_{L^{\infty}(\Gamma_t)}]\Big\}\nonumber\\
&\leq& \frac{C_{M,T,C_0, L_1}}{N} + \frac{C}{M}.
\end{eqnarray*}
Because the constant $M$ is arbitrary, it follows that
\begin{equation}\label{eq55}
\lim_{N \rightarrow \infty}\sup_{\epsilon \in (0,1)}\mathbb{P}\big(\int_0^T \| \widetilde{u}^\epsilon_t\|_{H^{1,1}(\Gamma_0)}dt > N\big)=0.
\end{equation}

For any $\varsigma, M>0$ and small constant $\Delta t > 0$, by Chebyshev's inequality and Lemma \ref{lemma3.3}, we have
\begin{eqnarray*}\label{eq56}
  &  & \sup_{\epsilon \in (0,1)}\mathbb{P}\big(\int_0^{T- \Delta t}\|\widetilde{u}^\epsilon(t+\Delta t) -\widetilde{u}^\epsilon(t)\|_{L^1(\Gamma_0)}dt > \varsigma\big)\nonumber\\
&\leq& \sup_{\epsilon \in (0,1)}\mathbb{P}\big(I_{\Omega_{M}}\int_0^{T- \Delta t}\|\widetilde{u}^\epsilon(t+\Delta t) -\widetilde{u}^\epsilon(t)\|_{L^1(\Gamma_0)}dt > \varsigma\big)
        + \mathbb{P}(\Omega_{M}^{c})\nonumber\\
&\leq& \frac{\sup_{\epsilon \in (0,1)}\mathbb{E}\big[I_{\Omega_{M}}\int_0^{T- \Delta t}\|\widetilde{u}^\epsilon(t+\Delta t) -\widetilde{u}^\epsilon(t)\|_{L^1(\Gamma_0)}dt\big]}{\varsigma}+\frac{1}{M}\Big\{\mathbb{E}[\sup_{t\in[0,T]}\|w(t)\|_{L^{\infty}(\Gamma_t)}] \nonumber\\
       & &+ \mathbb{E}[\sup_{t\in[0,T]}\|\nabla_{\Gamma_t}w(t)\|_{L^{\infty}(\Gamma_t)}] + \mathbb{E}[\sup_{t\in[0,T]}\|\nabla_{\Gamma_t}^2w(t)\|_{L^{\infty}(\Gamma_t)}] + \mathbb{E}[\sup_{t\in[0,T]}\|\nabla_{\Gamma_t}\triangle_{\Gamma_t}w(t)\|_{L^{\infty}(\Gamma_t)}]\Big\}\nonumber\\
&\leq& \frac{C_M(\sqrt{\Delta t} +  \Delta t)}{\varsigma} + \frac{C}{M},
\end{eqnarray*}
which implies that
\begin{equation}\label{eq57}
\lim_{\Delta t \rightarrow 0^{+}}\sup_{\epsilon \in (0,1)}\mathbb{P}\big(\int_0^{T- \Delta t}\|\widetilde{u}^\epsilon(t+\Delta t) -\widetilde{u}^\epsilon(t)\|_{L^1(\Gamma_0)}dt > \varsigma\big)=0.
\end{equation}
Therefore, by \eqref{eq55}, \eqref{eq57} and using the fact that $H^{1,1}(\Gamma_0)$ is compactly embedded into $L^1(\Gamma_0)$, together with the criterion for tightness of laws in $L^1([0,T]; L^1(\Gamma_0))$ (see Lemma 5.2 in \cite{MST}), we deduce that $\{\widetilde{u}^\epsilon\}_{\epsilon\in(0,1)}$ is tight in $L^1([0,T];L^1(\Gamma_0))$, completing the proof.

\end{proof}

Now we turn to the proof of Proposition \ref{Thm1}.
\begin{proof}
Let $u_0 \in H^{1,1}(\Gamma_0)\cap L^{\infty}(\Gamma_0)$. Then there exists a family $\{ u_0^\epsilon \}_{\epsilon \in (0,1)}$ satisfying \eqref{eq58},
\[ \| u_0^\epsilon\|_{L^\infty(\Gamma_0)} \leq  \| u_0 \|_{L^\infty(\Gamma_0)}, \]
and $u^\epsilon_0 \rightarrow u_0$ a.e. on $\Gamma_0$. Let $u^\epsilon$ be the solution to equation \eqref{eq75} with initial value $u^\epsilon_0$. By Lemma \ref{lemma3.4}, the family of laws $\{\mathcal{L}(u^\epsilon, B)\}_{\epsilon \in (0,1)}$ of the random vectors $\{(u^\epsilon, B) \}_{\epsilon \in (0,1)}$ is tight in $L^1([0,T]; L^1(\Gamma_{\cdot})) \times C([0,T]; \mathbb{R})$. By Prokhorov's theorem and Skorokhod's representation theorem, there exists a probability space $({\Omega}^*,\mathcal{F}^*,{\mathbb{P}}^*)$ and a sequence of $L^1([0,T]; L^1(\Gamma_{\cdot})) \times C([0,T]; \mathbb{R})$-valued random vectors $\{({u}^{\epsilon_n,*},B^*)\}_{n\geq 1}$ and a $L^1([0,T]; L^1(\Gamma_{\cdot})) \times C([0,T]; \mathbb{R})$-valued random vector $(u^*,B^*)$ such that $\mathcal{L}({u}^{\epsilon_n,*},B^*) = \mathcal{L}(u^{\epsilon_n}, B)$, and
\begin{equation*}
 {u}^{\epsilon_n,*} \rightarrow u^* \quad \text{in} \ L^1([0,T]; L^1(\Gamma_{\cdot})),  \ {\mathbb{P}}^*-a.s.
\end{equation*}


By Lemma \ref{lemma3.1}, we have
\begin{eqnarray*}
      \sup_{t \in [0,T]}\|u^{\epsilon_n, *}(t) \|_{L^{\infty}(\Gamma_t)}
 \leq e^{C^*}(1+ \| u_0 \|_{L^\infty(\Gamma_0)}),
\end{eqnarray*}
where
\begin{eqnarray*}
C^*&=& (T+1)(L+1)(1+\sup_{t\in[0,T]}\|\nabla_{\Gamma_t} \cdot v\|_{L^{\infty}(\Gamma_t)} +\sup_{t\in[0,T]}\|w^*(t)\|_{L^{\infty}(\Gamma_t)}+ \sup_{t\in[0,T]}\|\nabla_{\Gamma_t}w^*(t)\|_{L^{\infty}(\Gamma_t)} \nonumber\\
      & & + \sup_{t\in[0,T]}\|\Delta_{\Gamma_t}w^*(t)\|_{L^{\infty}(\Gamma_t)}  ),
\end{eqnarray*}
and $w^{*}(t,x) = \int_0^t \sigma(s, G(s, G^{-1}(t,x))) dB^{*}_s, \ (t,x) \in Q_T$.
It follows that
$\mathbb{P}^*$-a.s. $u^{*} \in L^{\infty}(Q_T)$ satisfies \eqref{eq193}. Moreover, for any $p \geq 1$,
\begin{equation}\label{eq194}
\mathbb{E}^{\mathbb{P}^*}[\sup_{n \in \mathbb{N}}\sup_{t \in [0,T]}\|u^{\epsilon_n, *}(t)\|^p_{L^{\infty}(\Gamma_t)}]<\infty.
\end{equation}
Therefore, by the dominated convergence theorem, we deduce that
\begin{equation}\label{eq196}
{u}^{\epsilon_n,*} \rightarrow u^* \ \text{in}\ L^1(\Omega^{*} \times Q_T; \mathbb{R}).
\end{equation}
By Lemma \ref{2.3}, we also have
\begin{equation}\label{eq195}
 \sup_{n \in \mathbb{N}}\epsilon_n \mathbb{E}^{\mathbb{P}^*}[\int_{0}^{T}\int_{\Gamma_t}|\nabla_{\Gamma_t} u^{\epsilon_n, *}(t)|^2\mathrm{dvol}_{\Gamma_t}dt] <\infty.
 \end{equation}
Combining \eqref{eq194}, \eqref{eq196} \eqref{eq195}, and arguing similarly to the proof of Lemma \ref{lemma2.4}, we conclude that $u^{*}$ fulfills the entropy condition \eqref{eq73} in Definition \ref{definition1}.

Thus, $({\Omega}^*,{\mathcal{F}}^*,{\mathcal{F}}_t^*,\mathbb{P}^*,u^*,B^*)$ is a martingale entropy solution to equation \eqref{eq74}.
\end{proof}

\section{Uniqueness for the conservation law}\label{Uniqueness for the conservative law}
In this section, we are going to prove the uniqueness of generalized entropy solutions as defined in Definition \ref{definition1}. Let us begin by introducing some definitions and basics. Recall that $\{ \chi^\alpha\}_{\alpha \leq N}$ is a partition of unity with regard to the open cover $\{ V^\alpha = X^\alpha(U^\alpha) \}_{\alpha \leq N}$ on $\Gamma_0$, and $\{\chi^\alpha_t= \chi^\alpha \circ G^{-1}(t, \cdot)\}_{\alpha \leq N}$ is a partition of unity with regard to the open cover $\{ V_t^{\alpha}=G(t, V^\alpha) \}_{\alpha \leq N}$ on $\Gamma_t$. Meanwhile, we also recall the pullback transformation and its associated local parametric representation $\widetilde{\cdot}$ and ${\underset{}{\widetilde{\cdot}}{}}^\alpha$ given by \eqref{chongxing4} and \eqref{chongxing5}, respectively.

For any $1\leq \alpha \leq N$, and sufficiently small $\delta> 0$, we denote by
\[ U^{\alpha}_{2,\delta} = \{y \in U^\alpha: dist(U^\alpha_2, y) < \delta  \}, \]
where $ U^\alpha_2 = supp(\chi^{\alpha}\circ X^{\alpha})$.
Obviously, $U^\alpha_2 \subset U^\alpha_{2,\delta} \subset \subset U^\alpha$.

Under Assumption \ref{Assumption1}(B1), for every $1 \leq \alpha \leq N$, $f$ has the local representation
\[  f((t,x),u) =a_{i}^{\alpha}((t,x),u)\frac{\partial X^\alpha_t}{\partial \theta_i}(X^{\alpha}_t(x)^{-1}), \quad ((t,x),u) \in \Big(\bigcup_{t \in [0,T]} \{t\}\times V_t^{\alpha}\Big) \times \mathbb{R}, \]
where $a_{i}^{\alpha}: \Big(\bigcup_{t \in [0,T]} \{t\}\times V_t^{\alpha}\Big) \times \mathbb{R} \rightarrow \mathbb{R} $. For simplicity, we also set $\widetilde{a}_{i}^{\alpha}(t,\theta,u) = a_{i}^{\alpha}((t, X_t^{\alpha}(\theta)), u )$, $(t,\theta,u) \in [0,T]\times U^\alpha \times \mathbb{R}$. Moreover, by Assumption \ref{Assumption1}(B3), $\widetilde{a}^{\alpha}$ satisfies the following conditions:
\begin{equation}\label{eq96}
|\widetilde{a}^\alpha_u(t,\theta,u)| \vee |\nabla_{\theta}\widetilde{a}^\alpha_u(t,\theta,u)| \vee |\frac{\partial \widetilde{a}^\alpha_u}{\partial t}(t,\theta,u)|  \leq C_{T,U^\alpha_{2,\delta}} (1+ |u|),
\end{equation}
and
\begin{eqnarray}\label{eq181}
      & &|\nabla_{\theta}\widetilde{a}^{\alpha}(t,\theta,u)| \vee |\frac{\partial \widetilde{a}^\alpha}{\partial t}(t,\theta,u)| \vee |\frac{\partial^2 \widetilde{a}^\alpha}{\partial t^2}(t,\theta,u)| \vee |\nabla_{\theta}(\frac{\partial \widetilde{a}^\alpha}{\partial t})(t,\theta,u)| \nonumber\\
      & & \vee |\nabla_{\theta}^2\widetilde{a}^{\alpha}(t,\theta,u)|  \vee |\widetilde{a}^\alpha_{uu}(t,\theta,u)|
\leq C_{T,U^\alpha_{2,\delta}}(1+ |u|^{q_0}),
\end{eqnarray}
for any $(t, \theta) \in [0,T]\times U_{2,\delta}^\alpha$ and $u \in \mathbb{R}$.

For any $1\leq \alpha \leq N$ and $1 \leq i\leq n$, define $F^{\alpha}_i: [0,T] \times U^\alpha \times \mathbb{R} \times \mathbb{R} \rightarrow \mathbb{R}$ as
\[ F^\alpha_i(t,\theta,u_1,u_2) = \mathrm{sign}(u_1 - u_2)\big(\widetilde{a}_i^\alpha(t,\theta,u_1+ \widetilde{w}^\alpha_t(\theta)) - \widetilde{a}_i^\alpha(t,\theta,u_2+\widetilde{w}^\alpha_t(\theta))\big),  \]
$(t,\theta,u_1,u_2) \in [0,T] \times U^\alpha \times \mathbb{R} \times \mathbb{R}$, where $\widetilde{w}^\alpha_t(\theta) = \int_0^t \widetilde{\sigma}^{\alpha}_s(\theta)dB_s$.

By Assumption \ref{Assumption2} and the Kolmogorov's continuity criterion (see Theorem A.3.3 in \cite{RM}), for every $\kappa \in (0, \frac{1}{2})$, there exists a non-negative random variable $C_{\kappa}$ such that with probability one, for all $1 \leq \alpha \leq N$ and $t,s \in [0,T]$
\begin{equation}\label{eq183}
\| \widetilde{w}^\alpha_t - \widetilde{w}^\alpha_s  \|_{L^{\infty}(U^\alpha_{2,\delta})} + \| \nabla_{\theta}\widetilde{w}^\alpha_t - \nabla_{\theta}\widetilde{w}^\alpha_s  \|_{L^{\infty}(U^\alpha_{2,\delta})}\leq C_{\kappa}|t-s|^{\kappa}.
\end{equation}

To prove uniqueness, we need to show that the initial value is approached in the following sense.
\begin{lemma}\label{lemma4.1}
Let Assumptions \ref{Assumption1} and \ref{Assumption2} hold. Assume that $u$ is a generalized entropy solution to the stochastic conservation law \eqref{eq74} with initial value $u_0 \in L^\infty(\Gamma_0)$. Then we have $\mathbb{P}$-a.s.
\[ \lim_{h \rightarrow 0}\frac{1}{h} \int_0^h \int_{\Gamma_0}|V(t, G(t,y)) - u_0(y)|\mathrm{dvol}_{\Gamma_0}(dy)dt = 0,   \]
where $V=u- w$.
\end{lemma}	

\begin{proof}


Using the partition of unity on $\Gamma_0$, it is enough to prove that $\mathbb{P}$-a.s.
\begin{eqnarray*}\label{eq100}
\lim_{h \rightarrow 0}\frac{1}{h} \int_0^h\int_{U^\alpha}|\widetilde{V}^\alpha_s(\theta) - \widetilde{u}_0^\alpha(\theta)|\widetilde{\chi}^{\alpha}(\theta) d\theta ds = 0, \quad \forall 1 \leq \alpha \leq N.
\end{eqnarray*}

Apart from a null set, we may assume that $u \in L^\infty(Q_T)$, \eqref{eq87} in Remark \ref{remark3} is satisfied for $V$, and \eqref{eq183} holds.

For any sufficient small $h > 0$, let $\gamma_h: [0,T] \rightarrow \mathbb{R}_{+}$ be a smooth function satisfying
\begin{equation}\label{eq92}
\gamma_h(0)=2, \quad \gamma_h\leq 2I_{[0,h]}, \quad \gamma^{\prime}_h \leq -\frac{1}{h}I_{[0,h]}.
\end{equation}

Now, fix $1 \leq \alpha \leq N$ and sufficiently small $\delta > 0$. For any $\beta \in U^\alpha_{2,\delta}$ and sufficiently small $h>0$, define $\eta: Q_T \rightarrow \mathbb{R}$ by
\begin{numcases}{\eta(t,x)=}
  \chi^\alpha_t(x)\gamma_h(t)J_{\delta}\bigl((X^\alpha_t)^{-1}(x) - \beta\bigr), & $(t,x)\in \bigcup_{t \in [0,T]} \{t\}\times V_t^{\alpha}$,\nonumber\\
  0, & $(t,x) \in Q_T\setminus\bigcup_{t \in [0,T]} \{t\}\times V_t^{\alpha}$,\nonumber
\end{numcases}
where $J_{\delta}(x) = \frac{1}{\delta^d}J(\frac{x}{\delta})$ and $J$ is the standard mollifier defined in \eqref{eq39}. Obviously, $\eta \in C^1(Q_T)$, $\eta \geq 0$ and $\eta(T)=0$.
Then using Remark \ref{remark3} and by a change of the variables, we get that for all $k \in \mathbb{R}$
\begin{eqnarray}\label{eq93}
& & 2\int_{U^\alpha_2}|\widetilde{u}_0^\alpha(\theta) - k|\widetilde{\chi}^{\alpha}(\theta)J_{\delta}(\theta - \beta)\sqrt{g^\alpha_0(\theta)}d\theta\\
& & - \int_0^T\int_{U^\alpha_2}\mathrm{sign}(\widetilde{V}^\alpha_s(\theta) - k)(k + \widetilde{w}^\alpha_s(\theta))\widetilde{\nabla_{\Gamma_s}\cdot v}^\alpha(\theta)\widetilde{\chi}^{\alpha}(\theta)\gamma_h(s)J_{\delta}(\theta - \beta)\sqrt{g^\alpha_s(\theta)} d\theta ds\nonumber\\
& & + \int_0^T\int_{U^\alpha_2} |\widetilde{V}^\alpha_s(\theta) - k|\widetilde{\chi}^{\alpha}(\theta)\gamma^{\prime}_h(s)J_{\delta}(\theta - \beta)\sqrt{g^\alpha_s(\theta)} d\theta ds\nonumber\\
& & + \int_0^T\int_{U^\alpha_2} F^\alpha_i(s,\theta,\widetilde{V}^\alpha_s(\theta),k)\frac{\partial \widetilde{\chi}^{\alpha}}{\partial \theta_i}(\theta) \gamma_h(s)J_{\delta}(\theta - \beta)\sqrt{g^\alpha_s(\theta)} d\theta ds\nonumber\\
& & + \int_0^T\int_{U^\alpha_2} F^\alpha_i(s,\theta,\widetilde{V}^\alpha_s(\theta),k)\widetilde{\chi}^{\alpha}(\theta)\gamma_h(s)\frac{\partial J_{\delta}(\theta - \beta)}{\partial\theta_i}\sqrt{g^\alpha_s(\theta)} d\theta ds\nonumber\\
& & - \int_0^T\int_{U^\alpha_2}\mathrm{sign}(\widetilde{V}^\alpha_s(\theta) - k)\widetilde{a}^\alpha_{iu}(s,\theta, k + \widetilde{w}^\alpha_s(\theta))\frac{\partial\widetilde{w}^\alpha_s}{\partial \theta_i}(\theta)\widetilde{\chi}^{\alpha}(\theta)\gamma_h(s)J_{\delta}(\theta - \beta)\sqrt{g^\alpha_s(\theta)} d\theta ds \geq 0.\nonumber\
\end{eqnarray}


We now set $k=\widetilde{u}_0^\alpha(\beta)$ in \eqref{eq93}, then integrate the resulting inequality over $U^\alpha_{2,\delta}$ with respect to $\beta$. Using \eqref{eq92} and \eqref{eq96}, we arrive at
\begin{eqnarray}\label{eq94}
      &  & \frac{C_{T, U^\alpha_2}}{h} \int_0^h\int_{U^\alpha_{2,\delta}}\int_{U^\alpha_2}|\widetilde{V}^\alpha_s(\theta) - \widetilde{u}_0^\alpha(\beta)|\widetilde{\chi}^{\alpha}(\theta)J_{\delta}(\theta - \beta) d\theta d\beta ds  \\
&\leq& C_{U^\alpha_2} \int_{U^\alpha_{2,\delta}}\int_{U^\alpha_2} |\widetilde{u}_0^\alpha(\theta) - \widetilde{u}_0^\alpha(\beta)|J_{\delta}(\theta - \beta)d\theta d\beta \nonumber\\
& & + C_{T, U^\alpha_2} \int_0^{2h}\int_{U^\alpha_{2,\delta}}\int_{U^\alpha_2} |\widetilde{u}_0^\alpha(\beta)+\widetilde{w}^\alpha_s(\theta)| J_{\delta}(\theta - \beta)d\theta d\beta ds  \nonumber\\
& & + C_{T, U^\alpha_2}  \int_0^{2h}\int_{U^\alpha_{2,\delta}}\int_{U^\alpha_2}(1+ |\widetilde{u}^\alpha_s(\theta) | + |\widetilde{u}_0^\alpha(\beta)+\widetilde{w}^\alpha_s(\theta)|)|\widetilde{V}^\alpha_s(\theta)-\widetilde{u}_0^\alpha(\beta)|J_{\delta}(\theta - \beta) d\theta d\beta ds   \nonumber\\
& & + C_{T, U^\alpha_2}  \int_0^{2h}\int_{U^\alpha_{2,\delta}}\int_{U^\alpha_2}(1+ |\widetilde{u}^\alpha_s(\theta) | + |\widetilde{u}_0^\alpha(\beta)+\widetilde{w}^\alpha_s(\theta)|)|\widetilde{V}^\alpha_s(\theta)-\widetilde{u}_0^\alpha(\beta)|\nabla_{\theta} J_{\delta}(\theta - \beta)| d\theta d\beta ds   \nonumber\\
& & +  C_{T, U^\alpha_2} \int_0^{2h}\int_{U^\alpha_{2,\delta}}\int_{U^\alpha_2}(1+ |\widetilde{u}_0^\alpha(\beta)+\widetilde{w}^\alpha_s(\theta)|)|\nabla_{\theta}\widetilde{w}^\alpha_s(\theta)|J_{\delta}(\theta - \beta) d\theta d\beta ds. \nonumber\
\end{eqnarray}
Hence, letting $h \rightarrow 0$ in \eqref{eq94}, we obatin
\begin{eqnarray}\label{eq97}
         & & \limsup_{h \rightarrow 0}\frac{1}{h} \int_0^h\int_{U^\alpha_{2,\delta}}\int_{U^\alpha_2}|\widetilde{V}^\alpha_s(\theta) - \widetilde{u}_0^\alpha(\beta)|\widetilde{\chi}^{\alpha}(\theta)J_{\delta}(\theta - \beta) d\theta d\beta ds  \nonumber\\
 & \leq & C_{U^\alpha_2} \int_{U^\alpha_{2,\delta}}\int_{U^\alpha_2} |\widetilde{u}_0^\alpha(\theta) - \widetilde{u}_0^\alpha(\beta)|J_{\delta}(\theta - \beta)d\theta d\beta.
\end{eqnarray}
Due to the fact that
\begin{eqnarray*}\label{eq98}
& & \frac{1}{h} \int_0^h\int_{U^\alpha_2}|\widetilde{V}^\alpha_s(\theta) - \widetilde{u}_0^\alpha(\theta)|\widetilde{\chi}^{\alpha}(\theta) d\theta ds\nonumber\\
&\leq& \frac{1}{h} \int_0^h\int_{U^\alpha_{2,\delta}}\int_{U^\alpha_2}|\widetilde{V}^\alpha_s(\theta) - \widetilde{u}_0^\alpha(\beta)|\widetilde{\chi}^{\alpha}(\theta)J_{\delta}(\theta - \beta) d\theta d\beta ds \nonumber\\
     & &+  C_{U^\alpha_2}\int_{U^\alpha_{2,\delta}}\int_{U^\alpha_2}| \widetilde{u}_0^\alpha(\beta) - \widetilde{u}_0^\alpha(\theta)|J_{\delta}(\theta - \beta) d\theta d\beta,
\end{eqnarray*}
combining with \eqref{eq97} yields
\begin{eqnarray}\label{eq99}
       & & \limsup_{h \rightarrow 0}\frac{1}{h} \int_0^h\int_{U^\alpha_2}|\widetilde{V}^\alpha_s(\theta) - \widetilde{u}_0^\alpha(\theta)|\widetilde{\chi}^{\alpha}(\theta) d\theta ds  \nonumber\\
&\leq& C_{U^\alpha_2}\int_{U^\alpha_{2,\delta}}\int_{U^\alpha_2} |\widetilde{u}_0^\alpha(\theta) - \widetilde{u}_0^\alpha(\beta)|J_{\delta}(\theta - \beta)d\theta d\beta.
\end{eqnarray}
Finally, letting $\delta \rightarrow 0$ in the right-hand side of \eqref{eq99}, we obtain
\begin{eqnarray*}\label{eq101}
 \limsup_{h \rightarrow 0}\frac{1}{h} \int_0^h\int_{U^\alpha_2}|\widetilde{V}^\alpha_s(\theta) - \widetilde{u}_0^\alpha(\theta)|\widetilde{\chi}^{\alpha}(\theta) d\theta ds  = 0,
\end{eqnarray*}
completing the proof.
\end{proof}

\begin{proposition}\label{Thm2}
Let Assumptions \ref{Assumption1} and \ref{Assumption2} hold. Suppose that $u_1$, $u_2$ are two general entropy solutions to the stochastic conservation law \eqref{eq74} with initial values $u_{1,0}$, $u_{2,0} \in L^{\infty}(\Gamma_0)$, respectively. Then there exists a positive constant $C_{1}$ such that $\mathbb{P}$-a.s.
\begin{eqnarray*}
    & & \mathrm{ess\,sup}_{t\in [0,T]}\|u_{1}(t) - u_{2}(t)\|_{L^1(\Gamma_t)} \nonumber\\
&\leq& C_{1} e^{C_{1} (1+ \|u_1\|_{L^\infty(Q_T)} + \|u_2\|_{L^\infty(Q_T)})}\|u_{1,0} - u_{2,0}\|_{L^1(\Gamma_0)}.
\end{eqnarray*}
In particular, the generalized entropy solution to equation \eqref{eq74} is unique.
\end{proposition}
\begin{proof}
Let $u_1$, $u_2$ be two generalized entropy solutions to equation \eqref{eq74} with initial values $u_{1,0}$, $u_{2,0} \in L^\infty(\Gamma_0)$, respectively. Define
\[ V_1= u_1-w, \quad  V_2 = u_2 - w.\]

Remark that $u_1$, $u_2 \in L^\infty(Q_T)$, both $V_1$ and $V_2$ satisfy \eqref{eq87} in Remark \ref{remark3}, and \eqref{eq183} holds.

Let $j: \mathbb{R} \rightarrow \mathbb{R}_{+}$ be a smooth function satisfying
\[ supp(j) \subset [-1, 0]\  \text{and}  \ \int_{-1}^{0} j(r)dr=1. \]
For any $\delta_0 > 0$, let $j_{\delta_0}(r) = \frac{1}{\delta_0}j(\frac{r}{\delta_0}) $. Recall that $J$ is the standard mollifier defined in \eqref{eq39}, and for any $\delta>0$, $J_{\delta}(x) = \frac{1}{\delta^d}J(\frac{x}{\delta})$.
Now, fix $1 \leq \alpha \leq N$ and sufficient small two positive constants $\delta$ and $\delta_0$. For a test function $\eta \in C^1(Q_T)$ with $\eta \geq 0$ and $\eta(\cdot, T)= 0$, we define $\eta_{\delta, \delta_0}: Q_T \times Q_T \rightarrow \mathbb{R}$ by
\begin{numcases}{\eta_{\delta,\delta_0}(t,x; s,y)=}
  j_{\delta_0}(t -s)J_{\delta}\bigl((X^\alpha_t)^{-1}(x) - (X^\alpha_s)^{-1}(y)\bigr)\chi^\alpha_s(y)\eta(s,y), & $(t,x; s,y)\in \Lambda_T$,\nonumber\\
  0, & $(t,x; s,y) \in (Q_T\times Q_T) \setminus \Lambda_T$,\nonumber
\end{numcases}
where $\Lambda_T= (\bigcup_{t \in [0,T]} \{t\}\times V_t^{\alpha}) \times (\bigcup_{s \in [0,T]} \{s\}\times V_s^{\alpha})$.
Note that $j_{\delta_0}(t-s) \neq 0$ only if $s - \delta_0 \leq t \leq s$, and therefore $\eta_{\delta, \delta_0}(t,x; s,y) = 0$ outside $(s - \delta_0)\vee 0 \leq t \leq s$.

Denote $U_T^\alpha = [0,T] \times U^\alpha$ and fix $(s, y)=(s,X_s^\alpha(\beta)) \in \bigcup_{s \in [0,T]} \{s\}\times V_s^{\alpha} $, $\beta \in U^\alpha$. Replacing the test function $\eta$ in Remark \ref{remark3} by $\eta_{\delta, \delta_0}(\cdot, \cdot, s, y)$ ($(s,y)$ fixed), and using the local coordinates, we have for all $k \in \mathbb{R}$,
\begin{eqnarray*}\label{eq102}
& & \int_{U^\alpha} |\widetilde{u}_{1,0}^\alpha(\theta) - k|j_{\delta_0}(-s)J_{\delta}(\theta - \beta)\widetilde{\chi}^{\alpha}(\beta)\widetilde{\eta}^{\alpha}_s(\beta)\sqrt{g^\alpha_0(\theta)}d\theta \nonumber\\
& & - \int_{U_T^\alpha} \mathrm{sign}(\widetilde{V}^\alpha_{1,t}(\theta) -k)(k+ \widetilde{w}^\alpha_t(\theta))\widetilde{\nabla_{\Gamma_t}\cdot v}^\alpha(\theta)j_{\delta_0}(t-s)J_{\delta}(\theta - \beta)\widetilde{\chi}^{\alpha}(\beta)\widetilde{\eta}^{\alpha}_s(\beta)\sqrt{g^\alpha_t(\theta)}d\theta dt \nonumber\\
& & + \int_{U_T^\alpha}  |\widetilde{V}^\alpha_{1,t}(\theta) -k|j^{\prime}_{\delta_0}(t-s)J_{\delta}(\theta - \beta)\widetilde{\chi}^{\alpha}(\beta)\widetilde{\eta}^{\alpha}_s(\beta)\sqrt{g^\alpha_t(\theta)}d\theta dt \nonumber\\
& & - \int_{U_T^\alpha} \mathrm{sign}(\widetilde{V}^\alpha_{1,t}(\theta) -k)\widetilde{a}^\alpha_{iu}(t,\theta,k+\widetilde{w}^\alpha_t(\theta))\frac{\widetilde{w}^\alpha_t}{\partial \theta_i}(\theta)j_{\delta_0}(t-s)J_{\delta}(\theta - \beta)\widetilde{\chi}^{\alpha}(\beta)\widetilde{\eta}^{\alpha}_s(\beta)\sqrt{g^\alpha_t(\theta)}d\theta dt \nonumber\\
& & + \int_{U_T^\alpha} F^\alpha_i(t,\theta,\widetilde{V}^\alpha_{1,t}(\theta),k)j_{\delta_0}(t-s)\frac{\partial J_{\delta}(\theta - \beta)}{\partial \theta_i}\widetilde{\chi}^{\alpha}(\beta)\widetilde{\eta}^{\alpha}_s(\beta)\sqrt{g^\alpha_t(\theta)}d\theta dt \geq 0. 
\end{eqnarray*}

Let $k = \widetilde{V}^\alpha_{2,s}(\beta)$ in the above inequality. Multiply both sides by $\sqrt{g^\alpha_s(\beta)}$ and integrate over $U_T^\alpha$ with respect to $(s,\beta)$ to obtain
\begin{eqnarray}\label{eq103}
& & \int_{U_T^\alpha}\int_{U^\alpha} |\widetilde{u}_{1,0}^\alpha(\theta) - \widetilde{V}^\alpha_{2,s}(\beta) |j_{\delta_0}(-s)J_{\delta}(\theta - \beta)\widetilde{\chi}^{\alpha}(\beta)\widetilde{\eta}^{\alpha}_s(\beta)\sqrt{g^\alpha_0(\theta)}\sqrt{g^\alpha_s(\beta)}d\theta d\beta ds \nonumber\\
& & + \int_{U_T^\alpha}\int_{U_T^\alpha} \Big\{-\mathrm{sign}(\widetilde{V}^\alpha_{1,t}(\theta) -\widetilde{V}^\alpha_{2,s}(\beta))(\widetilde{V}^\alpha_{2,s}(\beta)+ \widetilde{w}^\alpha_t(\theta))\widetilde{\nabla_{\Gamma_t}\cdot v}^\alpha(\theta)j_{\delta_0}(t-s)J_{\delta}(\theta - \beta)\widetilde{\chi}^{\alpha}(\beta)\widetilde{\eta}^{\alpha}_s(\beta)\nonumber\\
& &  + |\widetilde{V}^\alpha_{1,t}(\theta) -\widetilde{V}^\alpha_{2,s}(\beta)|j^{\prime}_{\delta_0}(t-s)J_{\delta}(\theta - \beta)\widetilde{\chi}^{\alpha}(\beta)\widetilde{\eta}^{\alpha}_s(\beta) \nonumber\\
& &  - \mathrm{sign}(\widetilde{V}^\alpha_{1,t}(\theta) -\widetilde{V}^\alpha_{2,s}(\beta))\widetilde{a}^\alpha_{iu}(t,\theta,\widetilde{V}^\alpha_{2,s}(\beta)+\widetilde{w}^\alpha_t(\theta))\frac{\widetilde{w}^\alpha_t}{\partial \theta_i}(\theta)j_{\delta_0}(t-s)J_{\delta}(\theta - \beta)\widetilde{\chi}^{\alpha}(\beta)\widetilde{\eta}^{\alpha}_s(\beta) \nonumber\\
& & + F^\alpha_i(t,\theta,\widetilde{V}^\alpha_{1,t}(\theta),\widetilde{V}^\alpha_{2,s}(\beta))j_{\delta_0}(t-s)\frac{\partial J_{\delta}(\theta - \beta)}{\partial \theta_i}\widetilde{\chi}^{\alpha}(\beta)\widetilde{\eta}^{\alpha}_s(\beta) \Big\} \sqrt{g^\alpha_t(\theta)}\sqrt{g^\alpha_s(\beta)}d\theta dt d\beta ds
\nonumber\\
& & \geq 0.
\end{eqnarray}

Now, taking $\eta_{\delta, \delta_0}(t,x, \cdot, \cdot)$ ($(t,x)$ fixed) as a test function in \eqref{eq87} and proceeding analogously for the generalized entropy solution $u_{2}(s,y)$ we also get
\begin{eqnarray}\label{eq104}
& & \int_{U_T^\alpha}\int_{U_T^\alpha} \Big\{ -\mathrm{sign}(\widetilde{V}^\alpha_{2,s}(\beta) -\widetilde{V}^\alpha_{1,t}(\theta))(\widetilde{V}^\alpha_{1,t}(\theta)+\widetilde{w}^\alpha_s(\beta))\widetilde{\nabla_{\Gamma_s}\cdot v}^\alpha(\beta)j_{\delta_0}(t-s)J_{\delta}(\theta - \beta)\widetilde{\chi}^{\alpha}(\beta)\widetilde{\eta}^{\alpha}_s(\beta) \nonumber\\
& &  - |\widetilde{V}^\alpha_{2,s}(\beta) -\widetilde{V}^\alpha_{1,t}(\theta)|j^{\prime}_{\delta_0}(t-s)J_{\delta}(\theta - \beta)\widetilde{\chi}^{\alpha}(\beta)\widetilde{\eta}^{\alpha}_s(\beta) \nonumber\\
& &  + |\widetilde{V}^\alpha_{2,s}(\beta) -\widetilde{V}^\alpha_{1,t}(\theta)|j_{\delta_0}(t-s)J_{\delta}(\theta - \beta)\widetilde{\chi}^{\alpha}(\beta)\frac{\partial \widetilde{\eta}^{\alpha}_s }{\partial s}(\beta) \nonumber\\
& &  - \mathrm{sign}(\widetilde{V}^\alpha_{2,s}(\beta) -\widetilde{V}^\alpha_{1,t}(\theta))\widetilde{a}^\alpha_{iu}(s,\beta,\widetilde{V}^\alpha_{1,t}(\theta)+\widetilde{w}^\alpha_s(\beta))\frac{\widetilde{w}^\alpha_s}{\partial \beta_i}(\beta)j_{\delta_0}(t-s)J_{\delta}(\theta - \beta)\widetilde{\chi}^{\alpha}(\beta)\widetilde{\eta}^{\alpha}_s(\beta) \nonumber\\
& &  - F^\alpha_i(s,\beta,\widetilde{V}^\alpha_{2,s}(\beta),\widetilde{V}^\alpha_{1,t}(\theta))j_{\delta_0}(t-s)\frac{\partial J_{\delta}(\theta - \beta)}{\partial \theta_i}\widetilde{\chi}^{\alpha}(\beta)\widetilde{\eta}^{\alpha}_s(\beta)\nonumber\\
& &  +  F^\alpha_i(s,\beta,\widetilde{V}^\alpha_{2,s}(\beta),\widetilde{V}^\alpha_{1,t}(\theta))j_{\delta_0}(t-s)J_{\delta}(\theta - \beta)\frac{\partial \widetilde{\chi}^{\alpha} }{\partial \beta_i}(\beta)\widetilde{\eta}^{\alpha}_s(\beta)\nonumber\\
& &  +  F^\alpha_i(s,\beta,\widetilde{V}^\alpha_{2,s}(\beta),\widetilde{V}^\alpha_{1,t}(\theta))j_{\delta_0}(t-s)J_{\delta}(\theta - \beta)\widetilde{\chi}^{\alpha}(\beta)\frac{\partial \widetilde{\eta}^{\alpha}_s }{\partial \beta_i}(\beta) \Big\}  \sqrt{g^\alpha_t(\theta)}\sqrt{g^\alpha_s(\beta)} d\beta ds d\theta dt \nonumber\\
& & \geq 0.
\end{eqnarray}
Adding \eqref{eq103} to \eqref{eq104}, we get
\begin{eqnarray}\label{eq105}
& &  \int_{U_T^\alpha}\int_{U^\alpha} |\widetilde{u}_{1,0}^\alpha(\theta) - \widetilde{V}^\alpha_{2,s}(\beta) |j_{\delta_0}(-s)J_{\delta}(\theta - \beta)\widetilde{\chi}^{\alpha}(\beta)\widetilde{\eta}^{\alpha}_s(\beta)\sqrt{g^\alpha_0(\theta)}\sqrt{g^\alpha_s(\beta)}d\theta d\beta ds \nonumber\\
& & + \int_{U_T^\alpha}\int_{U_T^\alpha}\Big\{ \big\{\mathrm{sign}(\widetilde{V}^\alpha_{1,t}(\theta) - \widetilde{V}^\alpha_{2,s}(\beta))\big( \widetilde{V}^\alpha_{1,t}(\theta)\widetilde{\nabla_{\Gamma_s}\cdot v}^\alpha(\beta) - \widetilde{V}^\alpha_{2,s}(\beta)\widetilde{\nabla_{\Gamma_t}\cdot v}^\alpha(\theta) \big)j_{\delta_0}(t-s)\nonumber\\
& & \cdot J_{\delta}(\theta - \beta)\widetilde{\chi}^{\alpha}(\beta)\widetilde{\eta}^{\alpha}_s(\beta) \big\} + \big\{ \mathrm{sign}(\widetilde{V}^\alpha_{1,t}(\theta) - \widetilde{V}^\alpha_{2,s}(\beta))\big( \widetilde{w}^\alpha_{s}(\beta)\widetilde{\nabla_{\Gamma_s}\cdot v}^\alpha(\beta) - \widetilde{w}^\alpha_{t}(\theta)\widetilde{\nabla_{\Gamma_t}\cdot v}^\alpha(\theta) \big)\nonumber\\
& &  \cdot j_{\delta_0}(t-s)J_{\delta}(\theta - \beta)\widetilde{\chi}^{\alpha}(\beta)\widetilde{\eta}^{\alpha}_s(\beta) \big\} + \big\{ |\widetilde{V}^\alpha_{2,s}(\beta) -\widetilde{V}^\alpha_{1,t}(\theta)|j_{\delta_0}(t-s)J_{\delta}(\theta - \beta)\widetilde{\chi}^{\alpha}(\beta)\frac{\partial \widetilde{\eta}^{\alpha}_s }{\partial s}(\beta) \big\} \nonumber\\
& &   - \big\{\mathrm{sign}(\widetilde{V}^\alpha_{1,t}(\theta)
-\widetilde{V}^\alpha_{2,s}(\beta))\widetilde{a}^\alpha_{iu}(t,\theta,\widetilde{V}^\alpha_{2,s}(\beta)+\widetilde{w}^\alpha_t(\theta))\frac{\widetilde{w}^\alpha_t}{\partial \theta_i}(\theta)j_{\delta_0}(t-s)J_{\delta}(\theta - \beta)\widetilde{\chi}^{\alpha}(\beta)\widetilde{\eta}^{\alpha}_s(\beta) \big\} \nonumber\\
& &   + \big\{ \mathrm{sign}(\widetilde{V}^\alpha_{1,t}(\theta)- \widetilde{V}^\alpha_{2,s}(\beta) )\widetilde{a}^\alpha_{iu}(s,\beta,\widetilde{V}^\alpha_{1,t}(\theta)+\widetilde{w}^\alpha_s(\beta))\frac{\widetilde{w}^\alpha_s}{\partial \beta_i}(\beta)j_{\delta_0}(t-s)J_{\delta}(\theta - \beta)\widetilde{\chi}^{\alpha}(\beta)\widetilde{\eta}^{\alpha}_s(\beta) \big\}\nonumber\\
& &   + \big\{\mathrm{sign}(\widetilde{V}^\alpha_{1,t}(\theta) -\widetilde{V}^\alpha_{2,s}(\beta))\big( \widetilde{a}^\alpha_{i}(t,\theta,\widetilde{u}^\alpha_{1,t}(\theta)) - \widetilde{a}^\alpha_{i}(s,\beta,\widetilde{V}^\alpha_{1,t}(\theta)+\widetilde{w}^\alpha_s(\beta)) \big)j_{\delta_0}(t-s)
  \frac{\partial J_{\delta}(\theta - \beta)}{\partial \theta_i} \nonumber\\
& & \cdot \widetilde{\chi}^{\alpha}(\beta)\widetilde{\eta}^{\alpha}_s(\beta) \big\}+ \big\{\mathrm{sign}(\widetilde{V}^\alpha_{1,t}(\theta) - \widetilde{V}^\alpha_{2,s}(\beta) )\big( \widetilde{a}^\alpha_{i}(s,\beta,\widetilde{u}^\alpha_{2,s}(\beta)) - \widetilde{a}^\alpha_{i}(t,\theta,\widetilde{V}^\alpha_{2,s}(\beta)+\widetilde{w}^\alpha_t(\theta)) \big)j_{\delta_0}(t-s) \nonumber\\
& & \cdot \frac{\partial J_{\delta}(\theta - \beta)}{\partial\theta_i}\widetilde{\chi}^{\alpha}(\beta)\widetilde{\eta}^{\alpha}_s(\beta) \big\}
    + \big\{ F^\alpha_i(s,\beta,\widetilde{V}^\alpha_{2,s}(\beta),\widetilde{V}^\alpha_{1,t}(\theta))j_{\delta_0}(t-s)J_{\delta}(\theta - \beta)\frac{\partial \widetilde{\chi}^{\alpha} }{\partial \beta_i}(\beta)\widetilde{\eta}^{\alpha}_s(\beta) \big\}\nonumber\\
 & & + \big\{ F^\alpha_i(s,\beta,\widetilde{V}^\alpha_{2,s}(\beta),\widetilde{V}^\alpha_{1,t}(\theta))j_{\delta_0}(t-s)J_{\delta}(\theta - \beta)\widetilde{\chi}^{\alpha}(\beta)\frac{\partial \widetilde{\eta}^{\alpha}_s }{\partial \beta_i}(\beta) \big\} \Big\}  \sqrt{g^\alpha_t(\theta)}\sqrt{g^\alpha_s(\beta)}  d\theta dtd\beta ds \nonumber\\
 & & =: \sum_{i=1}^{10}I_i \geq 0.
\end{eqnarray}

We claim that $\mathbb{P}$-a.s.
\begin{eqnarray}\label{eq157}
      & & \limsup_{\delta\rightarrow 0}\limsup_{\delta_0\rightarrow 0}\sum_{i=1}^{10}I_i \nonumber\\
&= &  \int_{U^\alpha}|\widetilde{u}_{1,0}^\alpha(\beta) - \widetilde{u}^\alpha_{2,0}(\beta) |\widetilde{\chi}^{\alpha}(\beta)\widetilde{\eta}^{\alpha}_0(\beta)g^\alpha_0(\beta)d\beta \nonumber\\
     & & + \int_{U_T^\alpha}|\widetilde{V}^\alpha_{1,s}(\beta) -\widetilde{V}^\alpha_{2,s}(\beta)| \widetilde{\nabla_{\Gamma_s}\cdot v}^\alpha(\beta)\widetilde{\chi}^{\alpha}(\beta)\widetilde{\eta}^{\alpha}_s(\beta)g^{\alpha}_s(\beta)d\beta ds \nonumber\\
     & & + \int_{U_T^\alpha} |\widetilde{V}^\alpha_{2,s}(\beta) -\widetilde{V}^\alpha_{1,s}(\beta)|\widetilde{\chi}^{\alpha}(\beta)\frac{\partial \widetilde{\eta}^{\alpha}_s }{\partial s}(\beta) g^\alpha_s(\beta) d\beta ds \nonumber\\
     & & - \int_{U_T^\alpha} \mathrm{sign}(\widetilde{V}^\alpha_{1,s}(\beta) -\widetilde{V}^\alpha_{2,s}(\beta))\big( \frac{\partial \widetilde{a}^\alpha_{i}}{\partial \beta_i}(s, \beta,\widetilde{u}^\alpha_{1,s}(\beta)) - \frac{\partial \widetilde{a}^\alpha_{i}}{\partial \beta_i}(s, \beta,\widetilde{u}^\alpha_{2,s}(\beta)) \big)\widetilde{\chi}^{\alpha}(\beta)\widetilde{\eta}^{\alpha}_s(\beta)g^\alpha_s(\beta)d\beta ds  \nonumber\\
     & & +\int_{U_T^\alpha}F^\alpha_i(s,\beta,\widetilde{V}^\alpha_{1,s}(\beta),\widetilde{V}^\alpha_{2,s}(\beta))\frac{\partial \widetilde{\chi}^{\alpha} }{\partial \beta_i}(\beta)\widetilde{\eta}^{\alpha}_s(\beta)g^\alpha_s(\beta) d\beta ds  \nonumber\\
     & & + \int_{U_T^\alpha} F^\alpha_i(s,\beta,\widetilde{V}^\alpha_{1,s}(\beta),\widetilde{V}^\alpha_{2,s}(\beta))\widetilde{\chi}^{\alpha}(\beta)\frac{\partial \widetilde{\eta}^{\alpha}_s }{\partial \beta_i}(\beta)g^\alpha_s(\beta) d\beta ds.
\end{eqnarray}
The proof of the above claim is long and it is provided in the Appendix.
Now, we continue to finish the proof of the Proposition admitting the claim \eqref{eq157}.
Since
\[ t \mapsto \int_{\Gamma_t}|u_{1}(t,x) -u_{2}(t,x)|\mathrm{dvol}_{\Gamma_t}(dx) \in L^1([0,T]),\]
there exists a set $\mathcal{E} \subset [0,T]$ such that the Lebesgue measure of $[0,T]\backslash \mathcal{E}$ is zero and each $t \in \mathcal{E}$ is a Lebesgue point of this function. We fix an arbitrary $t \in \mathcal{E}\cap (0,T)$. For sufficiently small $h>0$, let $a_h: [0,T] \rightarrow \mathbb{R}_{+}$ be a smooth function satisfying
\[ a_h(0)=2, \quad a^{\prime}_h \leq -\frac{1}{h}I_{[t, t+h]}, \quad  a_h \leq 2I_{[0, t+2h]}.  \]
In view of \eqref{eq96}, set $\eta(s,x) = a_h(s)$ in \eqref{eq157} and sum over $\alpha$ to obtain 
\begin{eqnarray*}\label{eq158}
    & & \frac{1}{h}\int_t^{t+h} \int_{\Gamma_s}|u_{1}(s,x) - u_{2}(s,x)|\mathrm{dvol}_{\Gamma_s}(dx)ds \nonumber\\
&\leq& C_{T,U^{1}_2,\cdots,U^{N}_2}\int_{\Gamma_0} |u_{1,0}(y) - u_{2,0}(y)|\mathrm{dvol}_{\Gamma_0}(dy)\nonumber\\
       &  &+ C_{T,U^{1}_2,\cdots,U^{N}_2}(1+N)(1+ \|u_1\|_{L^\infty(Q_T)} + \|u_2\|_{L^\infty(Q_T)})\int_0^{t+2h}\int_{\Gamma_s} |u_{1}(s,x) - u_{2}(s,x)|\mathrm{dvol}_{\Gamma_s}(dx)ds.\nonumber\\
\end{eqnarray*}
Letting $h\rightarrow 0$, we get
\begin{eqnarray*}\label{eq159}
      & & \int_{\Gamma_t}|u_{1}(t,x) - u_{2}(t,x)| \mathrm{dvol}_{\Gamma_t}(dx)\nonumber\\
 &\leq& C_{T,U^{1}_2,\cdots,U^{N}_2}\int_{\Gamma_0} |u_{1,0}(y) - u_{2,0}(y)|\mathrm{dvol}_{\Gamma_0}(dy) \nonumber\\
        & & + C_{T,U^1_2,\cdots,U^N_2}(1+N)(1+ \|u_1\|_{L^\infty(Q_T)} + \|u_2\|_{L^\infty(Q_T)})\int_0^{t}\int_{\Gamma_s} |u_{1}(s,x) - u_{2}(s,x)|\mathrm{dvol}_{\Gamma_s}(dx)ds. \nonumber\\
\end{eqnarray*}
By the Gronwall's lemma, we obtain
\begin{eqnarray*}\label{eq161}
    & & \mathrm{ess\,sup}_{t\in [0,T]}\|u_{1}(t) - u_{2}(t)\|_{L^1(\Gamma_t)} \nonumber\\
&\leq& C_{T,U^{1}_2,\cdots,U^{N}_2}e^{C_{T,U^1_2,\cdots,U^N_2}(1+N)(1+ \|u_1\|_{L^\infty(Q_T)} + \|u_2\|_{L^\infty(Q_T)})}\|u_{1,0} - u_{2,0}\|_{L^1(\Gamma_0)},
\end{eqnarray*}
which completes the proof.
\end{proof}
	
\section{Appendix}

In this section, we provide the proof of the claim \eqref{eq157}, which is completed by several Lemmas below.
\begin{lemma}\label{lemma5.1}
\begin{equation*}\label{eq106}
\lim_{(\delta_0,\delta) \rightarrow (0,0)}I_1 = \int_{U^\alpha}|\widetilde{u}_{1,0}^\alpha(\beta) - \widetilde{u}^\alpha_{2,0}(\beta) |\widetilde{\chi}^{\alpha}(\beta)\widetilde{\eta}^{\alpha}_0(\beta)g^\alpha_0(\beta)d\beta.
\end{equation*}
\end{lemma}

\begin{proof}
Since $\widetilde{\eta}^{\alpha}_{\cdot}(\cdot)$ and $\sqrt{g^\alpha_{\cdot}(\cdot)}$ belong to $C^1([0,T]\times U^\alpha)$, we have
\begin{eqnarray}\label{eq108}
     &  & | I_1 - \int_{U^\alpha}|\widetilde{u}_{1,0}^\alpha(\beta) - \widetilde{u}^\alpha_{2,0}(\beta) |\widetilde{\chi}^{\alpha}(\beta)\widetilde{\eta}^{\alpha}_0(\beta)g^\alpha_0(\beta)d\beta |\nonumber\\
&\leq&   \int_{U_T^\alpha}\int_{U^\alpha} \big||\widetilde{u}_{1,0}^\alpha(\theta) - \widetilde{V}^\alpha_{2,s}(\beta) | - |\widetilde{u}_{1,0}^\alpha(\beta) - \widetilde{u}^\alpha_{2,0}(\beta) | \big|j_{\delta_0}(-s)J_{\delta}(\theta -\beta)\widetilde{\chi}^{\alpha}(\beta)\widetilde{\eta}^{\alpha}_s(\beta) \nonumber\\
         & & \cdot\sqrt{g^\alpha_0(\theta)}\sqrt{g^\alpha_s(\beta)}d\theta d\beta ds  \nonumber\\
     & & +\int_{U_T^\alpha}\int_{U^\alpha} |\widetilde{u}_{1,0}^\alpha(\beta) - \widetilde{u}^\alpha_{2,0}(\beta) |j_{\delta_0}(-s)
         J_{\delta}(\theta - \beta)\widetilde{\chi}^{\alpha}(\beta)\sqrt{g^\alpha_0(\theta)}
          |\widetilde{\eta}^{\alpha}_s(\beta)\sqrt{g^\alpha_s(\beta)}-\widetilde{\eta}^{\alpha}_0(\beta)
                 \sqrt{g^\alpha_0(\beta)}|d\theta d\beta ds\nonumber\\
     & & + \int_{U^\alpha}\int_{U^\alpha} |\widetilde{u}_{1,0}^\alpha(\beta) - \widetilde{u}^\alpha_{2,0}(\beta) |J_{\delta}(\theta - \beta)\widetilde{\chi}^{\alpha}(\beta)\widetilde{\eta}^{\alpha}_0(\beta)\sqrt{g^\alpha_0(\beta)}|\sqrt{g^\alpha_0(\theta)} - \sqrt{g^\alpha_0(\beta)}|d\theta d\beta \nonumber\\
&\leq& C_{T,U^\alpha_2}\int_{U^\alpha_2}\int_{U^\alpha_{2,\delta}} |\widetilde{u}_{1,0}^\alpha(\theta) - \widetilde{u}_{1,0}^\alpha(\beta)|
       J_{\delta}(\theta - \beta)d\theta d\beta + \frac{C_{T,U^\alpha_2}}{\delta_0} \int_0^{\delta_0}\int_{U^\alpha_2}| \widetilde{u}^\alpha_{2,0}(\beta) - \widetilde{V}^\alpha_{2,s}(\beta) |\widetilde{\chi}^{\alpha}(\beta) d\beta ds \nonumber\\
     & & +C_{T,U^\alpha_2}\int_0^T j_{\delta_0}(-s)|s|ds +C_{U^\alpha_2}\int_{U^\alpha_2}\int_{U^\alpha_{2,\delta}}J_{\delta}(\theta - \beta)|\theta - \beta|d\theta d\beta \nonumber\\
&\leq& C_{T,U^\alpha_2}\int_{|r| < 1}\int_{U^\alpha_2}|\widetilde{u}_{1,0}^\alpha(\beta + \delta r) - \widetilde{u}_{1,0}^\alpha(\beta )|J(r)d\beta dr
       + \frac{C_{T,U^\alpha_2}}{\delta_0} \int_0^{\delta_0}\int_{U^\alpha_2}| \widetilde{u}^\alpha_{2,0}(\beta) - \widetilde{V}^\alpha_{2,s}(\beta) |\widetilde{\chi}^{\alpha}(\beta) d\beta ds \nonumber\\
       & & +C_{T,U^\alpha_2}\delta_0 + C_{U^\alpha_2}\delta.
\end{eqnarray}
Note that for all $|r|<1$,
\[\lim_{\delta \rightarrow 0}\int_{U^\alpha_2}|\widetilde{u}_{1,0}^\alpha(\beta + \delta r) - \widetilde{u}_{1,0}^\alpha(\beta )|d\beta = 0,\]
therefore by bounded convergence theorem, we have
\[\lim_{\delta \rightarrow 0}\int_{|r| < 1}\int_{U^\alpha_2}|\widetilde{u}_{1,0}^\alpha(\beta + \delta r) - \widetilde{u}_{1,0}^\alpha(\beta )|J(r)d\beta dr = 0.\]
Combining this with Lemma \ref{lemma4.1} and letting $(\delta_0,\delta) \rightarrow (0,0)$ in \eqref{eq108}, we obtain
\[ \limsup_{\delta \rightarrow 0}\limsup_{\delta_0 \rightarrow 0}| I_1 - \int_{U^\alpha}|\widetilde{u}_{1,0}^\alpha(\beta) - \widetilde{u}^\alpha_{2,0}(\beta) |\widetilde{\chi}^{\alpha}(\beta)\widetilde{\eta}^{\alpha}_0(\beta)g^\alpha_0(\beta)d\beta |=0,  \]
completing the proof.
\end{proof}

\begin{lemma}\label{lemma5.2}
It holds that
\begin{eqnarray*}\label{eq112}
 \lim_{(\delta_0, \delta)  \rightarrow (0, 0)} (I_2+I_3) = \int_{U_T^\alpha}|\widetilde{V}^\alpha_{1,s}(\beta) -\widetilde{V}^\alpha_{2,s}(\beta)| \widetilde{\nabla_{\Gamma_s}\cdot v}^\alpha(\beta)\widetilde{\chi}^{\alpha}(\beta)\widetilde{\eta}^{\alpha}_s(\beta)g^{\alpha}_s(\beta)d\beta ds.
\end{eqnarray*}
and
\begin{equation}\label{eq121}
  \lim_{(\delta_0, \delta)  \rightarrow (0, 0)} I_4
= \int_{U_T^\alpha} |\widetilde{V}^\alpha_{2,s}(\beta) -\widetilde{V}^\alpha_{1,s}(\beta)|\widetilde{\chi}^{\alpha}(\beta)\frac{\partial \widetilde{\eta}^{\alpha}_s }{\partial s}(\beta) g^\alpha_s(\beta) d\beta ds.
\end{equation}
\end{lemma}

\begin{proof}
We first consider $I_2$. Due to the fact that $\widetilde{\nabla_{\Gamma_\cdot}\cdot v}^\alpha(\cdot)$, $\sqrt{g^\alpha_{\cdot}(\cdot)}$ and $\widetilde{\eta}^{\alpha}_{\cdot}(\cdot)$ belong to $C^1([0,T]\times U^\alpha)$, and $\widetilde{V}^\alpha_{1,\cdot}(\cdot)$, $\widetilde{V}^\alpha_{2,\cdot}(\cdot)$ are in $L^\infty([0,T]\times U^\alpha)$, we have
\begin{eqnarray}\label{eq114}
      & & \big |I_2 - \int_{U_T^\alpha}|\widetilde{V}^\alpha_{1,s}(\beta) -\widetilde{V}^\alpha_{2,s}(\beta)| \widetilde{\nabla_{\Gamma_s}\cdot v}^\alpha(\beta)\widetilde{\chi}^{\alpha}(\beta)\widetilde{\eta}^{\alpha}_s(\beta)g^{\alpha}_s(\beta)d\beta ds \big| \nonumber\\
&\leq& \int_{U_T^\alpha}\int_{U_T^\alpha}|\widetilde{V}^\alpha_{1,t}(\theta)|\cdot |\widetilde{\nabla_{\Gamma_s}\cdot v}^\alpha(\beta)|j_{\delta_0}(t-s)J_{\delta}(\theta - \beta)\widetilde{\chi}^{\alpha}(\beta)\widetilde{\eta}^{\alpha}_s(\beta)\sqrt{g^\alpha_s(\beta)}\nonumber\\
       & & \cdot |\sqrt{g^\alpha_t(\theta)} - \sqrt{g^\alpha_s(\beta)}| d\theta dtd\beta ds \nonumber\\
       & &+ \int_{U_T^\alpha}\int_{U_T^\alpha} |\widetilde{V}^\alpha_{2,s}(\beta)| \cdot \big| \widetilde{\nabla_{\Gamma_s}\cdot v}^\alpha(\beta)\sqrt{g^\alpha_s(\beta)} - \widetilde{\nabla_{\Gamma_t}\cdot v}^\alpha(\theta)\sqrt{g^\alpha_t(\theta)} \big| j_{\delta_0}(t-s)J_{\delta}(\theta - \beta) \nonumber\\
        & & \cdot \widetilde{\chi}^{\alpha}(\beta)\widetilde{\eta}^{\alpha}_s(\beta)\sqrt{g^\alpha_s(\beta)} d\theta dtd\beta ds \nonumber\\
        & & + \int_{U_T^\alpha}\int_{U_T^\alpha} \big| |\widetilde{V}^\alpha_{1,t}(\theta) - \widetilde{V}^\alpha_{2,s}(\beta)| - |\widetilde{V}^\alpha_{1,s}(\beta) - \widetilde{V}^\alpha_{2,s}(\beta)|  \big| \cdot |\widetilde{\nabla_{\Gamma_s}\cdot v}^\alpha(\beta)| j_{\delta_0}(t-s) \nonumber\\
        & &  \cdot J_{\delta}(\theta - \beta) \widetilde{\chi}^{\alpha}(\beta)\widetilde{\eta}^{\alpha}_s(\beta)g^\alpha_s(\beta) d\theta dtd\beta ds  \nonumber\\
        & & + \int_0^{\delta_0}\int_{U^\alpha}\int_0^T  |\widetilde{V}^\alpha_{1,s}(\beta) - \widetilde{V}^\alpha_{2,s}(\beta)| \cdot |\widetilde{\nabla_{\Gamma_s}\cdot v}^\alpha(\beta)| j_{\delta_0}(t-s)
         \widetilde{\chi}^{\alpha}(\beta)\widetilde{\eta}^{\alpha}_s(\beta)g^\alpha_s(\beta)dtd\beta ds  \nonumber\\
        & & + \int_0^{\delta_0}\int_{U^\alpha}|\widetilde{V}^\alpha_{1,s}(\beta) - \widetilde{V}^\alpha_{2,s}(\beta)| \cdot |\widetilde{\nabla_{\Gamma_s}\cdot v}^\alpha(\beta)| \widetilde{\chi}^{\alpha}(\beta)\widetilde{\eta}^{\alpha}_s(\beta)g^\alpha_s(\beta) d\beta ds   \nonumber\\
&\leq& C_{T,U^\alpha_2}\int_0^T\int_{U^\alpha_2}\int_0^T\int_{U^\alpha_{2,\delta}} (|\theta - \beta|+|t-s|)j_{\delta_0}(t-s)J_{\delta}(\theta - \beta) d\theta dtd\beta ds \nonumber\\
     & & + C_{T,U^\alpha_2}\int_0^T\int_{U^\alpha_2}\int_0^T\int_{U^\alpha_{2,\delta}} |\widetilde{V}^\alpha_{1,t}(\theta)- \widetilde{V}^\alpha_{1,s}(\beta)|j_{\delta_0}(t-s)J_{\delta}(\theta - \beta) d\theta dtd\beta ds
         + C_{T,U^\alpha_2}\delta_0 \nonumber\\
&\leq& C_{T,U^\alpha_2} (\delta+\delta_0)
      + C_{T,U^\alpha_2}\int_0^1\int_0^T\int_{U^\alpha_{2,\delta}}|\widetilde{V}^\alpha_{1,t}(\theta)- \widetilde{V}^\alpha_{1,t+\delta_0 \tilde{r}}(\theta)|I_{\{t+\delta_0 \tilde{r} \leq T\}}j(-\tilde{r})d\theta dtd\tilde{r} \nonumber\\
     & & + C_{T,U^\alpha_2}\int_{|r|<1}\int_0^T\int_{U^\alpha_2} |\widetilde{V}^\alpha_{1,s}(\beta+\delta r)- \widetilde{V}^\alpha_{1,s}(\beta)|J(r)d\beta dsdr.
\end{eqnarray}
Noting that for all $\tilde{r} \in [0,1]$ and $|r| <1$,
\[ \lim_{\delta_0 \rightarrow 0} \int_0^T\int_{U^\alpha_{2,\delta}}|\widetilde{V}^\alpha_{1,t}(\theta)- \widetilde{V}^\alpha_{1,t+\delta_0 \tilde{r}}(\theta)|I_{\{t+\delta_0 \tilde{r} \leq T\}}d\theta dt = 0, \]
and
\[\lim_{\delta \rightarrow 0} \int_0^T\int_{U^\alpha_2} |\widetilde{V}^\alpha_{1,s}(\beta+\delta r)- \widetilde{V}^\alpha_{1,s}(\beta)|d\beta ds =0,  \]
therefore by bounded convergence theorem, we have
\begin{eqnarray}\label{eq115}
\lim_{\delta_0 \rightarrow 0}\int_0^1\int_0^T\int_{U^\alpha_{2,\delta}}|\widetilde{V}^\alpha_{1,t}(\theta)- \widetilde{V}^\alpha_{1,t+\delta_0 \tilde{r}}(\theta)|I_{\{t+\delta_0 \tilde{r} \leq T\}}j(-\tilde{r})d\theta dtd\tilde{r}=0,
\end{eqnarray}
and
\begin{eqnarray}\label{eq116}
\lim_{\delta \rightarrow 0} \int_{|r|<1}\int_0^T\int_{U^\alpha_2} |\widetilde{V}^\alpha_{1,s}(\beta+\delta r)- \widetilde{V}^\alpha_{1,s}(\beta)|J(r)d\beta dsdr = 0.
\end{eqnarray}
By \eqref{eq115} and \eqref{eq116}, and letting $(\delta_0, \delta)  \rightarrow (0, 0)$ in \eqref{eq114}, we obtain
\begin{eqnarray}\label{eq117}
\lim_{(\delta_0,\delta) \rightarrow (0,0)}\big |\widetilde{I_2} - \int_{U_T^\alpha}|\widetilde{V}^\alpha_{1,s}(\beta) -\widetilde{V}^\alpha_{2,s}(\beta)| \widetilde{\nabla_{\Gamma_s}\cdot v}^\alpha(\beta)\widetilde{\chi}^{\alpha}(\beta)\widetilde{\eta}^{\alpha}_s(\beta)g^{\alpha}_s(\beta)d\beta ds \big| =0.
\end{eqnarray}

For the term $ I_3$, by \eqref{eq183} and the fact that $\widetilde{\nabla_{\Gamma_\cdot}\cdot v}^\alpha(\cdot)$, $\sqrt{g^\alpha_{\cdot}(\cdot)}$ and $\widetilde{\eta}^{\alpha}_{\cdot}(\cdot)$ belong to $C^1([0,T]\times U^\alpha)$, we get
\begin{eqnarray}\label{eq184}
       & & |I_3| \nonumber\\
&\leq& \int_{U_T^\alpha}\int_{U_T^\alpha} |\widetilde{w}^\alpha_{s}(\beta)\widetilde{\nabla_{\Gamma_s}\cdot v}^\alpha(\beta) - \widetilde{w}^\alpha_{t}(\theta)\widetilde{\nabla_{\Gamma_t}\cdot v}^\alpha(\theta)|j_{\delta_0}(t-s)J_{\delta}(\theta - \beta)\widetilde{\chi}^{\alpha}(\beta)\widetilde{\eta}^{\alpha}_s(\beta) \nonumber\\
      & & \cdot \sqrt{g^\alpha_t(\theta)}\sqrt{g^\alpha_s(\beta)}  d\theta dtd\beta ds \nonumber\\
&\leq& C_{T,U^\alpha_2}\int_0^T\int_{U^\alpha_2}\int_0^T\int_{U^\alpha_{2,\delta}}  |\widetilde{w}^\alpha_{s}(\beta)|\cdot |\widetilde{\nabla_{\Gamma_s}\cdot v}^\alpha(\beta) - \widetilde{\nabla_{\Gamma_t}\cdot v}^\alpha(\theta)| j_{\delta_0}(t-s)J_{\delta}(\theta - \beta)d\theta dtd\beta ds \nonumber\\
      & & + C_{T,U^\alpha_2}\int_0^T\int_{U^\alpha_2}\int_0^T\int_{U^\alpha_{2,\delta}} |\widetilde{\nabla_{\Gamma_t}\cdot v}^\alpha(\theta)|\cdot |\widetilde{w}^\alpha_{s}(\beta) -\widetilde{w}^\alpha_{t}(\theta)| j_{\delta_0}(t-s)J_{\delta}(\theta - \beta) d\theta dtd\beta ds \nonumber\\
&\leq& C_{T,U^\alpha_2,C_\kappa}\int_0^T\int_{U^\alpha_2}\int_0^T\int_{U^\alpha_{2,\delta}} (|\theta - \beta|+|t-s| + |t-s|^{\kappa})j_{\delta_0}(t-s)J_{\delta}(\theta - \beta) d\theta dtd\beta ds \nonumber\\
&\leq & C_{T,U^\alpha_2,C_\kappa}(\delta+\delta_0 +\delta_0^{\kappa} ).
\end{eqnarray}
Letting $(\delta_0, \delta)  \rightarrow (0, 0)$ in \eqref{eq184} yields
\begin{equation}\label{eq119}
\lim_{(\delta_0, \delta)  \rightarrow (0, 0)}|I_3|=0.
\end{equation}
Combining \eqref{eq117} with \eqref{eq119}, we obtain
\begin{equation*}\label{eq185}
\lim_{(\delta_0, \delta)  \rightarrow (0, 0)} (I_2+I_3) = \int_{U_T^\alpha}|\widetilde{V}^\alpha_{1,s}(\beta) -\widetilde{V}^\alpha_{2,s}(\beta)| \widetilde{\nabla_{\Gamma_s}\cdot v}^\alpha(\beta)\widetilde{\chi}^{\alpha}(\beta)\widetilde{\eta}^{\alpha}_s(\beta)g^{\alpha}_s(\beta)d\beta ds.
\end{equation*}

Since the proof of \eqref{eq121} is similar to that of \eqref{eq117}, we omit the details and thereby complete the proof of the Lemma.

\end{proof}

\vskip 0.2cm

Next we consider the limit of $I_5+I_6+I_7+I_8$ as $\delta_0 \rightarrow 0$ and $\delta \rightarrow 0$.


\begin{lemma}\label{lemma5.4}
It holds that
\begin{eqnarray}\label{eq122}
        & & \limsup_{\delta \rightarrow 0}\limsup_{\delta_0 \rightarrow 0} (I_5+I_6+I_7+I_8) \nonumber\\
&=&  \int_{U_T^\alpha} \mathrm{sign}(\widetilde{V}^\alpha_{1,s}(\beta) - \widetilde{V}^\alpha_{2,s}(\beta))\big( \frac{\partial\widetilde{a}^\alpha_{i}}{\partial \beta_i}(s,\beta,\widetilde{u}^\alpha_{2,s}(\beta)) - \frac{\partial\widetilde{a}^\alpha_{i}}{\partial\beta_i}(s,\beta,\widetilde{u}^\alpha_{1,s}(\beta) ) \big)\widetilde{\chi}^{\alpha}(\beta)\widetilde{\eta}^{\alpha}_s(\beta) g^\alpha_s(\beta) d\beta ds.\nonumber\\
\end{eqnarray}
\end{lemma}

\begin{proof}
We observe that
\begin{eqnarray*}\label{eq124}
    & &  I_7\nonumber\\
&=& \int_{U_T^\alpha}\int_{U_T^\alpha} \mathrm{sign}(\widetilde{V}^\alpha_{1,t}(\theta) -\widetilde{V}^\alpha_{2,s}(\beta))\big( \widetilde{a}^\alpha_{i}(t,\theta,\widetilde{u}^\alpha_{1,t}(\theta)) - \widetilde{a}^\alpha_{i}(s,\beta,\widetilde{u}^\alpha_{1,t}(\theta)) \big)j_{\delta_0}(t-s) \nonumber\\
   & & \cdot \frac{\partial J_{\delta}(\theta - \beta)}{\partial \theta_i} \widetilde{\chi}^{\alpha}(\beta)\widetilde{\eta}^{\alpha}_s(\beta)\sqrt{g^\alpha_t(\theta)}\sqrt{g^\alpha_s(\beta)}  d\theta dtd\beta ds  \nonumber\\
    & & + \int_{U_T^\alpha}\int_{U_T^\alpha}  \mathrm{sign}(\widetilde{V}^\alpha_{1,t}(\theta) -\widetilde{V}^\alpha_{2,s}(\beta))\big( \widetilde{a}^\alpha_{i}(s,\beta,\widetilde{u}^\alpha_{1,t}(\theta)) - \widetilde{a}^\alpha_{i}(s,\beta,\widetilde{V}^\alpha_{1,t}(\theta)+\widetilde{w}^\alpha_s(\beta))\big) j_{\delta_0}(t-s) \nonumber\\
   & &  \cdot \frac{\partial J_{\delta}(\theta - \beta)}{\partial \theta_i} \widetilde{\chi}^{\alpha}(\beta)\widetilde{\eta}^{\alpha}_s(\beta)\sqrt{g^\alpha_t(\theta)}\sqrt{g^\alpha_s(\beta)}d\theta dtd\beta ds \nonumber\\
&=& \int_{U_T^\alpha}\int_{U_T^\alpha} \mathrm{sign}(\widetilde{V}^\alpha_{1,t}(\theta) -\widetilde{V}^\alpha_{2,s}(\beta))\big( \widetilde{a}^\alpha_{i}(t,\theta,\widetilde{u}^\alpha_{1,t}(\theta)) - \widetilde{a}^\alpha_{i}(s,\beta,\widetilde{u}^\alpha_{1,t}(\theta)) \big)j_{\delta_0}(t-s) \nonumber\\
   & &  \cdot \frac{\partial J_{\delta}(\theta - \beta)}{\partial \theta_i} \widetilde{\chi}^{\alpha}(\beta)\widetilde{\eta}^{\alpha}_s(\beta)\sqrt{g^\alpha_t(\theta)}\sqrt{g^\alpha_s(\beta)}  d\theta dtd\beta ds  \nonumber\\
   & & + \int_{U_T^\alpha}\int_{U_T^\alpha}  \mathrm{sign}(\widetilde{V}^\alpha_{1,t}(\theta) -\widetilde{V}^\alpha_{2,s}(\beta))\big(\int_0^1 \widetilde{a}^\alpha_{iu}(s,\beta,\widetilde{u}^\alpha_{1,t}(\theta)+k(\widetilde{w}^\alpha_s(\beta) - \widetilde{w}^\alpha_t(\theta)) )  dk \big)\nonumber\\
   & & \cdot (\widetilde{w}^\alpha_t(\theta)-\widetilde{w}^\alpha_s(\beta))j_{\delta_0}(t-s)\frac{\partial J_{\delta}(\theta - \beta)}{\partial \theta_i} \widetilde{\chi}^{\alpha}(\beta)\widetilde{\eta}^{\alpha}_s(\beta)\sqrt{g^\alpha_t(\theta)}\sqrt{g^\alpha_s(\beta)}d\theta dtd\beta ds   \nonumber\\
&=& \int_{U_T^\alpha}\int_{U_T^\alpha} \mathrm{sign}(\widetilde{V}^\alpha_{1,t}(\theta) -\widetilde{V}^\alpha_{2,s}(\beta))\big( \widetilde{a}^\alpha_{i}(t,\theta,\widetilde{u}^\alpha_{1,t}(\theta)) - \widetilde{a}^\alpha_{i}(s,\beta,\widetilde{u}^\alpha_{1,t}(\theta)) \big)j_{\delta_0}(t-s) \nonumber\\
   & & \cdot \frac{\partial J_{\delta}(\theta - \beta)}{\partial \theta_i} \widetilde{\chi}^{\alpha}(\beta)\widetilde{\eta}^{\alpha}_s(\beta)\sqrt{g^\alpha_t(\theta)}\sqrt{g^\alpha_s(\beta)}  d\theta dtd\beta ds  \nonumber\\
   & & + \int_{U_T^\alpha}\int_{U_T^\alpha}  \mathrm{sign}(\widetilde{V}^\alpha_{1,t}(\theta) -\widetilde{V}^\alpha_{2,s}(\beta))\big(\int_0^1\int^{\widetilde{u}^\alpha_{1,t}(\theta)+k(\widetilde{w}^\alpha_s(\beta) - \widetilde{w}^\alpha_t(\theta))}_{\widetilde{V}^\alpha_{1,t}(\theta)+\widetilde{w}^\alpha_s(\beta)} \widetilde{a}^\alpha_{iuu}(s,\beta,u)du  dk \big)\nonumber\\
   & &  \cdot (\widetilde{w}^\alpha_t(\theta)-\widetilde{w}^\alpha_s(\beta))j_{\delta_0}(t-s)\frac{\partial J_{\delta}(\theta - \beta)}{\partial \theta_i} \widetilde{\chi}^{\alpha}(\beta)\widetilde{\eta}^{\alpha}_s(\beta)\sqrt{g^\alpha_t(\theta)}\sqrt{g^\alpha_s(\beta)}d\theta dtd\beta ds    \nonumber\\
   & & + \int_{U_T^\alpha}\int_{U_T^\alpha}  \mathrm{sign}(\widetilde{V}^\alpha_{1,t}(\theta) -\widetilde{V}^\alpha_{2,s}(\beta))\widetilde{a}^\alpha_{iu}(s,\beta, \widetilde{V}^\alpha_{1,t}(\theta)+\widetilde{w}^\alpha_s(\beta))
    (\widetilde{w}^\alpha_t(\theta)-\widetilde{w}^\alpha_s(\beta))j_{\delta_0}(t-s) \nonumber\\
   & &  \cdot\frac{\partial J_{\delta}(\theta - \beta)}{\partial \theta_i} \widetilde{\chi}^{\alpha}(\beta)\widetilde{\eta}^{\alpha}_s(\beta)\sqrt{g^\alpha_t(\theta)}\sqrt{g^\alpha_s(\beta)}d\theta dtd\beta ds  \nonumber\\
&=:& I_{7,1} + I_{7,2} + I_{7,3}.
\end{eqnarray*}
Similarly, for $I_8$, we have
\begin{eqnarray*}\label{eq132}
    & &  I_8 \nonumber\\
&=& \int_{U_T^\alpha}\int_{U_T^\alpha} \mathrm{sign}(\widetilde{V}^\alpha_{1,t}(\theta) -\widetilde{V}^\alpha_{2,s}(\beta))\big( \widetilde{a}^\alpha_{i}(s,\beta,\widetilde{u}^\alpha_{2,s}(\beta)) - \widetilde{a}^\alpha_{i}(t,\theta,\widetilde{u}^\alpha_{2,s}(\beta)) \big)j_{\delta_0}(t-s) \nonumber\\
   & & \cdot \frac{\partial J_{\delta}(\theta - \beta)}{\partial \theta_i} \widetilde{\chi}^{\alpha}(\beta)\widetilde{\eta}^{\alpha}_s(\beta)\sqrt{g^\alpha_t(\theta)}\sqrt{g^\alpha_s(\beta)}  d\theta dtd\beta ds  \nonumber\\
   & & + \int_{U_T^\alpha}\int_{U_T^\alpha}  \mathrm{sign}(\widetilde{V}^\alpha_{1,t}(\theta) -\widetilde{V}^\alpha_{2,s}(\beta))\big(\int_0^1\int^{\widetilde{u}^\alpha_{2,s}(\beta)+k(\widetilde{w}^\alpha_t(\theta) - \widetilde{w}^\alpha_s(\beta))}_{\widetilde{V}^\alpha_{2,s}(\beta)+\widetilde{w}^\alpha_t(\theta)} \widetilde{a}^\alpha_{iuu}(t,\theta,u)du  dk \big)\nonumber\\
   & &  \cdot (\widetilde{w}^\alpha_s(\beta)-\widetilde{w}^\alpha_t(\theta))j_{\delta_0}(t-s)\frac{\partial J_{\delta}(\theta - \beta)}{\partial \theta_i} \widetilde{\chi}^{\alpha}(\beta)\widetilde{\eta}^{\alpha}_s(\beta)\sqrt{g^\alpha_t(\theta)}\sqrt{g^\alpha_s(\beta)}d\theta dtd\beta ds    \nonumber\\
   & & + \int_{U_T^\alpha}\int_{U_T^\alpha}  \mathrm{sign}(\widetilde{V}^\alpha_{1,t}(\theta) -\widetilde{V}^\alpha_{2,s}(\beta))\widetilde{a}^\alpha_{iu}(t,\theta, \widetilde{V}^\alpha_{2,s}(\beta)+\widetilde{w}^\alpha_t(\theta))
    (\widetilde{w}^\alpha_s(\beta)-\widetilde{w}^\alpha_t(\theta))j_{\delta_0}(t-s) \nonumber\\
   & &  \cdot\frac{\partial J_{\delta}(\theta - \beta)}{\partial \theta_i} \widetilde{\chi}^{\alpha}(\beta)\widetilde{\eta}^{\alpha}_s(\beta)\sqrt{g^\alpha_t(\theta)}\sqrt{g^\alpha_s(\beta)}d\theta dtd\beta ds  \nonumber\\
&=:& I_{8,1} + I_{8,2} + I_{8,3}.
\end{eqnarray*}

Next, we divide the remaining proof of \eqref{eq122} into four steps.

{\bf Step 1:} In this step, we prove
\begin{eqnarray}\label{eq128}
    & & \limsup_{\delta \rightarrow 0}\limsup_{\delta_0 \rightarrow 0}\big|I_{7,1} -
\int_{U_T^\alpha}\int_{U_T^\alpha} \mathrm{sign}(\widetilde{V}^\alpha_{1,t}(\theta) -\widetilde{V}^\alpha_{2,s}(\beta))\frac{\partial \widetilde{a}^\alpha_{i}}{\partial \beta_j}(s, \beta,\widetilde{u}^\alpha_{1,t}(\theta))(\theta_j- \beta_j)j_{\delta_0}(t-s) \nonumber\\
& &  \cdot \frac{\partial J_{\delta}(\theta - \beta)}{\partial \theta_i}\widetilde{\chi}^{\alpha}(\beta) \widetilde{\eta}^{\alpha}_s(\beta)\sqrt{g^\alpha_t(\theta)}\sqrt{g^\alpha_s(\beta)}  d\theta dtd\beta ds  \big|
= 0,
\end{eqnarray}
and
\begin{eqnarray}\label{eq133}
    & & \limsup_{\delta \rightarrow 0}\limsup_{\delta_0 \rightarrow 0}\big|I_{8,1} -
\int_{U_T^\alpha}\int_{U_T^\alpha} \mathrm{sign}(\widetilde{V}^\alpha_{1,t}(\theta) -\widetilde{V}^\alpha_{2,s}(\beta))\frac{\partial \widetilde{a}^\alpha_{i}}{\partial \beta_j}(s, \beta,\widetilde{u}^\alpha_{2,s}(\beta))(\beta_j- \theta_j)j_{\delta_0}(t-s) \nonumber\\
& & \cdot \frac{\partial J_{\delta}(\theta - \beta)}{\partial \theta_i}\widetilde{\chi}^{\alpha}(\beta) \widetilde{\eta}^{\alpha}_s(\beta)\sqrt{g^\alpha_t(\theta)}\sqrt{g^\alpha_s(\beta)}  d\theta dtd\beta ds  \big|
= 0.
\end{eqnarray}

By Taylor's formula, we have
\begin{eqnarray}\label{eq125}
& & \widetilde{a}^\alpha_{i}(t,\theta,\widetilde{u}^\alpha_{1,t}(\theta)) - \widetilde{a}^\alpha_{i}(s,\beta,\widetilde{u}^\alpha_{1,t}(\theta)) \nonumber\\
&=& \frac{\partial \widetilde{a}^\alpha_{i}}{\partial s}(s, \beta,\widetilde{u}^\alpha_{1,t}(\theta))(t-s)
      + \frac{\partial \widetilde{a}^\alpha_{i}}{\partial \beta_j}(s, \beta,\widetilde{u}^\alpha_{1,t}(\theta))(\theta_j- \beta_j)\nonumber\\
  & & + \frac{1}{2}\frac{\partial^2 \widetilde{a}^\alpha_{i}}{\partial s^2}(s+l(t-s), \beta+l(\theta - \beta), \widetilde{u}^\alpha_{1,t}(\theta))(t-s)^2 \nonumber\\
  & & + \frac{\partial^2 \widetilde{a}^\alpha_{i}}{\partial s \partial \beta_j}(s+l(t-s), \beta+l(\theta - \beta), \widetilde{u}^\alpha_{1,t}(\theta))(t-s)(\theta_j- \beta_j)\nonumber\\
  & & + \frac{1}{2}\frac{\partial^2 \widetilde{a}^\alpha_{i}}{ \partial \beta_j\partial \beta_k}(s+l(t-s), \beta+l(\theta - \beta), \widetilde{u}^\alpha_{1,t}(\theta))(\theta_j- \beta_j)(\theta_k- \beta_k)\nonumber\\
&=:& \mathcal{A}_1+ \frac{\partial \widetilde{a}^\alpha_{i}}{\partial \beta_j}(s, \beta,\widetilde{u}^\alpha_{1,t}(\theta))(\theta_j- \beta_j) + \mathcal{A}_2 + \mathcal{A}_3 + \mathcal{A}_4,
\end{eqnarray}
where $l \in (0,1)$.
Thus, by \eqref{eq125}, \eqref{eq181} and the fact that $\widetilde{u}^\alpha_{1,\cdot}(\cdot) \in L^\infty([0,T]\times U^\alpha)$, we obtain
\begin{eqnarray}\label{eq126}
     & & \Big|I_{7,1} -
          \int_{U_T^\alpha}\int_{U_T^\alpha} \mathrm{sign}(\widetilde{V}^\alpha_{1,t}(\theta) -\widetilde{V}^\alpha_{2,s}(\beta))\frac{\partial \widetilde{a}^\alpha_{i}}{\partial \beta_j}(s, \beta,\widetilde{u}^\alpha_{1,t}(\theta))(\theta_j- \beta_j)j_{\delta_0}(t-s) \nonumber\\
           & &  \cdot \frac{\partial J_{\delta}(\theta - \beta)}{\partial \theta_i}\widetilde{\chi}^{\alpha}(\beta) \widetilde{\eta}^{\alpha}_s(\beta)\sqrt{g^\alpha_t(\theta)}\sqrt{g^\alpha_s(\beta)}  d\theta dtd\beta ds   \Big|  \nonumber\\
&=& \Big| \int_{U_T^\alpha}\int_{U_T^\alpha} \mathrm{sign}(\widetilde{V}^\alpha_{1,t}(\theta) -\widetilde{V}^\alpha_{2,s}(\beta))\big( \mathcal{A}_1 + \mathcal{A}_2 + \mathcal{A}_3 + \mathcal{A}_4\big)j_{\delta_0}(t-s)\frac{\partial J_{\delta}(\theta - \beta)}{\partial \theta_i} \nonumber\\
           & &  \cdot \widetilde{\chi}^{\alpha}(\beta) \widetilde{\eta}^{\alpha}_s(\beta)\sqrt{g^\alpha_t(\theta)}\sqrt{g^\alpha_s(\beta)}  d\theta dtd\beta ds \big|\nonumber\\
&\leq& C_{T,U^\alpha_{2,\delta}} \int_0^T\int_{U^\alpha_2}\int_0^T\int_{U^\alpha_{2,\delta}}(1+ |\widetilde{u}^\alpha_{1,t}(\theta)|^{q_0})\big( |t-s|+|t-s|^2 + |t-s|\cdot|\theta - \beta| + |\theta - \beta|^2 \big)\nonumber\\
       & &\cdot j_{\delta_0}(t-s)|\nabla_{\theta} J_{\delta}(\theta - \beta)|d\theta dtd\beta ds  \nonumber\\
&\leq& C_{T,U^\alpha_{2,\delta},q_0}(\frac{\delta_0}{\delta} + \frac{\delta_0^2}{\delta} + \delta_0 + \delta).
\end{eqnarray}
Letting $\delta_0 \rightarrow 0$, and then letting $\delta \rightarrow 0$ in \eqref{eq126}, we obtain \eqref{eq128}.

Similar to the above arguments, we can also derive \eqref{eq133}.
\vskip 0.2cm
{\bf Step 2:} In this step, we prove
\begin{eqnarray*}\label{eq131}
\limsup_{\delta \rightarrow 0}\limsup_{\delta_0 \rightarrow 0}|I_{7,2} | = 0,
\end{eqnarray*}
and
\begin{eqnarray*}\label{eq134}
\limsup_{\delta \rightarrow 0}\limsup_{\delta_0 \rightarrow 0}|I_{8,2} | = 0.
\end{eqnarray*}

By \eqref{eq181}, for any $(s,\beta) \in [0,T]\times U^{\alpha}_2$, we have
\begin{eqnarray*}\label{eq129}
      & & \big| \int_0^1\int^{\widetilde{u}^\alpha_{1,t}(\theta)+k(\widetilde{w}^\alpha_s(\beta) - \widetilde{w}^\alpha_t(\theta))}_{\widetilde{V}^\alpha_{1,t}(\theta)+\widetilde{w}^\alpha_s(\beta)} \widetilde{a}^\alpha_{iuu}(s,\beta,u)dudk \big| \nonumber\\
&\leq& C_{T,U^\alpha_{2,\delta}}\int_0^1 (1-k)(1 + |\widetilde{u}^\alpha_{1,t}(\theta)+k(\widetilde{w}^\alpha_s(\beta) - \widetilde{w}^\alpha_t(\theta))|^{q_0} + |\widetilde{V}^\alpha_{1,t}(\theta)+\widetilde{w}^\alpha_s(\beta)|^{q_0} )|\widetilde{w}^\alpha_t(\theta) - \widetilde{w}^\alpha_s(\beta)| dk \nonumber\\
&\leq& C_{T,U^\alpha_{2,\delta},q_0}(1+ |\widetilde{u}^\alpha_{1,t}(\theta)|^{q_0} + |\widetilde{w}^\alpha_s(\beta) - \widetilde{w}^\alpha_t(\theta)|^{q_0})|\widetilde{w}^\alpha_t(\theta) - \widetilde{w}^\alpha_s(\beta)|.
\end{eqnarray*}
Combining the above estimate with the fact that $\widetilde{u}^\alpha_{1,\cdot}(\cdot) \in L^\infty([0,T]\times U^\alpha)$ and \eqref{eq183}, we derive that
\begin{eqnarray*}\label{eq130}
        & & |I_{7,2} | \nonumber\\
&\leq& C_{T,U^\alpha_{2,\delta},q_0}\int_0^T\int_{U^\alpha_2}\int_0^T\int_{U^\alpha_{2,\delta}}(1+ |\widetilde{u}^\alpha_{1,t}(\theta)|^{q_0} + |\widetilde{w}^\alpha_s(\beta) - \widetilde{w}^\alpha_t(\theta)|^{q_0})|\widetilde{w}^\alpha_t(\theta) - \widetilde{w}^\alpha_s(\beta)|^2j_{\delta_0}(t-s) \nonumber\\
       & & \cdot|\nabla_{\theta} J_{\delta}(\theta - \beta)|d\theta dtd\beta ds  \nonumber\\
&\leq& C_{T,U^\alpha_{2,\delta},q_0,C_\kappa}\int_0^T\int_{U^\alpha_2}\int_0^T\int_{U^\alpha_{2,\delta}} (|\widetilde{w}^\alpha_t(\theta) - \widetilde{w}^\alpha_s(\theta)|^2 + |\widetilde{w}^\alpha_s(\theta) - \widetilde{w}^\alpha_s(\beta)|^2)j_{\delta_0}(t-s) \nonumber\\
       & & \cdot|\nabla_{\theta} J_{\delta}(\theta - \beta)|d\theta dtd\beta ds \nonumber\\
&\leq& C_{T,U^\alpha_{2,\delta},q_0,C_\kappa}(\frac{\delta_0^{2\kappa}}{\delta} + \delta),
\end{eqnarray*}
which implies that
\[\limsup_{\delta \rightarrow 0}\limsup_{\delta_0 \rightarrow 0}|I_{7,2} | = 0.\]
Similarly, on can show that
\[\limsup_{\delta \rightarrow 0}\limsup_{\delta_0 \rightarrow 0}|I_{8,2} | = 0.\]

{\bf Step 3:} In this step, we prove
\begin{equation*}\label{eq186}
\limsup_{\delta \rightarrow 0}\limsup_{\delta_0 \rightarrow 0}|I_5 + I_{8,3} + I_6+ I_{7,3}| = 0.
\end{equation*}

Note that
\begin{eqnarray}\label{eq136}
   & & I_5 + I_{8,3} + I_6+ I_{7,3}  \nonumber\\
&=& \int_{U_T^\alpha}\int_{U_T^\alpha} \mathrm{sign}(\widetilde{V}^\alpha_{1,t}(\theta) -\widetilde{V}^\alpha_{2,s}(\beta)) \widetilde{a}^\alpha_{iu}(t,\theta,\widetilde{V}^\alpha_{2,s}(\beta)+\widetilde{w}^\alpha_t(\theta))\frac{\partial((\widetilde{w}^\alpha_s(\beta) - \widetilde{w}^\alpha_t(\theta))J_{\delta}(\theta - \beta))}{\partial \theta_i} \nonumber\\
   & & \cdot j_{\delta_0}(t-s)\widetilde{\chi}^{\alpha}(\beta)\widetilde{\eta}^{\alpha}_s(\beta)\sqrt{g^\alpha_t(\theta)}\sqrt{g^\alpha_s(\beta)}  d\theta dtd\beta ds   \nonumber\\
   & & + \int_{U_T^\alpha}\int_{U_T^\alpha} \mathrm{sign}(\widetilde{V}^\alpha_{1,t}(\theta) -\widetilde{V}^\alpha_{2,s}(\beta)) \widetilde{a}^\alpha_{iu}(s,\beta,\widetilde{V}^\alpha_{1,t}(\theta)+\widetilde{w}^\alpha_s(\beta))\frac{\partial((\widetilde{w}^\alpha_s(\beta) - \widetilde{w}^\alpha_t(\theta))J_{\delta}(\theta - \beta))}{\partial \beta_i} \nonumber\\
   & & \cdot j_{\delta_0}(t-s)\widetilde{\chi}^{\alpha}(\beta)\widetilde{\eta}^{\alpha}_s(\beta)\sqrt{g^\alpha_t(\theta)}\sqrt{g^\alpha_s(\beta)}  d\theta dtd\beta ds   \nonumber\\
&=& \int_{U_T^\alpha}\int_{U_T^\alpha} \Big\{ \mathrm{sign}(\widetilde{V}^\alpha_{1,t}(\theta) -\widetilde{V}^\alpha_{2,s}(\beta)) \big(\widetilde{a}^\alpha_{iu}(t,\theta,\widetilde{V}^\alpha_{2,s}(\beta)+\widetilde{w}^\alpha_t(\theta)) - \widetilde{a}^\alpha_{iu}(s,\beta,\widetilde{V}^\alpha_{1,t}(\theta)+\widetilde{w}^\alpha_s(\beta))\big) \nonumber\\
   & & \cdot \frac{\partial((\widetilde{w}^\alpha_s(\beta) - \widetilde{w}^\alpha_t(\theta))J_{\delta}(\theta - \beta))}{\partial \theta_i} j_{\delta_0}(t-s)\widetilde{\chi}^{\alpha}(\beta)\widetilde{\eta}^{\alpha}_s(\beta)\sqrt{g^\alpha_t(\theta)}\sqrt{g^\alpha_s(\beta)} \Big\}d\theta dtd\beta ds  \nonumber\\
   & & +\int_{U_T^\alpha}\int_{U_T^\alpha} \Big\{ \mathrm{sign}(\widetilde{V}^\alpha_{1,t}(\theta) -\widetilde{V}^\alpha_{2,s}(\beta)) \widetilde{a}^\alpha_{iu}(s,\beta,\widetilde{V}^\alpha_{1,t}(\theta)+\widetilde{w}^\alpha_s(\beta))\big( \frac{\partial((\widetilde{w}^\alpha_s(\beta) - \widetilde{w}^\alpha_t(\theta))J_{\delta}(\theta - \beta))}{\partial \theta_i} \nonumber\\
   & & +\frac{\partial((\widetilde{w}^\alpha_s(\beta) - \widetilde{w}^\alpha_t(\theta))J_{\delta}(\theta - \beta))}{\partial \beta_i} \big)j_{\delta_0}(t-s)\widetilde{\chi}^{\alpha}(\beta)\widetilde{\eta}^{\alpha}_s(\beta)\sqrt{g^\alpha_t(\theta)}\sqrt{g^\alpha_s(\beta)} \Big\}d\theta dtd\beta ds  \nonumber\\
&=:& D_1+ D_2.
\end{eqnarray}
To estimate $D_1$, we further rewrite it as
\begin{eqnarray*}\label{eq137}
D_1&=&\int_{U_T^\alpha}\int_{U_T^\alpha} \Big\{ \mathrm{sign}(\widetilde{V}^\alpha_{1,t}(\theta) -\widetilde{V}^\alpha_{2,s}(\beta)) \big(\widetilde{a}^\alpha_{iu}(t,\theta,\widetilde{V}^\alpha_{2,s}(\beta)+\widetilde{w}^\alpha_t(\theta)) - \widetilde{a}^\alpha_{iu}(s,\beta,\widetilde{V}^\alpha_{2,s}(\beta)+\widetilde{w}^\alpha_t(\theta))\big) \nonumber\\
   & &\cdot \frac{\partial((\widetilde{w}^\alpha_s(\beta) - \widetilde{w}^\alpha_t(\theta))J_{\delta}(\theta - \beta))}{\partial \theta_i} j_{\delta_0}(t-s)\widetilde{\chi}^{\alpha}(\beta)\widetilde{\eta}^{\alpha}_s(\beta)\sqrt{g^\alpha_t(\theta)}\sqrt{g^\alpha_s(\beta)} \Big\} d\theta dtd\beta ds  \nonumber\\
   & & + \int_{U_T^\alpha}\int_{U_T^\alpha} \Big\{ \mathrm{sign}(\widetilde{V}^\alpha_{1,t}(\theta) -\widetilde{V}^\alpha_{2,s}(\beta)) \big(\widetilde{a}^\alpha_{iu}(s,\beta,\widetilde{V}^\alpha_{2,s}(\beta)+\widetilde{w}^\alpha_t(\theta)) - \widetilde{a}^\alpha_{iu}(s,\beta,\widetilde{V}^\alpha_{2,s}(\beta)+\widetilde{w}^\alpha_s(\beta))\big) \nonumber\\
   & &  \cdot \frac{\partial((\widetilde{w}^\alpha_s(\beta) - \widetilde{w}^\alpha_t(\theta))J_{\delta}(\theta - \beta))}{\partial \theta_i} j_{\delta_0}(t-s)\widetilde{\chi}^{\alpha}(\beta)\widetilde{\eta}^{\alpha}_s(\beta)\sqrt{g^\alpha_t(\theta)}\sqrt{g^\alpha_s(\beta)} \Big\} d\theta dtd\beta ds  \nonumber\\
   & & + \int_{U_T^\alpha}\int_{U_T^\alpha} \Big \{ \big\{  \mathrm{sign}(\widetilde{V}^\alpha_{1,t}(\theta) -\widetilde{V}^\alpha_{2,s}(\beta)) \big(\widetilde{a}^\alpha_{iu}(s,\beta,\widetilde{V}^\alpha_{2,s}(\beta)+\widetilde{w}^\alpha_s(\beta)) - \widetilde{a}^\alpha_{iu}(s,\beta,\widetilde{V}^\alpha_{1,t}(\theta)+\widetilde{w}^\alpha_s(\beta)\big) \nonumber\\
   & & - \mathrm{sign}(\widetilde{V}^\alpha_{1,s}(\beta) -\widetilde{V}^\alpha_{2,s}(\beta)) \big(\widetilde{a}^\alpha_{iu}(s,\beta,\widetilde{V}^\alpha_{2,s}(\beta)+\widetilde{w}^\alpha_s(\beta)) - \widetilde{a}^\alpha_{iu}(s,\beta,\widetilde{V}^\alpha_{1,s}(\beta)+\widetilde{w}^\alpha_s(\beta)\big) \big\} \nonumber\\
   & &  \cdot \frac{\partial((\widetilde{w}^\alpha_s(\beta) - \widetilde{w}^\alpha_t(\theta))J_{\delta}(\theta - \beta))}{\partial \theta_i} j_{\delta_0}(t-s)\widetilde{\chi}^{\alpha}(\beta)\widetilde{\eta}^{\alpha}_s(\beta)\sqrt{g^\alpha_t(\theta)}\sqrt{g^\alpha_s(\beta)} \Big\}d\theta dtd\beta ds  \nonumber\\
   & & + \int_{U_T^\alpha}\int_{U_T^\alpha} \Big\{ \mathrm{sign}(\widetilde{V}^\alpha_{1,s}(\beta) -\widetilde{V}^\alpha_{2,s}(\beta)) \big(\widetilde{a}^\alpha_{iu}(s,\beta,\widetilde{V}^\alpha_{2,s}(\beta)+\widetilde{w}^\alpha_s(\beta)) - \widetilde{a}^\alpha_{iu}(s,\beta,\widetilde{V}^\alpha_{1,s}(\beta)+\widetilde{w}^\alpha_s(\beta)\big)  \nonumber\\
   & & \cdot \frac{\partial((\widetilde{w}^\alpha_s(\beta) - \widetilde{w}^\alpha_t(\theta))J_{\delta}(\theta - \beta))}{\partial \theta_i} j_{\delta_0}(t-s)\widetilde{\chi}^{\alpha}(\beta)\widetilde{\eta}^{\alpha}_s(\beta)\sqrt{g^\alpha_t(\theta)}\sqrt{g^\alpha_s(\beta)} \Big\}d\theta dtd\beta ds  \nonumber\\
&=:& D_{1,1} + D_{1,2} + D_{1,3} + D_{1,4}.
\end{eqnarray*}

By \eqref{eq96}, for any $t, s \in [0,T]$, $\theta \in U^{\alpha}_{2,\delta}$, $\beta \in U^{\alpha}_2$ and $u \in \mathbb{R}$, we have
\begin{equation*}\label{eq187}
| \widetilde{a}^\alpha_{iu}(t,\theta,u) -  \widetilde{a}^\alpha_{iu}(s,\beta,u) | \leq C_{T,U^{\alpha}_{2,\delta}}(1+ |u|)(|t-s| + |\theta - \beta|).
\end{equation*}
Together with \eqref{eq183} and the fact that $\widetilde{V}^\alpha_{2,\cdot}(\cdot) \in L^\infty([0,T]\times U^\alpha)$, it follows that
\begin{eqnarray}\label{eq138}
    & & |D_{1,1}| \nonumber\\
&\leq& C_{T,U^{\alpha}_{2,\delta}} \int_0^T\int_{U^\alpha_2}\int_0^T\int_{U^\alpha_{2,\delta}} (1 + |\widetilde{V}^\alpha_{2,s}(\beta)+\widetilde{w}^\alpha_t(\theta)|)(|t-s|+|\theta - \beta|)(J_{\delta}(\theta - \beta)|\nabla_{\theta}\widetilde{w}^\alpha_t(\theta)|  \nonumber\\
       & & + |\nabla_{\theta}J_{\delta}(\theta - \beta)|\cdot|\widetilde{w}^\alpha_s(\beta) - \widetilde{w}^\alpha_t(\theta)|)j_{\delta_0}(t-s)d\theta dtd\beta ds  \nonumber\\
&\leq& C_{T,U^{\alpha}_{2,\delta},C_\kappa}\int_0^T\int_{U^\alpha_2}\int_0^T\int_{U^\alpha_{2,\delta}} (|t-s|+|\theta - \beta|)J_{\delta}(\theta - \beta)j_{\delta_0}(t-s)d\theta dtd\beta ds  \nonumber\\
      & & + C_{T,U^{\alpha}_{2,\delta},C_\kappa}\int_0^T\int_{U^\alpha_2}\int_0^T\int_{U^\alpha_{2,\delta}}(|t-s|+|\theta - \beta|)(|t-s|^\kappa+|\theta - \beta|)|\nabla_{\theta}J_{\delta}(\theta - \beta)|j_{\delta_0}(t-s)d\theta dtd\beta ds \nonumber\\
&\leq& C_{T,U^{\alpha}_{2,\delta},C_\kappa}(\frac{\delta_0^{\kappa+1}}{\delta} + \delta_0^{\kappa} + \delta_0 + \delta).
\end{eqnarray}
Hence, letting $\delta_0 \rightarrow 0$, and then letting $\delta \rightarrow 0$ in \eqref{eq138}, we obtain
\begin{eqnarray}\label{eq139}
\limsup_{\delta \rightarrow 0}\limsup_{\delta_0 \rightarrow 0}|D_{1,1}|=0.
\end{eqnarray}
Since for any $s \in [0,T]$, $\beta \in U^{\alpha}_{2,\delta}$ and $u_1, u_2 \in \mathbb{R}$,
\[ | \widetilde{a}^\alpha_{iu}(s,\beta,u_1) -  \widetilde{a}^\alpha_{iu}(s,\beta,u_2) | \leq C_{T,U^{\alpha}_{2,\delta}}(1+ |u_1|^{q_0} + |u_2|^{q_0})|u_1- u_2|, \]
an argument analogous to that used in the estimate of $D_{1,1}$ yields
\begin{eqnarray}\label{eq140}
\limsup_{\delta \rightarrow 0}\limsup_{\delta_0 \rightarrow 0}|D_{1,2}|=0.
\end{eqnarray}

By \eqref{eq181}, it is easy to verify that for any $s \in [0,T]$, $\beta \in U^{\alpha}_{2,\delta}$ and $u_1, u_2, u_3 \in \mathbb{R}$
\begin{eqnarray*}\label{eq143}
      & &  \big| \mathrm{sign}(u_1 -u_2)(\widetilde{a}^\alpha_{iu}(s,\beta,u_1) - \widetilde{a}^\alpha_{iu}(s,\beta,u_2)) - \mathrm{sign}(u_3 -u_2)(\widetilde{a}^\alpha_{iu}(s,\beta,u_3) - \widetilde{a}^\alpha_{iu}(s,\beta,u_2)) \big| \nonumber\\
&\leq& C_{T,U^{\alpha}_{2,\delta}}(1 + |u_1|^{q_0} + |u_3|^{q_0})|u_1 - u_3|.
\end{eqnarray*}
Then we have
\begin{eqnarray}\label{eq144}
       & &  |D_{1,3}|\nonumber\\
&\leq& C_{T,U^{\alpha}_{2,\delta}}\int_0^T\int_{U^\alpha_2}\int_0^T\int_{U^\alpha_{2,\delta}} (1 + |\widetilde{V}^\alpha_{1,t}(\theta)+\widetilde{w}^\alpha_s(\beta)|^{q_0} + |\widetilde{V}^\alpha_{1,s}(\beta)+\widetilde{w}^\alpha_s(\beta)|^{q_0})|\widetilde{V}^\alpha_{1,t}(\theta) - \widetilde{V}^\alpha_{1,s}(\beta)| \nonumber\\
      & & \cdot(J_{\delta}(\theta - \beta)|\nabla_{\theta}\widetilde{w}^\alpha_t(\theta)|+ |\nabla_{\theta}J_{\delta}(\theta - \beta)|\cdot|\widetilde{w}^\alpha_s(\beta) - \widetilde{w}^\alpha_t(\theta)|)j_{\delta_0}(t-s)d\theta dtd\beta ds \nonumber\\
&\leq& C_{T,U^{\alpha}_{2,\delta},C_\kappa, q_0}\int_0^T\int_{U^\alpha_2}\int_0^T\int_{U^\alpha_{2,\delta}} |\widetilde{V}^\alpha_{1,t}(\theta) - \widetilde{V}^\alpha_{1,s}(\beta)|J_{\delta}(\theta - \beta)j_{\delta_0}(t-s)d\theta dtd\beta ds\nonumber\\
      & & + C_{T,U^{\alpha}_{2,\delta},C_\kappa, q_0}\int_0^T\int_{U^\alpha_2}\int_0^T\int_{U^\alpha_{2,\delta}}  |\widetilde{V}^\alpha_{1,t}(\theta) - \widetilde{V}^\alpha_{1,s}(\beta)|\cdot|\nabla_{\theta}J_{\delta}(\theta - \beta)|(|t-s|^\kappa + |\theta - \beta|) \nonumber\\
      & & \cdot j_{\delta_0}(t-s)d\theta dtd\beta ds \nonumber\\
&\leq& C_{T,U^{\alpha}_{2,\delta},C_\kappa, q_0}\int_0^1\int_0^T\int_{U^\alpha_{2,\delta}}|\widetilde{V}^\alpha_{1,t}(\theta)- \widetilde{V}^\alpha_{1,t+\delta_0 \tilde{r}}(\theta)|I_{\{t+\delta_0 \tilde{r} \leq T\}}j(-\tilde{r})d\theta dtd\tilde{r} \nonumber\\
     & & + C_{T,U^{\alpha}_{2,\delta},C_\kappa, q_0}\int_{|r|<1}\int_0^T\int_{U^\alpha_2} |\widetilde{V}^\alpha_{1,s}(\beta+\delta r)- \widetilde{V}^\alpha_{1,s}(\beta)|J(r)d\beta dsdr  \nonumber\\
     & & + C_{T,U^{\alpha}_{2,\delta},C_\kappa, q_0}\frac{\delta_0^\kappa}{\delta} + C_{T,U^{\alpha}_{2,\delta},C_\kappa, q_0}\int_{|r|<1}\int_0^T\int_{U^\alpha_2} |\widetilde{V}^\alpha_{1,s}(\beta+\delta r)- \widetilde{V}^\alpha_{1,s}(\beta)|\nabla_{x}J(r)d\beta dsdr.  \nonumber\\
\end{eqnarray}
By \eqref{eq115} and \eqref{eq116}, we first let $\delta_0 \rightarrow 0$, and then take $\delta \rightarrow 0$ in \eqref{eq144} to obtain
\begin{equation}\label{eq147}
\limsup_{\delta\rightarrow 0}\limsup_{\delta_0 \rightarrow 0}|D_{1,3}| =0.
\end{equation}

Regarding $D_{1,4}$, using the divergence Theorem, \eqref{eq96} and the regularity of $\sqrt{g^\alpha}$, we have the following bound:
\begin{eqnarray*}\label{eq141}
    & &|D_{1,4}|\nonumber\\
&=&|\int_{U_T^\alpha}\int_{U_T^\alpha} \mathrm{sign}(\widetilde{V}^\alpha_{1,s}(\beta) -\widetilde{V}^\alpha_{2,s}(\beta)) \big(\widetilde{a}^\alpha_{iu}(s,\beta,\widetilde{V}^\alpha_{2,s}(\beta)+\widetilde{w}^\alpha_s(\beta)) - \widetilde{a}^\alpha_{iu}(s,\beta,\widetilde{V}^\alpha_{1,s}(\beta)+\widetilde{w}^\alpha_s(\beta)\big)  \nonumber\\
   & &\cdot (\widetilde{w}^\alpha_s(\beta) - \widetilde{w}^\alpha_t(\theta))J_{\delta}(\theta - \beta) j_{\delta_0}(t-s)\widetilde{\chi}^{\alpha}(\beta)\widetilde{\eta}^{\alpha}_s(\beta)\frac{\partial\sqrt{g^\alpha_t(\theta)}}{\partial \theta_i}\sqrt{g^\alpha_s(\beta)}  d\theta dtd\beta ds |\nonumber\\
&\leq& C_{T,U^\alpha_{2,\delta}} \int_0^T\int_{U^\alpha_2}\int_0^T\int_{U^\alpha_{2,\delta}}(1 + |\widetilde{u}^\alpha_{2,s}(\beta)| +|\widetilde{u}^\alpha_{1,s}(\beta)|)|\widetilde{w}^\alpha_s(\beta) - \widetilde{w}^\alpha_t(\theta)|J_{\delta}(\theta - \beta) j_{\delta_0}(t-s)d\theta dtd\beta ds, \nonumber\\
\end{eqnarray*}
and in the similar way as the treatment of $|I_{7,2}|$, we can prove that
\begin{eqnarray}\label{eq142}
\limsup_{\delta \rightarrow 0}\limsup_{\delta_0 \rightarrow 0}|D_{1,4}|=0.
\end{eqnarray}
Therefore, by \eqref{eq139}, \eqref{eq140}, \eqref{eq147} and \eqref{eq142}, we arrive at
\begin{eqnarray}\label{eq188}
\limsup_{\delta \rightarrow 0}\limsup_{\delta_0 \rightarrow 0}|D_{1}|=0.
\end{eqnarray}

Next we consider $D_2$. By \eqref{eq96}, we have
\begin{eqnarray*}\label{eq149}
       & &  |D_2|\nonumber\\
&\leq& C_{T,U^\alpha_{2,\delta}}\int_0^T\int_{U^\alpha_2}\int_0^T\int_{U^\alpha_{2,\delta}} (1 + |\widetilde{V}^\alpha_{1,t}(\theta) + \widetilde{w}^\alpha_s(\beta)|)|\nabla_{\beta}\widetilde{w}^\alpha_s(\beta) - \nabla_{\theta}\widetilde{w}^\alpha_s(\theta)| J_{\delta}(\theta - \beta) j_{\delta_0}(t-s)d\theta dtd\beta ds.\nonumber\\
\end{eqnarray*}
Then, using an argument similar to that for estimating $D_{1,3}$, we obtain
\begin{equation}\label{eq150}
\limsup_{\delta \rightarrow 0}\limsup_{\delta_0 \rightarrow 0}|D_{2}|=0.
\end{equation}
Combining \eqref{eq136}, \eqref{eq188} and \eqref{eq150}, it follows that
\[\limsup_{\delta \rightarrow 0}\limsup_{\delta_0 \rightarrow 0}|I_5 + I_{8,3} + I_6+ I_{7,3}| = 0. \]

{\bf Step 4:} In this step, we prove
\begin{eqnarray*}\label{eq189}
& & \limsup_{\delta \rightarrow 0}\limsup_{\delta_0 \rightarrow 0}| \mathcal{I} - \int_{U_T^\alpha} \mathrm{sign}(\widetilde{V}^\alpha_{1,s}(\beta) - \widetilde{V}^\alpha_{2,s}(\beta))\big( \frac{\partial\widetilde{a}^\alpha_{i}}{\partial \beta_i}(s,\beta,\widetilde{u}^\alpha_{2,s}(\beta)) - \frac{\partial\widetilde{a}^\alpha_{i}}{\partial\beta_i}(s,\beta,\widetilde{u}^\alpha_{1,s}(\beta) ) \big)\nonumber\\
     & & \cdot \widetilde{\chi}^{\alpha}(\beta)\widetilde{\eta}^{\alpha}_s(\beta) g^\alpha_s(\beta) d\beta ds  | = 0,\nonumber\\
\end{eqnarray*}
where
\begin{eqnarray*}\label{eq190}
\mathcal{I}&=& \int_{U_T^\alpha}\int_{U_T^\alpha} \mathrm{sign}(\widetilde{V}^\alpha_{1,t}(\theta) -\widetilde{V}^\alpha_{2,s}(\beta))\big( \frac{\partial \widetilde{a}^\alpha_{i}}{\partial \beta_j}(s, \beta,\widetilde{u}^\alpha_{1,t}(\theta)) - \frac{\partial \widetilde{a}^\alpha_{i}}{\partial \beta_j}(s, \beta,\widetilde{u}^\alpha_{2,s}(\beta)) \big)(\theta_j- \beta_j) \nonumber\\
     & &  \cdot j_{\delta_0}(t-s)\frac{\partial J_{\delta}(\theta - \beta)}{\partial \theta_i} \widetilde{\chi}^{\alpha}(\beta)\widetilde{\eta}^{\alpha}_s(\beta)\sqrt{g^\alpha_t(\theta)}\sqrt{g^\alpha_s(\beta)}  d\theta dtd\beta ds.
\end{eqnarray*}

We notice that
\begin{eqnarray*}\label{eq135}
\mathcal{I}
&=& \int_{U_T^\alpha}\int_{U_T^\alpha} \mathrm{sign}(\widetilde{V}^\alpha_{1,t}(\theta) -\widetilde{V}^\alpha_{2,s}(\beta))\big( \frac{\partial \widetilde{a}^\alpha_{i}}{\partial \beta_j}(s, \beta,\widetilde{u}^\alpha_{1,t}(\theta)) - \frac{\partial \widetilde{a}^\alpha_{i}}{\partial \beta_j}(s, \beta,\widetilde{u}^\alpha_{2,s}(\beta)) \big)j_{\delta_0}(t-s) \nonumber\\
     & & \cdot \frac{\partial ((\theta_j- \beta_j)J_{\delta}(\theta - \beta))}{\partial \theta_i} \widetilde{\chi}^{\alpha}(\beta)\widetilde{\eta}^{\alpha}_s(\beta)\sqrt{g^\alpha_t(\theta)}\sqrt{g^\alpha_s(\beta)}  d\theta dtd\beta ds   \nonumber\\
     & & - \int_{U_T^\alpha}\int_{U_T^\alpha} \mathrm{sign}(\widetilde{V}^\alpha_{1,t}(\theta) -\widetilde{V}^\alpha_{2,s}(\beta))\big( \frac{\partial \widetilde{a}^\alpha_{i}}{\partial \beta_i}(s, \beta,\widetilde{u}^\alpha_{1,t}(\theta)) - \frac{\partial \widetilde{a}^\alpha_{i}}{\partial \beta_i}(s, \beta,\widetilde{u}^\alpha_{2,s}(\beta)) \big)j_{\delta_0}(t-s) \nonumber\\
     & & \cdot J_{\delta}(\theta - \beta)\widetilde{\chi}^{\alpha}(\beta)\widetilde{\eta}^{\alpha}_s(\beta)\sqrt{g^\alpha_t(\theta)}\sqrt{g^\alpha_s(\beta)}  d\theta dtd\beta ds   \nonumber\\
&=:& \mathcal{I}_1 + \mathcal{I}_2.
\end{eqnarray*}

Following similar arguments to those used in the estimate of $D_1$, we conclude that
\begin{eqnarray*}\label{eq152}
\limsup_{\delta \rightarrow 0}\limsup_{\delta_0 \rightarrow 0}|\mathcal{I}_1|=0.
\end{eqnarray*}

Next we consider $\mathcal{I}_2$. Since 
\begin{eqnarray*}\label{eq153}
      & & |\mathcal{I}_2 -  \int_{U_T^\alpha} \mathrm{sign}(\widetilde{V}^\alpha_{1,s}(\beta) -\widetilde{V}^\alpha_{2,s}(\beta))\big( \frac{\partial \widetilde{a}^\alpha_{i}}{\partial \beta_i}(s, \beta,\widetilde{u}^\alpha_{2,s}(\beta)) - \frac{\partial \widetilde{a}^\alpha_{i}}{\partial \beta_i}(s, \beta,\widetilde{u}^\alpha_{1,s}(\beta)) \big)\widetilde{\chi}^{\alpha}(\beta)\widetilde{\eta}^{\alpha}_s(\beta)g^\alpha_s(\beta)d\beta ds   | \nonumber\\
&\leq& \int_{U_T^\alpha}\int_{U_T^\alpha} \big|\frac{\partial \widetilde{a}^\alpha_{i}}{\partial \beta_i}(s, \beta,\widetilde{u}^\alpha_{2,s}(\beta)) - \frac{\partial \widetilde{a}^\alpha_{i}}{\partial \beta_i}(s, \beta,\widetilde{u}^\alpha_{1,t}(\theta)) \big| j_{\delta_0}(t-s)J_{\delta}(\theta - \beta)\widetilde{\chi}^{\alpha}(\beta)\widetilde{\eta}^{\alpha}_s(\beta)\sqrt{g^\alpha_s(\beta)} \nonumber\\
      & & \cdot|\sqrt{g^\alpha_t(\theta)} - \sqrt{g^\alpha_s(\beta)}|d\theta dtd\beta ds \nonumber\\
      & & + \int_{U_T^\alpha}\int_{U_T^\alpha} \big|\frac{\partial \widetilde{a}^\alpha_{i}}{\partial \beta_i}(s, \beta,\widetilde{V}^\alpha_{1,t}(\theta) + \widetilde{w}^\alpha_{s}(\beta)) - \frac{\partial \widetilde{a}^\alpha_{i}}{\partial \beta_i}(s, \beta,\widetilde{u}^\alpha_{1,t}(\theta)) \big| j_{\delta_0}(t-s)J_{\delta}(\theta - \beta)\widetilde{\chi}^{\alpha}(\beta)
      \widetilde{\eta}^{\alpha}_s(\beta)\nonumber\\
      & & \cdot g^\alpha_s(\beta) d\theta dtd\beta ds \nonumber\\
      & & + \int_{U_T^\alpha}\int_{U_T^\alpha} \big|\mathrm{sign}(\widetilde{V}^\alpha_{1,t}(\theta) -\widetilde{V}^\alpha_{2,s}(\beta))\big( \frac{\partial \widetilde{a}^\alpha_{i}}{\partial \beta_i}(s, \beta,\widetilde{u}^\alpha_{2,s}(\beta) ) - \frac{\partial \widetilde{a}^\alpha_{i}}{\partial \beta_i}(s, \beta,\widetilde{V}^\alpha_{1,t}(\theta)+ \widetilde{w}^\alpha_{s}(\beta)) \big) \nonumber\\
      & & -  \mathrm{sign}(\widetilde{V}^\alpha_{1,s}(\beta) -\widetilde{V}^\alpha_{2,s}(\beta))\big( \frac{\partial \widetilde{a}^\alpha_{i}}{\partial \beta_i}(s, \beta,\widetilde{u}^\alpha_{2,s}(\beta)) - \frac{\partial \widetilde{a}^\alpha_{i}}{\partial \beta_i}(s, \beta,\widetilde{u}^\alpha_{1,s}(\beta)) \big)  \big| j_{\delta_0}(t-s)J_{\delta}(\theta - \beta)\widetilde{\chi}^{\alpha}(\beta)  \nonumber\\
      & & \cdot \widetilde{\eta}^{\alpha}_s(\beta)g^\alpha_s(\beta)  d\theta dtd\beta ds   \nonumber\\
       & & + \int_0^{\delta_0}\int_{U^\alpha}\big|\frac{\partial \widetilde{a}^\alpha_{i}}{\partial \beta_i}(s, \beta,\widetilde{u}^\alpha_{2,s}(\beta)) - \frac{\partial \widetilde{a}^\alpha_{i}}{\partial \beta_i}(s, \beta,\widetilde{u}^\alpha_{1,s}(\beta)) \big|\cdot (\int_0^Tj_{\delta_0}(t-s)dt +1) \widetilde{\chi}^{\alpha}(\beta)\widetilde{\eta}^{\alpha}_s(\beta)g^\alpha_s(\beta) d\beta ds,\nonumber\\
\end{eqnarray*}
using the similar arguments as in the proof of \eqref{eq117} in Lemma 7.3 and
in the proof of \eqref{eq142} and \eqref{eq144}, we infer that
\begin{eqnarray*}\label{eq154}
& & \limsup_{\delta \rightarrow 0}\limsup_{\delta_0 \rightarrow 0}|\mathcal{I}_2 -  \int_{U_T^\alpha} \mathrm{sign}(\widetilde{V}^\alpha_{1,s}(\beta) -\widetilde{V}^\alpha_{2,s}(\beta))\big( \frac{\partial \widetilde{a}^\alpha_{i}}{\partial \beta_i}(s, \beta,\widetilde{u}^\alpha_{2,s}(\beta)) - \frac{\partial \widetilde{a}^\alpha_{i}}{\partial \beta_i}(s, \beta,\widetilde{u}^\alpha_{1,s}(\beta)) \big) \nonumber\\
    & &\cdot \widetilde{\chi}^{\alpha}(\beta)\widetilde{\eta}^{\alpha}_s(\beta)g^\alpha_s(\beta)d\beta ds  | =0.
\end{eqnarray*}

Combining the above four steps yields \eqref{eq122}, which completes the proof.
\end{proof}

\vskip 0.2cm
Finally, we turn to the limit of $I_9$ and $I_{10}$ as $\delta_0 \rightarrow 0$ and $\delta\rightarrow 0$.
\begin{lemma}\label{lemma5.5}
\begin{eqnarray*}\label{eq155}
& &  \limsup_{\delta \rightarrow 0}\limsup_{\delta_0 \rightarrow 0} |I_9 - \int_{U_T^\alpha}\mathrm{sign}(\widetilde{V}^\alpha_{1,s}(\beta) - \widetilde{V}^\alpha_{2,s}(\beta) )\big( \widetilde{a}^\alpha_{i}(s,\beta,\widetilde{u}^\alpha_{1,s}(\beta) ) -\widetilde{a}^\alpha_{i}(s,\beta,\widetilde{u}^\alpha_{2,s}(\beta) ) \big) \nonumber\\
   & & \cdot \frac{\partial \widetilde{\chi}^{\alpha} }{\partial \beta}(\beta)\widetilde{\eta}^{\alpha}_s(\beta)g^\alpha_s(\beta) d\beta ds |=0,
\end{eqnarray*}
and
\begin{eqnarray*}\label{eq156}
& &  \limsup_{\delta \rightarrow 0}\limsup_{\delta_0 \rightarrow 0} |I_{10} - \int_{U_T^\alpha}\mathrm{sign}(\widetilde{V}^\alpha_{1,s}(\beta)- \widetilde{V}^\alpha_{2,s}(\beta) )\big( \widetilde{a}^\alpha_{i}(s,\beta,\widetilde{u}^\alpha_{1,s}(\beta))
- \widetilde{a}^\alpha_{i}(s,\beta,\widetilde{u}^\alpha_{2,s}(\beta)) \big) \nonumber\\
     & & \cdot \widetilde{\chi}^{\alpha}(\beta)\frac{\partial \widetilde{\eta}^{\alpha}_s }{\partial \beta_i}(\beta)g^\alpha_s(\beta) d\beta ds| = 0.
\end{eqnarray*}
\end{lemma}
\begin{proof}
The proof of this lemma is analogous to that of \eqref{eq147} in Lemma \ref{lemma5.4}. We omit the details.
\end{proof}

\vskip 0.2cm
Now, putting the above Lemmas \ref{lemma5.1}-\ref{lemma5.5} together, letting $\delta_0 \rightarrow 0$, and then letting $\delta \rightarrow 0$ in \eqref{eq105}, we obtain the claim \eqref{eq157}.

\vskip 0.3cm
	
	\noindent{\bf Acknowledgements}\\
	This work is partially supported by the National Key R\&D
	Program of China (No. 2022 YFA1006001), the National Natural Science Foundation of China (No. 12131019, 12171321, 12371151, 12426655, 12501182), and the Fundamental Research Funds for the Central Universities (No. JZ2025HGQA0098).

\end{document}